%% file: master.tex
\documentclass{article}
\usepackage{times}
\usepackage[hyperindex=true,pageanchor=true,hyperfigures=true,backref=false]{hyperref} 
\usepackage{amsmath}
\usepackage{amssymb}
\usepackage{mathtools}

\usepackage{url}
\usepackage{dsfont} 
\DeclareSymbolFontAlphabet{\Bbb}{AMSb}
\usepackage[T1]{fontenc}

\newlength{\myleftmargin}
\input{macros}

\input{locmacros}

\title{Conditioning of Banach Space Valued Gaussian Random Variables: An Approximation Approach Based on Martingales}

\author{Ingo Steinwart\footnote{I want to thank Aleksandar Arsenijevic,  Daniel Winkle, and Marvin Pf\"ortner  for helpful discussions. Funded by Deutsche Forschungsgemeinschaft (DFG, German Research Foundation) under Germany’s Excellence Strategy - EXC 2075 – 390740016.}\\
University of Stuttgart\\
Faculty 8: Mathematics and Physics\\
Institute for Stochastics and Applications\\
D-70569 Stuttgart Germany \\
\texttt{\small ingo.steinwart@mathematik.uni-stuttgart.de}
}

\begin{document}

\maketitle

\begin{abstract}
\input{abstract}
\end{abstract}

\textbf{Mathematical Subject Classification (2010).} Primary 60G15; Secondary 60B10, 60B11, 62F15, 62G05, 62G20, 28C20, 68T05.

\textbf{Key Words.} Conditional Gaussian distributions in Banach spaces, Gaussian processes for machine learning, Probabilistic numerics

\input{intro}

\input{prelims}

\input{results}

\input{filt-seq}
\input{examples}

\input{proofs}

\input{proofs-filt-seq}

\input{proofs-examples}

\numberwithin{equation}{section}

\newcounter{appendix}
\setcounter{appendix}{0}
\renewcommand\thesection{\Alph{appendix}}

\stepcounter{appendix}

\input{appendix-integration}

\stepcounter{appendix}
\input{appendix-covariances}

\stepcounter{appendix}
\input{appendix-cf+ft}

\stepcounter{appendix}
\input{appendix-rcp}

\stepcounter{appendix}
\input{appendix-gp-on-ct}

%
%

\bibliographystyle{plain}
\bibliography{../../literatur-db/steinwart-mine,../../literatur-db/steinwart-books,../../literatur-db/steinwart-article}

\end{document}

%% file: macros.tex
\usepackage{amsmath}
\usepackage{amssymb}
\usepackage{amsthm}
\usepackage{color}
\DeclareSymbolFontAlphabet{\Bbb}{AMSb}
\newcommand{\frk}[1]{\mathfrak{#1}}

\newtheorem{theorem}{Theorem}[section]
\newtheorem{lemma}[theorem]{Lemma}
\newtheorem{proposition}[theorem]{Proposition}
\newtheorem{corollary}[theorem]{Corollary}
\newtheorem{definition}[theorem]{Definition}

\newcommand{\atob}[2]{\emph{#1)} $\Rightarrow$ \emph{#2)}.} 
\newcommand{\aeqb}[2]{\emph{#1)} $\Leftrightarrow$ \emph{#2)}.} 
\newcommand{\ada}[1]{\emph{#1).}}

\newlength{\fixboxwidth}
\definecolor{darkgreen}{rgb}{0,0.6,0}

\newcommand{\scriptonlyraw}[1]{}

\newcommand{\ca}[1]{{\cal #1}}

\newcommand{\eul}{\mathrm{e}}
\newcommand{\imi}{\mathrm{i}}

\newcommand{\bidualmap}{\iota}
\newcommand{\Cdualmap}{\hat}

\newcommand{\transpose}{^{\mathsf{T}}}
\newcommand{\mpinv}{^\dagger}

\newcommand{\coordproj}[1]{\pi_{#1}}
\newcommand{\orthproj}{\Pi}
\newcommand{\crdfct}[1]{\frk C_{#1}}

\newcommand{\hochkl}[1]{^{(#1)}}
\newcommand{\hochkls}[1]{^{#1)}}

\newcommand{\annil}{{\perp\hspace*{-1.1ex}\perp}}

\newcommand{\mycdot}{\,\cdot\,}

\newcommand{\myqquad}{\qquad\qquad}

\newcommand{\N}{\mathbb{N}}
\newcommand{\Q}{\mathbb{Q}}
\newcommand{\R}{\mathbb{R}}
\newcommand{\Rd}{\mathbb{R}^d}

\newcommand{\Cfield}{\mathbb{C}}

\renewcommand{\P}{\mathrm{P}}
\newcommand{\Qm}{\mathrm{Q}}

\newcommand{\PX}{\P_X}

\newcommand{\PU}{\P_U}

\newcommand{\dirac}[1]{\d_{\{#1\}}}

\newcommand{\diracf}[1]{\d_{#1}}
\newcommand{\gauss}[2]{\ca N(#1,#2)}
\newcommand{\gaussm}[2]{\ca N\bigl(#1,#2\bigr)}

\newcommand{\gaussspace}[1]{\ca G(#1)}

\newcommand{\intd}{\, \mathrm{d}}
\newcommand{\ind}{\mathrm{d}}

\newcommand{\E}{\mathbb{E}}
\DeclareMathOperator{\var}{var}
\DeclareMathOperator{\cov}{cov}

\newcommand{\corXY}{R_{X,Y}}

\renewcommand{\a}{\alpha}
\renewcommand{\b}{\beta}

\renewcommand{\d}{\delta}

\newcommand{\e}{\varepsilon}

\newcommand{\lb}{\lambda}

\newcommand{\s}{\sigma}

\newcommand{\p}{\varphi}
\newcommand{\om}{\omega}
\newcommand{\Om}{\Omega}

\DeclareMathOperator{\spann}{span}

\DeclareMathOperator{\ran}{ran}

\DeclareMathOperator{\co}{co}

\DeclareMathOperator{\expoints}{ex}
\DeclareMathOperator{\aco}{aco}

\newcommand{\cfp}[1]{\varphi_{#1}}

\DeclareMathOperator{\id}{id}
\newcommand{\eins}{\boldsymbol{1}}

\newcommand{\unimodcont}[2]{\om(#1,#2)}

\newcommand{\snorm}[1]{\Vert #1 \Vert}
\newcommand{\mnorm}[1]{\bigl\Vert \, #1 \, \bigr\Vert}
\newcommand{\bnorm}[1]{\Bigl\Vert \, #1 \, \Bigr\Vert}

\newcommand{\inorm}[1]{\Vert #1 \Vert_\infty}
\newcommand{\nucnorm}[1]{\Vert #1 \Vert_{\mathrm{nuc}}}

\DeclareMathOperator{\metric}{d}
\newcommand{\topol}{\ca T}

\newcommand{\tauweaks}{w^*}
\newcommand{\tauweaksseq}{w^*\mathrm{\!-seq}}

\newcommand{\sborel}{\frk B}
\newcommand{\sborelnormx}[1]{\sborel(#1)}
\newcommand{\sborelnorm}[1]{\sborelnormx #1}

\newcommand{\sborelEx}[1]{\s(#1')}

\newcommand{\siC}{\ca C}

\newcommand{\sC}[1] {C(#1)}
\newcommand{\sCb}[1] {C_b(#1)}
\newcommand{\measpace}[1]{\ca M(#1)}

\newcommand{\denseblo}{\frk D}
\newcommand{\blonspace}[2]{\frk N(#1,#2)}

\newcommand{\sLx}[2]{{\ca L_{#1}(#2)}}

\newcommand{\Lx}[2]{{L_{#1}(#2)}}

\newcommand{\skprod}[2]{\langle #1, #2 \rangle}

\newcommand{\dualpairxy}[4]{\langle #1, #2 \rangle_{#3,#4}}
\newcommand{\dualpairxyb}[4]{\bigl\langle #1, #2 \bigr\rangle_{#3,#4}}
\newcommand{\dualpairxyB}[4]{\Bigl\langle #1, #2 \Bigr\rangle_{#3,#4}}
\newcommand{\dualpairxybb}[4]{\biggl\langle #1, #2 \biggr\rangle_{#3,#4}}

\newcommand{\dualpair}[3]{\dualpairxy{#1}{#2}{#3'}{#3}}
\newcommand{\dualpairb}[3]{\dualpairxyb{#1}{#2}{#3'}{#3}}
\newcommand{\dualpairB}[3]{\dualpairxyB{#1}{#2}{#3'}{#3}}
\newcommand{\dualpairbb}[3]{\dualpairxybb{#1}{#2}{#3'}{#3}}

\newcommand{\dualpairmc}[3]{\dualpairxy {#1}{#2}{\measpace #3}{\sC #3}}

\newcommand{\muup}{\mu_{X|Y}}
\newcommand{\muupx}{\mu_{X|X}}
\newcommand{\muupi}{\mu_{X_i|Y}}
\newcommand{\muupn}{\mu_{X|Y_n}}
\newcommand{\muupxn}{\mu_{X|X_n}}
\newcommand{\muupA}{\mu_{X|Y, \nfiltseq}}

\newcommand{\KX}{K_{X,X}}
\newcommand{\KXi}{K_{X_i, X_i}}
\newcommand{\Kup}{K_{X|Y}}
\newcommand{\Kupi}{K_{X_i|Y}}
\newcommand{\KupO}{K_{X_0|Y}}
\newcommand{\covup}{C_{X|Y}}
\newcommand{\covupn}{C_{X|Y_n}}
\newcommand{\covupx}{C_{X|X}}
\newcommand{\covupxn}{C_{X|X_n}}
\newcommand{\KXY}{K_{X,Y}}
\newcommand{\KYX}{K_{Y,X}}
\newcommand{\KXiY}{K_{X_i,Y}}
\newcommand{\KYXi}{K_{Y,X_i}}
\newcommand{\KXOY}{K_{X_0,Y}}
\newcommand{\KYXO}{K_{Y,X_0}}
\newcommand{\KY}{K_{Y,Y}}

\newcommand{\meanupn}{m_{X|Y_n=y_n}}
\newcommand{\meanup}{m_{X|Y=y}}
\newcommand{\meanupg}{m_{X|Y=g}}

\newcommand{\kfctup}{k_{X|Y}}


%% file: locmacros.tex
\newcommand{\npeval}[1]{I_{#1}}
\newcommand{\nfilts}[1]{A_{#1}}
\newcommand{\nfiltseq}{\frk A}

\newcommand{\Fconv}{F_{\mathrm{conv}, \nfiltseq}}

\newcommand{\obsop}{\mathsf S}

\newcommand{\sppath}[1]{#1_{\bullet}}

\newcommand{\sA}{\ca A}

\newcommand{\sB}{\ca B}

%% file: abstract.tex
We investigate the conditional distributions of two Banach space valued, jointly Gaussian random variables. In particular, 
we show that these conditional distributions are again Gaussian and that their means and covariances can be determined by a general
finite dimensional approximation scheme. 
Here, it turns out that the covariance operators occurring in this scheme converge with respect to the nuclear norm and that 
the conditional probabilities converge weakly. Furthermore, we discuss how our approximation scheme 
can be implemented in several classes of important Banach spaces such as (reproducing kernel) Hilbert spaces, spaces of continuous functions,
and other spaces consisting of functions.
As an example,
we then apply our general results to the case of continuous Gaussian processes  that are
conditioned to partial but infinite observations of their paths. Here we show that conditioning on   sufficiently rich, increasing sets of finitely many observations leads 
to  consistent approximations, that is, both the mean and covariance functions converge uniformly
and the conditional probabilities converge weakly.
Moreover, we discuss 
how these results improve our understanding of the popular Gaussian processes for machine learning.
From a technical perspective our results are based upon a Banach space valued martingale approach for regular conditional probabilities.

%% file: intro.tex
\section{Introduction}\label{sec:intro}

Gaussian processes for machine learning (GP4ML), see e.g.~\cite{RaWi06,Bishop06,KaHeSeSrXXa}, are a standard and widespread non-parametric
Bayesian technique that allows for both prediction and uncertainty quantification in a single approach.
In addition, their Bayesian update rule also plays an important role in other fields such as
information-based complexity, see e.g.~\cite{TrWaWo88,Ritter00,NoWo08,NoWo10}, probabilistic numerics and uncertainty quantification, see e.g.~\cite{Sullivan15,CoOaSuGi19a,HeOsKe22}, 
Bayesian optimization, see e.g.~\cite{Garnett23},
Bayesian methods for inverse problems, see e.g.~\cite{Stuart10a}, and spatial statistics, see e.g.~\cite{GeDiFuGu10,GeSc16a}. 
Finally, beginning with \cite{Neal96} they are also
interesting for neural networks. 

In the simplest setup, GP4MLs assume that a Gaussian process $X:= (X_t)_{t\in T}$ is given, where $T$ is some index set.
Moreover, it is assumed that both the mean function $m:T\to \R$ and the covariance function $k:T\times T\to \R$ defined by
\begin{align}\label{eq:gp4ml-mean}
m(t) &:= \E X_t \, ,\\ \label{eq:gp4ml-cov}
k(t_1,t_2) &:= \cov(X_{t_1}, X_{t_2}) 
\end{align}
are available. For some fixed, finite set $S:= \{s_1,\dots,s_n\}\subset T$ it is further assumed that
we can observe the Gaussian vector $Y:= (X_s)_{s\in S}$, that is, we can observe the Gaussian process $X$ at the points $s_1,\dots,s_n$.
Given an observation $y = (y_1,\dots,y_n)\in \R^n$ of $Y$, the ``conditional Gaussian process'' of $X$ given the observation $Y=y$ is
then also a Gaussian process, whose mean and covariance can be computed  by
\begin{align}\label{thm:grv-conditioning-y-fin-dim-CT-mean-intro}
\meanup(t) &= m(t) +    K_{t,S}   K_{S,S}\mpinv(y  - \E Y) \, ,\\ \label{thm:grv-conditioning-y-fin-dim-CT-cov-intro}
\kfctup(t_1,t_2)    &= k(t_1,t_2) -   K_{t_1,S}   K_{S,S}\mpinv K_{S,t_2} 
\end{align}
for all $t,t_1,t_2\in T$, where
  $K_{t_1,S}$, $K_{S,S}$, and $K_{S,t_2}$ denote covariance matrices with respect to the indicated indices and
$K_{S,S}\mpinv$ is the Moore-Penrose inverse of the $n\times n$-matrix $K_{S,S}$.
Moreover,
recall that the GP4ML literature often assumes noisy observations of the form 
$Y + \e$, where $\e = (\e_1,\dots,\e_n)$ has i.i.d.~entries with $\e_i \sim \gauss 0 {\s^2}$ and 
$\e$ is independent of $X$. In this case, $K_{S,S}\mpinv$ needs to be replaced by $(K_{S,S} + \s^2 I_n)^{-1}$, where $I_n$ denotes the $n$-dimensional 
identity matrix, see again  e.g.~\cite[Chapter 2.2]{RaWi06} and \cite[Chapter 6.4]{Bishop06}, as well as the discussion following \eqref{eq:obs-mod-plus-moise}.

In the GP4ML literature, the Gaussian process $X$ is usually
viewed as a mechanism to generate random functions, e.g.~random continuous functions on a compact metric space $T$.
In this case, the ``conditional Gaussian process'' of $X$ given the observations $Y=y$ is also viewed as a mechanism generating random continuous functions,
despite the fact that the derivations of \eqref{thm:grv-conditioning-y-fin-dim-CT-mean-intro} and \eqref{thm:grv-conditioning-y-fin-dim-CT-cov-intro},
see e.g.~\cite[Chapter 2.2]{RaWi06} and \cite[Chapter 6.4]{Bishop06},
only consider the conditional distributions of $X_{t_1},\dots, X_{t_m}$ for arbitrary, but  \emph{a-priori fixed}   $t_1,\dots,t_m\in T$,
instead of the conditional distribution of entire paths. For this reason, we used quotation remarks when referring to the conditional Gaussian process.
One way to address this issue, is to view $X$ as a $\sC T$-valued
Gaussian random variable, where $\sC T$ denotes the Banach space  of continuous $\R$-valued functions on $T$,
see Sections \ref{sec:prelims} and \ref{sec:examples} for these and the following notions.
In this case, the above $Y$ is an $\R^n$-valued Gaussian random variable and $X$ and $Y$ are jointly Gaussian.
Given their joint distribution $\P_{(X,Y)}$ on $\sC T\times \R^n$, one then seeks a version $\P_{X|Y}(\mycdot|\mycdot)$ of the regular
conditional probability of $X$ given $Y$, since $\P_{X|Y}(\mycdot|y)$ is the conditional distribution  of $X$ given the observation $Y=y$.
Note that by definition, each  $\P_{X|Y}(\mycdot|y)$ is a distribution on $\sC T$ and therefore we can also
ensure that the ``conditional Gaussian process'' again produces continuous functions, provided that 
\eqref{thm:grv-conditioning-y-fin-dim-CT-mean-intro} and \eqref{thm:grv-conditioning-y-fin-dim-CT-cov-intro} match the mean and covariance functions of $\P_{X|Y}(\mycdot|y)$.

The references mentioned at the beginning indicate that for the conditioning problem,
$\sC T$ is by no means the only Banach space one is interested in.
Therefore, let us now consider the general case. To this end, let 
$E$ and $F$ be Banach spaces, which due to technical simplicity are assumed to be separable.
Moreover, let $X$ be an $E$-valued random variable and $Y$ be an $F$-valued random variable such that $X$ and $Y$ are jointly Gaussian.
In the abstract setting one then seeks a version $\P_{X|Y}(\mycdot|\mycdot)$ of the regular
conditional probability of $X$ given $Y$, since  $\P_{X|Y}(\mycdot|y)$ is the distribution on $E$ that equals
the conditional distribution  of $X$ given the observation $Y=y$. 
Compared to the $\sC T$-case described earlier, this generalization
opens several additional possibilities. First, of course, we can consider   Banach spaces of functions $E$ other than $\sC T$, e.g.~spaces
of differentiable functions. Second, we no longer need to assume that we can observe the random functions on some finite number of observational points.
Instead, we could assume, e.g., that we observe derivatives of the random functions at some observational points \cite{PfStHeWeXXa}.
More generally, $Y$ can, for example, be composed by any other finite set of bounded linear functionals as considered in e.g.~\cite{TrGi24a}
and in the context of (finite element) PDE solvers in e.g.~\cite{CoOaSuGi17a,ChHoOwSt21a,PfStHeWeXXa,PoKeRoMe24a}. Moreover, even 
  infinite sets of
  functionals can be meaningful:  To be concrete, in the $\sC T$-example above, the abstract setting makes it possible to
consider arbitrary subsets $S\subset T$ for the observational points. This generalization is  is important for e.g.~PDE-related applications, 
where $S$ may be used to describe (partial) boundary conditions
see e.g.~the recent \cite{HeOsKe22,PfStHeWeXXa} and the references mentioned therein.
Abstracting from these examples leads to the  ``observational model''
\begin{align}\label{eq:obs-model}
Y = \obsop \circ X\, ,
\end{align}
where $\obsop:E\to F$ is a
known bounded linear operator. 
Here we  note that  in the $\sC T$-example, $\obsop$ is the restriction operator $\sC T \to \sC S$ given by $f\mapsto f_{|S}$. However, 
\eqref{eq:obs-model} goes far beyond this case, and in fact \eqref{eq:obs-model} is a somewhat classical instance of the general conditional problem as we will 
see below when reviewing the existing literature.


Although working with infinite dimensional $Y$ has clear conceptual advantages, there is also a  drawback: 
Indeed,  it is often infeasible to fully observe or process infinite dimensional $Y$, or to  
compute $\P_{X|Y}$ for such $Y$.
In these cases, a natural ansatz is to use a sequence $(Y_n)$ of
observable, finite dimensional Gaussian random variables, for which the condition step is feasible. For example, in the $\sC T$-case with
infinite $S\subset T$, one could choose an increasing sequence $(S_n)$ of finite subsets $S_n\subset S$, whose union is dense in $S$, and
set 
\begin{align}\label{eq:Yn}
Y_n := (X_s)_{s\in S_n}\, .
\end{align}
 The hope of this approach is,
of course, that the resulting $\P_{X|Y_n}$'s provide a convergent approximation of $\P_{X|Y}$, e.g.~in terms
of their means and covariances, or as distributions. Now, 
in the case $S=T$ of full observations,  there are indeed results that guarantee convergence of the means and variances as we will discuss below,
and since in this case  the conditional distributions $\P_{X|Y}$ are point masses at the observations, these results can be interpreted as contractions of the
Bayesian method given by \eqref{eq:gp4ml-mean} and \eqref{eq:gp4ml-cov}.
Because of the applications mentioned earlier, we are, however,   mostly interested in case $S\neq T$ of partial observations, in which 
$\P_{X|Y}$ does not simply concentrate on single observations, and where it is usually impossible to 
compute $\P_{X|Y}$ in closed form. 
This case covers, for example,
 the PDE applications mentioned above, where \eqref{eq:Yn} provides a simple approximation approach, see e.g.~\cite{CoOaSuGi17a,RaPeKa17a,PfStHeWeXXa} for some concrete examples.

Let us now discuss what is already known for 
the fully infinite dimensional conditioning and its possible approximations by finite dimensional observations:
To begin with, in the case of finite dimensional observations, that is
 $F= \R^n$,   \cite[Corollary 3.10.3]{Bogachev98} shows that there is a version
 $\P_{X|Y}(\mycdot|\mycdot)$, such that $\P_{X|Y}(\mycdot|y)$  is Gaussian for all $y$. However, no
formulas for the mean and covariance of $\P_{X|Y}(\mycdot|y)$ are provided.
Moreover, the observational model \eqref{eq:obs-model}  for centered $X$ and with arbitrary $E$ and $F$ has been investigated in
\cite{TaVa07a}, where the authors show that the regular conditional probabilities are Gaussian measures, see their Theorem 3.11.
Furthermore, in the case $F=\R^n$, their Corollary 3.13 presents explicit formulas for the conditional mean and covariance
with the help of a so-called representing sequence of the covariance operator $\cov (Y)$.
In addition, a closer look at their proofs reveals, that analogous formulas hold true for general, separable $F$.
The latter formulas have been made explicit and generalized to non-centered $X$ in \cite{TrGi24a}. Moreover, both papers
establish the linearity of the conditional mean mapping, and \cite{TrGi24a} also describes formulas for sequential updates, and 
in the case of $E= \sC T$ and $F=\R^n$, the update rules \eqref{thm:grv-conditioning-y-fin-dim-CT-mean-intro} and \eqref{thm:grv-conditioning-y-fin-dim-CT-cov-intro}
are generalized to arbitrary, bounded linear observational operators $\obsop:\sC T\to \R^n$.
Unfortunately, however, the mean and covariance formulas for infinite dimensional $F$ become less handy, as they require
to have a representing sequence for $\cov (Y)$. Although such
sequences as well as abstract constructions for them exist, see e.g.~\cite[Lemma 3.5]{TaVa07a} or \cite[Lemma 8.2.3]{Stroock11a},
they are rarely at hand in situations in which $F$ is not a Hilbert space.
As a consequence, it seems fair to say that the derived formulas have some limited practical value.

The case of $\sC T$-valued $X$ is also considered in \cite{LaGatta13a}. There, $T=[0,t^*]$ is
an interval  and $Y$ is given by a partial observation of $X$ up to some earlier time $s\in [0,t^*]$,
that is $Y:= X_{|[0,s]}$. Under an additional assumption,
the author then proves that the conditional distribution $\P_{X|Y}(\mycdot|y)$
is a Gaussian measure on $\sC T$ and that  the map  $y\mapsto \P_{X|Y}(\mycdot|y)$ is   continuous 
with respect to the weak convergence of probability measures.
Furthermore, it is shown that his additional assumption is also necessary for the existence of such continuous  conditional probabilities.
In fact it is worth to mention that both results are actually first shown for the general observational model \eqref{eq:obs-model} and then applied to the
$\sC T$-case, but the motivation in \cite{LaGatta13a} clearly focuses on the $\sC T$-case.

A different dependence of the conditional distributions is considered in  \cite{AlYu17a}. There, the general case of joint Gaussian random variables
$X$ and $Y$ is considered and, in addition, it is assumed that their joint distribution $\P_{(X,Y), \a}$
depends measurably on a parameter $\a$. The main result of
 \cite{AlYu17a} then shows that there exist  (regular) conditional probabilities $\P_{(X|Y), \a}$
 such that $(y,\a)\mapsto \P_{(X|Y),\a}(\mycdot|y)$ is measurable if the set of probability measures is equipped with the Borel $\s$-algebra
 generated by the topology of weak convergence of probability measures.

If both $E$ and $F$ are Hilbert spaces, the conditioning problem has also been investigated by different techniques.
For example, in \cite{OwSc18a} formulas for the mean and covariance of $\P_{X|Y}(\mycdot|y)$ are
established in the full infinite dimensional setting in terms of both \emph{a)} so-called shorted operators, 
and   \emph{b)} suitable sequences of oblique, i.e.~non-orthogonal, projections.
Moreover, \cite{KlSpSu21a} shows, besides many other results, that some of the formulas derived in
\cite{OwSc18a} can be significantly simplified. Independently, \cite{Mandelbaum84a,GoMa08a} established simple formulas for the the mean and covariance,
which essentially resemble \eqref{thm:grv-conditioning-y-fin-dim-CT-mean-intro} and \eqref{thm:grv-conditioning-y-fin-dim-CT-cov-intro} in an abstract sense, under the additional assumption that $\cov (Y)$ is injective.

The approximation \eqref{eq:Yn} in the $\sC T$-case with $S= T$ has also been considered in the literature.
For example, convergence rates for the conditional mean functions, that is, for $\meanupn \to \meanup$,
have been derived in \cite{VaZa11a} for specific Gaussian processes
and the results of \cite{FiSt20a} can also be used to obtain such rates for random observational points and additional, vanishing observational noise.
We refer to \cite[Section 5.1]{KaHeSeSrXXa} for a detailed discussion, where we note that \cite{FiSt20a} provides the better results compared to
the source \cite{StHuSc09b} mentioned in \cite{KaHeSeSrXXa}, since \emph{a)} a clipping step for the mean is avoided, and \emph{b)} also stronger norms 
for the convergence are
considered.
In addition, convergence rates for Sobolev type covariance functions can be derived using the fill-distance and results from kernel interpolation
\cite{Wendland05}. For  a  detailed discussion, we    refer  to \cite[Section 5.2]{KaHeSeSrXXa} and the references mentioned therein.
Finally, \cite{KoPf21a} recently established a uniform convergence result for the conditional mean and covariance functions under rather general
assumptions on the underlying space $T$ and the prior kernel $k$. In addition, this paper provides a small list of earlier, less general convergence results
in the $\sC T$-case with $S= T$. 


In this work, we address the general infinite dimensional conditioning problem and a suitable approximation scheme simultaneously.
Namely, we show that given a sufficiently rich sequence $(y_n')\subset F'$ in the dual of $F$, see Definition \ref{def:filt-seq} and the examples in 
Section \ref{seq:filt-seq}, we have the weak convergence
\begin{align*}
\P_{X|Y_n}(\mycdot| y_n) & \to \P_{X|Y}(\mycdot|y)
\end{align*}
for almost all observations $y$ of $Y$, where $Y_n$ is the  $\R^n$-valued Gaussian random variable 
$Y_n := (\dualpair {y_1'}YF,\dots, \dualpair {y_n'}YF$ and 
$y_n := (\dualpair {y_1'}yF,\dots, \dualpair {y_n'}yF$ is an observation of $Y_n$.
In addition, we show that the mean vectors  and covariance operators also converge, see Theorem \ref{thm:grv-conditioning-y-inf-dim} for these three statements.
Moreover, we prove that the means and covariances of $\P_{X|Y_n}(\mycdot| y_n)$ can be computed analogously to the GP4ML case
\eqref{thm:grv-conditioning-y-fin-dim-CT-mean-intro} and \eqref{thm:grv-conditioning-y-fin-dim-CT-cov-intro}, see
 Theorem 
 \ref{thm:grv-conditioning-y-inf-dim}. Furthermore,
 we  show that the set of $y\in F$ for which the means converge is a measurable, affine linear subspace of $F$ with full measure. This in turn
   can be used to construct a  ``benign'' version of $\P_{X|Y}(\mycdot|\mycdot)$, see Theorem \ref{thm:canonical-rcp}.
 Finally, we apply our  theory to the $\sC T$-valued case, where we establish uniformly convergent 
 limit versions of the update rules
\eqref{thm:grv-conditioning-y-fin-dim-CT-mean-intro} and \eqref{thm:grv-conditioning-y-fin-dim-CT-cov-intro} for general $S\subset T$. These results 
generalize the findings of 
\cite{KoPf21a} in the sense that neither $S=T$ is   required nor additional assumptions on the observational model are needed.

 As the title of this paper indicates, the major work horse in our proofs is a Banach space-valued martingale approach,
 which is coupled with the relationship between regular conditional probabilities and conditional expectations
 in the Banach space case, see Theorem \ref{thm:limit-conditioning}. In this respect we note that martingale arguments have already been
 used by e.g.~\cite{Mandelbaum84a,OwSc18a} in the Hilbert space setting, but as far as we know, not in the general Banach space setting.
 In addition, the resulting martingale approximations in  \cite{Mandelbaum84a,OwSc18a} are restricted to rather specific filtrations, whereas our
 results only require a natural richness assumption on the filtration.

The rest of this paper is organized as follows: In Section \ref{sec:prelims} we recall  all  notions and notations from probability theory and
functional analysis that are necessary to formulate our main results. These main results are then presented in Section \ref{sec:main-results}. 
Moreover, in Section \ref{seq:filt-seq} we investigate the   ``sufficiently rich'' sequences $(y_n')\subset F'$ mentioned above
in more detail. In particular, we establish a general existence result and discuss simple constructions for some concrete types of spaces including 
(reproducing kernel) Hilbert spaces and spaces of continuous or smooth functions. 
The 
 $\sC T$-valued case is subsequently investigated in more detail  in Section \ref{sec:examples}.
Sections \ref{sec:proofs}
to \ref{sec:example-proofs} contain the proofs of the results of Sections \ref{sec:main-results} to \ref{sec:examples}.
Finally, we have attached five detailed supplements on
\emph{A)} measurability in Banach spaces and Banach space valued integration,
\emph{B)} covariance operators,
\emph{C)} weak convergence, characteristic functions, and Gaussian random variables,
\emph{D)} regular conditional probabilities and conditional expectations in the Banach space case,
and
\emph{E)} continuous Gaussian processes.
Most material in these supplements has been simply collected and unified from the existing literature, but a couple of
mostly simple results are also added since we could not find them in the literature.

%% file: prelims.tex
\section{Preliminaries}\label{sec:prelims}

Throughout this paper, $E$ and $F$ denote Banach spaces. Moreover, we write $B_E$ for the closed unit ball 
of $E$, and $E'$ for the dual of $E$. 
Furthermore,  $\dualpair \mycdot\mycdot E:E'\times E\to \R$  denotes the standard dual pairing defined by
\begin{align*}
\dualpair {x'}x E := x'(x) \, , \myqquad x'\in E',x\in E.
\end{align*}
Moreover, recall that the $\tauweaks$-topology on $E'$ is the smallest topology for which all $\dualpair \mycdot xE:E'\to \R$
are continuous. 
%
%

On $\Rd$ we always consider the usual Borel $\s$-algebra $\sborel^d$ and in the case $d=1$ we simply write 
$\sborel$. Furthermore recall that on general Banach spaces $E$, there are
 at least two natural
$\s$-algebras: The first one is the usual Borel $\s$-algebra 
$\sborelnorm E$ that is generated by the norm topology on $E$ and the second one is 
$\sborelEx E$, that is, the smallest $\s$-algebra on $E$, for which all bounded, linear $x':E\to \R$ are measurable.
In general, we have $\sborelEx E \subset \sborelnorm E$, but for separable $E$ it is well-known that  $\sborelEx E =  \sborelnorm E$, see 
Theorem \ref{thm:pettis-var}. 

%
%
%
%

If  $(\Om, \sA)$ is a measurable space and  $E$ is a Banach space, then
 an $X:\Om\to E$ is weakly measurable if it is $(\sA, \sborelEx E)$-measurable, or equivalently, if $\dualpair {x'}XE:\Om\to   \R$
is measurable for all $x'\in E'$.
%
Moreover, $X$ is   strongly measurable if it is $(\sA, \sborelnorm E)$-measurable and $X(\Om)$ is separable.
If $E$ is separable, then 
 weak and strong measurability are equivalent, see  Theorem \ref{thm:pettis-var}, and for this reason 
 we sometimes simply speak of measurable $X$.
 Finally, if $(\Om, \sA,\P)$ is a probability
 space, then a strongly measurable $X:\Om\to E$ is called a  random variable, 
 and in this case its image measure $\P_X$ on $\sborelnorm E$ is called the distribution of $X$.
 
%

For a  fixed  probability space $(\Om,\sA,\P)$,       a Banach space $E$, and $p\in [1,\infty)$, we write
\begin{align*}
\sLx p {\P,E} := \biggl\{X:\Om \to E \,\Bigl |\, X \mbox{ is strongly measurable and } \int_\Om \snorm {X(\om)}_E^p \intd \P(\om) < \infty \biggr\} 
\end{align*}
for the space of $p$-times Bochner integrable functions. 
Each $X\in \sLx 1 {\P,E}$ is called Bochner integrable and for these $X$ one can define the expectation by the $E$-valued Bochner integral
\begin{align*}
\E X := \int_\Om X \intd \P \,  ,
\end{align*}
see also
Supplement \ref{app:bs-integration}, which collects some  information on Bochner integrals.
Moreover, for $X\in \sLx 2 {\P,E}$ and $Y \in \sLx 2 {\P,E}$, where $F$ is another Banach space, 
the cross covariance of $X$ and $Y$ is the bounded linear
operator $\cov(X,Y):F'\to E$ defined by the Bochner integrals
\begin{align*}
\cov(X,Y)(y') = \int_\Om \dualpair {y'} {Y-\E Y} F \, (X-\E X) \intd \P \, , \myqquad y'\in F'.
\end{align*}
Note that $\cov(X,Y)$ is even nuclear and thus also compact. We refer to
e.g.~\cite[Theorem 2.5 in Chapter III.2]{VaTaCh87}, or Lemma \ref{lem:cross-cov-as-integral} for an alternative
argument, and to \eqref{def:nuc-op} and \eqref{eq:nuc-norm} for the definition of nuclear operators and their norms.
In particular, the covariance of $X$, which is defined by 
\begin{align*}
\cov(X) := \cov(X,X)\, ,
\end{align*}
is a nuclear operator $\cov(X):E'\to E$. For more information on (cross) covariances we refer to 
Supplement \ref{app:covariances}.
Finally, in the case $E= \R^m$ and $F= \R^n$ the cross covariance operator $\cov(X,Y):(\R^n)'\to \R^m$
can, of course, be described by the $m\times n$-covariance matrix 
 \begin{align}\label{eq:cross-cov-matrix}
 \KXY := \bigl( \cov(X_i,Y_j)  \bigr)_{i,j} \, ,
 \end{align}
where $X_i$ and $Y_j$ denote the $i$th, respectively $j$th, component   of $X$ and $Y$.

%
%
%

Recall that a probability measure $\P$ on $\R$ is a Gaussian measure, if 
there are $\mu\in \R$ and $\s\in [0,\infty)$ such that 
the characteristic function of $\P$
is given by 
\begin{align*}
\cfp \P(t) =\eul^{\imi  t\mu} \cdot \exp\biggl( -\frac{ \s^2 t^2}{2} \biggr) 
\, , \qquad \qquad t\in \R.
\end{align*}
In the case $\s>0$, we thus have $\P = \gauss \mu{\s^2}$, i.e.~$\P$ is a normal distribution with expectation $\mu$
and variance $\s^2$.
If $\s = 0$, the  probability measure $\P$ has all its mass on $\mu$ and we write $\gauss \mu0 := \P$.
Finally, if $(\Om,\sA,\P)$ is a probability space and $X:\Om\to \R$ is measurable, then $X$ is said to be an $\R$-valued
Gaussian random variable  
if its distribution $\P_X$
is a Gaussian measure on  $\R$.
With these preparations we can now define both $E$-valued Gaussian random variables and Gaussian measures
on $E$. 

\begin{definition}
Let $(\Om,\sA,\P)$  be a probability space and $E$ be a separable Banach space. Then $X:\Om\to E$ is 
a Gaussian random variable, if 
$\dualpair {x'}X E:\Om\to \R$ is a Gaussian random variable 
for all $x'\in E'$.

Moreover, a probability measure $\Qm$ on $\sborelEx E = \sborelnormx E$ is a Gaussian measure, if
the image measure $\Qm_{x'}$ on $\R$ is a Gaussian measure for all $x'\in E'$.  

Finally, if we have   another  separable Banach spaces $F$ and a $Y:\Om\to F$, then we say that 
 $X$ and $Y$ are jointly Gaussian
random variables, if $(X,Y):\Om\to E\times F$ is a Gaussian random variable, that is, 
\begin{align}\label{def:joint-gauss-eq}
\dualpair{x'}XE + \dualpair{y'}YF :\Om\to \R
\end{align}
is a Gaussian random variable
for all 
$x'\in E'$ and $y'\in F'$.
\end{definition}

Since $E$ is separable, Gaussian random variables are automatically strongly measurable. Moreover, 
it directly follows from the definitions that 
$\Qm$ is a Gaussian measure   on $\sborelEx E$, if and only if the identity map 
$\id_E:E\to E$ defined on the probability space $(E, \sborelEx E, \Qm)$ is a Gaussian random variable.
%
%
%
%
%
Furthermore, recall that 
Fernique's theorem, see e.g.~\cite[Corollary 3.2]{LeTa91}, shows that Gaussian random variables 
$X:\Om\to E$ on some probability space $(\Om,\sA,\P)$  satisfy 
\begin{align}\label{eq:Fernique's-theorem}
\int_\Om \snorm {X(\om)}_E^p \intd \P(\om) < \infty
\end{align}
for all $p\in [1,\infty)$. Consequently, for every Gaussian random variable $X$ its
expectation $\E X$ and covariance $\cov(X)$ are defined. Moreover, if $\Qm$ is a Gaussian measure on $\sborelEx E$, then
we define its expectation and covariance by those of the Gaussian random variable $\id_E$ mentioned above, that is 
\begin{align*}
\mu_\Qm &:= \E_\Qm \id_E = \int_E x \intd \Qm(x) \, , \\
\cov(\Qm)(x')  &:= \cov_\Qm (\id_E)(x')  = \int_E \dualpair {x'} {x-\mu_\Qm} E \cdot (x-\mu_\Qm) \intd \Qm(x) \, , \myqquad x'\in E'.
\end{align*}
Like in the finite dimensional case, the expectation and covariance
uniquely determine a Gaussian measure. Namely, if $\P$ and $\Qm$ are  
Gaussian measures on $\sborelEx E$ with $\mu_\P = \mu_\Qm$ and $\cov(\P) = \cov (\Qm)$, then we have $\P = \Qm$, see 
Proposition
\ref{prop:cf-of-general-gm} and the discussion at the beginning of Supplement \ref{app:cf}.
%
For  Gaussian measures $\Qm$ we thus write 
\begin{align}\label{eq:gauss-meas-ex-cov}
\gaussm {\mu_\Qm}{\cov(\Qm)}:= \Qm\, .
\end{align}
Finally, for a Gaussian random variable 
$X:\Om\to E$ on a probability space $(\Om,\sA,\P)$ 
we call 
\begin{align*}
\gaussspace X := \overline{\bigl\{  \dualpair{x'}{X-\E X}E: x'\in E' \bigr\}}^{\Lx 2 \P}
\end{align*}
the Gaussian Hilbert space associated to $X$. 
 We refer to  Lemma \ref{lem:gaussspace-of-grv} for more information.

%





Given jointly Gaussian
random variables $X:\Om\to E$ and $Y:\Om\to F$, the major goal of this work is to investigate the 
conditional distribution of $X$ given an observation $Y=y$. 
To make this mathematically rigorous, 
we need to recall the following definition.

\begin{definition}\label{def:reg-cond-prob}
Let $(T,\sB)$ and $(U,\sA)$ be measurable spaces and $\P$ be a probability measure on $(T\times U, \sB\otimes \sA)$.
Then we say that a map
$\P(\mycdot|\mycdot):\sB\times U\to [0,1]$ is a  regular conditional probability
of $\P$ given $U$,
if the following three
conditions are satisfied:
\begin{enumerate}
\item For all $u\in U$ the map $\P(\mycdot|u):\sB\to [0,1]$ is a probability measure.
\item For all $B\in \sB$ the map $\P(B|\mycdot): U\to [0,1]$ is $\sA$-measurable.
\item For all $B\in \sB$ and $A\in \sA$ we have
\begin{align}\label{eq:def-reg-cond-prob}
\P(B\times A) = \int_U \eins_A(u) \P(B|u) \intd \PU(u)\, .
\end{align}
Here, $\PU$ denotes the marginal distribution of $\P$ on $U$, that is $\P_U := \P_{\coordproj U}$, where $\P_{\coordproj U}$
is the image measure of $\P$ under the 
coordinate projection $\coordproj U:T\times U\to U$.
\end{enumerate}
\end{definition}

If $E$ and $F$ are separable Banach spaces  and $\Qm$ is a probability measure on 
$\sborelnormx E\otimes \sborelnormx F$, then there always exists a regular conditional probability
$\Qm(\mycdot|\mycdot):\sborelnormx E\otimes F\to [0,1]$ of $\Qm$ given $F$. In addition, this regular 
conditional probability is always $\Qm_F$-almost surely unique. These statements directly follow from a more general and well-known
result on the existence and uniqueness of regular conditional probabilities, which is recalled in Theorem \ref{thm:ex-ei-reg-cond-prob}.

Let us now apply these notion and results to jointly Gaussian
random variables $X:\Om\to E$ and $Y:\Om\to F$   defined on some probability space 
$(\Om,\sA,\P)$.
To this end,  let
 $\Qm := \P_{(X,Y)}$ be their joint distribution on $\sborelnormx E\otimes \sborelnormx F$.  
A quick calculation yields $\Qm_F = \P_Y$, and by our discussion above  there exists 
a $\P_Y$-almost surely unique regular conditional probability 
\begin{align*}
 \Qm(\mycdot|\mycdot) : \sborelnormx E\times F\to [0,1]
\end{align*}
of $\Qm$ given $F$.
We write $\P_{X|Y}(\mycdot|\mycdot) := \Qm(\mycdot|\mycdot)$ and speak of a version of the regular conditional probability of $X$ given $Y$.
Now recall that $\P_{X|Y}(B|y)$ describes the probability of $\{X\in B\}$ under the condition that we have observed $y=Y(\om)$.
Consequently, our goal is to 
 describe 
$\P_{X|Y}(\mycdot|\mycdot)$ modulo $\P_Y$-zero sets.
This task will be accomplished by  
\begin{enumerate}
\item showing that $\P_{X|Y}(\mycdot|y)$ is a Gaussian measure, and 
\item determining the expectation and covariance of $\P_{X|Y}(\mycdot|y)$.
\end{enumerate}
As discussed around \eqref{eq:gauss-meas-ex-cov} these two steps completely describe $\P_{X|Y}(\mycdot|y)$.

In the finite dimensional case, the solution of this task is folklore,
see e.g.~\cite[Appendix A.2]{RaWi06}, \cite[Chapter 2.3]{Bishop06}, \cite[Lemma 4.3]{HaStVoWi05a}, \cite[Theorem 6.20]{Stuart10a}, and also \cite[Section 4]{PfStHeWeXXa}.
In the following formulation of this solution, we write $A\mpinv$ for the Moore-Penrose pseudo-inverse of an 
$m\times n$-matrix $A$, and later we use the same notation for  linear maps $A:\R^n\to \R^m$. 

\begin{theorem}\label{thm:grv-conditioning-fin-dim}
Let $(\Om,\sA,\P)$ be a probability space and $X:\Om\to \R^m$ and $Y:\Om\to \R^n$ be
jointly Gaussian random variables. Then for every version
$\P_{X|Y}(\mycdot|\mycdot) :\sborel^m\times \R^n\to[0,1]$ of the regular conditional probability of $X$ given $Y$
there exists an $N\in \sborel^n$ with
$\P_Y(N) = 0$ such that for all $y\in \R^n\setminus N$ we have 
\begin{align*}
\P_{X|Y}(\mycdot|y) = \gaussm {\muup(y)}{\Kup}\, ,
\end{align*}
where 
\begin{align*}
\muup(y) &:= \E X + \KXY \KY\mpinv(y- \E Y) \, ,\\
\Kup     &:= \KX -\KXY \KY\mpinv \KYX \, .
\end{align*}
\end{theorem}


Finally, recall that conditional expectations $\E(X|\siC)$ and  $\E(X|Y)$
of Banach space valued random variables $X$ are defined 
analogously to the  $\R$-valued case. In addition, many properties of $\R$-valued conditional expectations 
still hold in the Banach space case. For details we refer to Definition \ref{def:con-exp} and the subsequent results in Supplement \ref{app:rcp+cond-ex}.


%% file: results.tex
\section{Main Results}\label{sec:main-results}

In this section we present the main results of this work that investigate the regular conditional probabilities arising from
conditioning one Banach space valued Gaussian random variable with respect to another one.

We begin by presenting a theorem
that generalizes and extends the finite dimensional Theorem \ref{thm:grv-conditioning-fin-dim} to the case in which the random variable 
$X$ is Banach space valued.  
Of course, for such $X$  we can no longer work with matrices, but apart from this difference, the first part of the following
theorem is an almost literal copy of  Theorem \ref{thm:grv-conditioning-fin-dim}.

\begin{theorem}\label{thm:grv-conditioning-y-fin-dim}
Let $(\Om,\sA,\P)$ be a probability space,  $E$ be a separable Banach space,  and $X:\Om\to E$ and $Y:\Om\to \R^n$ be
 be 
jointly Gaussian random variables. 
Moreover, let 
\begin{align}\label{thm:grv-conditioning-y-fin-dim-muup-def}
\muup(y) &:= \E X + \cov ( X,Y) (\cov Y)\mpinv(y- \E Y) \, , \myqquad y\in \R^n,\\ \label{thm:grv-conditioning-y-fin-dim-covup-def}
\covup     &:= \cov ( X) - \cov ( X,Y) (\cov Y)\mpinv \cov (Y, X) \, .
\end{align}
Then for every version
$\P_{X|Y}(\mycdot|\mycdot) :\sborelnormx E\times \R^n\to[0,1]$ of the regular conditional probability of $X$ given $Y$
there exists an $N\in \sborel^n$ with 
$\P_Y(N) = 0$ such that   
\begin{align*}
\P_{X|Y}(\mycdot|y) = \gaussm {\muup(y)}{\covup}
\end{align*}
for all  $y\in \R^n\setminus N$.
In addition, the random variable $Z:\Om\to E$ defined by
\begin{align*}
Z:=  \muup \circ Y =  \E X + \cov ( X,Y) (\cov Y)\mpinv(Y- \E Y)
\end{align*}
is both a version of $\E(X|Y)$ and a
 Gaussian random variable with $\E Z = \E X$ and
\begin{align*}
\cov (Z) &= \cov(X) - \covup\, , \\
\gaussspace {Z} &\subset  \gaussspace {Y} \, .
\end{align*}
\end{theorem}

As mentioned in the introduction, \cite[Corollary 3.10.3]{Bogachev98} has   shown that there is a version
$\P_{X|Y}(\mycdot|\mycdot)$, such that $\P_{X|Y}(\mycdot|y)$  is Gaussian for \emph{all} $y$.
Moreover, for more specific cases, 
formulas for the mean and covariance have also been derived by other authors as laid out in the introduction. 
However, we could not find such formulas in the full generality presented in Theorem \ref{thm:grv-conditioning-y-fin-dim}.
Since these formulas play a crucial role for our approximation scheme in the case of infinite dimensional $Y$, we thus decided 
to present and prove Theorem \ref{thm:grv-conditioning-y-fin-dim} as a   result.

Our next and main goal of this section is to generalize Theorem \ref{thm:grv-conditioning-y-fin-dim} to infinite dimensional $Y$. Here we recall that we cannot expect a literal
generalization. To illustrate this, let us assume for simplicity that $H:= F$ is an infinite dimensional Hilbert space. Using the
canonical Riesz-Fr\'echet-identification of $H'$ with $H$, we can then view $\cov (Y)$ as a compact linear operator $\cov (Y):H\to H$.
Now recall that a bounded linear  $A:H\to H$ has a Moore-Penrose inverse if and only if its image $\ran A$ is closed, see e.g.~\cite[Theorem 2.4]{HaRoSi01}.
Moreover,  it is well-known that compact operators have a closed range, if and only if their range is actually finite dimensional, see e.g.~\cite[Proposition 3.4.6]{Megginson98}.
 Consequently,
the Moore-Penrose inverse $(\cov Y)\mpinv$ exists if and only if $\ran \cov (Y)$ is finite dimensional, and 
the latter is only possible, if the range of $Y$ is almost surely contained in a finite-dimensional subspace of $F$,
see e.g.~\cite[Theorem 2.1]{LaGatta13a} or \cite[Theorem 2.43]{Sullivan15}. In other words, even in the Hilbert space case there is no hope for a literal generalization of \eqref{thm:grv-conditioning-y-fin-dim-muup-def} and \eqref{thm:grv-conditioning-y-fin-dim-covup-def}
to truly infinite dimensional $Y$.
 

In the following, we will resolve this issue by considering a sequence $(Y_n)$ of suitably chosen finite dimensional ``projections''
of $Y$. For each $Y_n$ we can then apply Theorem \ref{thm:grv-conditioning-y-fin-dim}, and it remains to show that the resulting
sequence of regular conditional probabilities $\P_{X|Y_n}$ converges to the regular conditional probability $\P_{X|Y}$.
To begin with, the following definition introduces the type of ``projections'' we need for this approach.

\begin{definition}\label{def:filt-seq}
Let $F$ be a separable Banach space and $(\nfilts n)$ be a sequence of bounded linear operators $\nfilts n:F\to \R^n$. Then we say that 
$(\nfilts n)$ is a filtering sequence for $F$, if we have $\s(\nfilts n)\subset \s(\nfilts {n+1})$  for all $n\geq 1$ and
\begin{align*}
\sborelEx F =  \s\bigl( \nfilts n: n\geq 1 \bigr)\, .
\end{align*}
\end{definition}

The measure-theoretic description of filtering sequences is useful for establishing our general
result on conditioning jointly Gaussian random variables, yet it is somewhat cumbersome to verify.
For this reason we present some simple approaches for constructing filtering sequences in Section   \ref{seq:filt-seq}.
In addition, we will see in Theorem \ref{thm:dense-gives-filter-seq}
that for any separable Banach space $F$ there   exists a filtering sequence. Furthermore, we show in Section   \ref{seq:filt-seq} that
for important classes of spaces filtering sequences can be easily constructed and applied. As an illustrative example, we will 
work out the details for the spaces $C(T)$ in Section \ref{sec:examples}.

With the help of filtering sequences we can now present our main result on conditioning one Gaussian random variable with respect
to another one.

\begin{theorem}\label{thm:grv-conditioning-y-inf-dim}
Let $(\Om,\sA,\P)$ be a probability space, $E$ and $F$ be separable Banach spaces, 
and $X:\Om\to E$ and $Y:\Om\to F$ be jointly  Gaussian random variables. 
Moreover, let $(\nfilts n)$ be a filtering sequence for $F$ and
\begin{align*}
Y_n := \nfilts n\circ Y \, , \myqquad n\geq 1.
\end{align*}
Then for all versions $\P_{X|Y_n}(\mycdot|\mycdot) :\sborelnormx E\times \R^n\to[0,1]$ and
$\P_{X|Y}(\mycdot|\mycdot) :\sborelnormx E\times F\to[0,1]$ of the regular conditional probabilities of $X$ given $Y_n$, respectively $Y$,
there exists an $N\in \sborelEx F$ with 
$\P_Y(N) = 0$ such that for all $y\in F\setminus N$
the following statements hold true:
\begin{enumerate}
\item For all $n\geq 1$ we have 
\begin{align*}
\P_{X|Y_n}(\mycdot|\nfilts  n(y)) = \gaussm {\muupn(\nfilts  n(y))}{\covupn}\, ,
\end{align*}
where 
$\muupn$ and $\covupn$ are defined as in 
\eqref{thm:grv-conditioning-y-fin-dim-muup-def}, respectively
\eqref{thm:grv-conditioning-y-fin-dim-covup-def}.
%
\item The measure  $\P_{X|Y}(\mycdot|y)$ is a Gaussian measure on $E$.
\item The expectation $\muup(y)$ of  $\P_{X|Y}(\mycdot|y)$ is given by
\begin{align*}
\muup(y) =  \E X +  \lim_{n\to \infty}\cov ( X,Y_n) (\cov Y_n)\mpinv(\nfilts  n(y)- \E Y_n) \, ,
\end{align*}
where  the convergence is with respect to $\snorm\cdot _E$.
\item The covariance operator $\covup:E'\to E$ of $\P_{X|Y}(\mycdot|y)$ is given by
\begin{align*}
\covup =  \cov ( X) - \lim_{n\to \infty} \cov ( X,Y_n) (\cov Y_n)\mpinv \cov (Y_n, X)\, ,
\end{align*}
where the convergence is with respect to the nuclear norm $\nucnorm\cdot$ and thus also with respect to the operator norm.
\item The weak convergence  $\P_{X|Y_n}(\mycdot| \nfilts  n(y)) \to \P_{X|Y}(\mycdot|y)$ holds true.
\item If we define the random variable $Z:\Om\to E$ by $Z(\om) := 0$ for all $\om \in Y^{-1}(N)$ and 
\begin{align*}
Z(\om) :=  \E X + \lim_{n\to \infty}\cov ( X,Y_n) (\cov Y_n)\mpinv\bigl(    Y_n(\om)- \E Y_n\bigr) \, , \myqquad \om \in \Om\setminus  Y^{-1}(N),
\end{align*}
then $Z$ is both a version of $\E(X|Y)$ and 
a Gaussian random variable with  $\E Z = \E X$ and 
\begin{align*}
\cov (Z) &= \cov(X) - \covup\, .
\end{align*}
Finally, the limit in the definition of $Z$ also holds in the $\sLx 2 {\P,E}$-sense.
\item We have $\gaussspace {Z} \subset  \gaussspace {Y}$ and
the orthogonal projection $\orthproj_Y:\gaussspace{X,Y} \to \gaussspace{X,Y}$ onto the subspace $\gaussspace Y$ is given by
\begin{align*}
\orthproj_Y f = \E(f|Y)\, , \myqquad f\in \gaussspace {X,Y}.
\end{align*}
\item For all $x'\in E'$, for which $\dualpair{x'}XE:\Om\to \R$ is $\s(Y)$-measurable, and all $\om \not \in Y^{-1}(N)$ we have
\begin{align*}
\dualpairb {x'} {\muup (Y(\om))}E &= \dualpair{x'}{X(\om)}E \\
\covup x' &= 0 \, .
\end{align*}
\end{enumerate}
\end{theorem}

In a nutshell, \emph{i)} simply applies Theorem \ref{thm:grv-conditioning-y-fin-dim}
to the random variables $X$ and $Y_n$. The statements \emph{ii)} to \emph{v)} provide the Gaussianity of the 
conditional probabilities as well as the approximations announced in the introduction.  
In addition, \emph{vi)}  shows
\begin{align*}
\muup \circ Y = \E(X|Y),
\end{align*}
%
which below we will investigate in more detail.
Furthermore, \emph{vii)} describes the behaviour of  the orthogonal projection $\E(\mycdot|Y):\Lx 2 \P\to \Lx 2 \P$ when it is
restricted to 
$\gaussspace{X,Y}$.
Here we note that in the observational model \eqref{eq:obs-model}, we have $\gaussspace {X,Y} = \gaussspace X$ and $\gaussspace Y \subset \gaussspace X$,
as one can quickly derive from Lemma \ref{lem:gaussspace-of-grv},
and therefore
the conditional expectation considered in \emph{vi)} 
becomes the orthogonal projection $\orthproj_Y:\gaussspace{X} \to \gaussspace{X}$ onto the subspace $\gaussspace Y$.
Finally, \emph{viii)} generalizes the well-known concentration of Gaussian 
posterior measures on the observations from the finite to the infinite dimensional case.

 Our next goal is to shed some more light on the structure of the conditional distributions. 
To this end,
let us assume that we are in the situation of Theorem \ref{thm:grv-conditioning-y-inf-dim}. 
Then by \eqref{thm:grv-conditioning-y-fin-dim-muup-def}
each $\muupn \circ \nfilts n:F\to E$ is continuous and thus strongly measurable, where we use the $\s$-algebra  $\sborelnormx F = \sborelEx F$ on $F$.
Let us define the map $\muup :F\to E$ by
\begin{align}\label{eq:muup-on-F}
\muup (y) := \lim_{n\to \infty} \eins_{F\setminus N}(y) \cdot \muupn(\nfilts n( y))\, , \myqquad y\in Y\, ,
\end{align}
which is possible by and consistent with  part  \emph{iii)}. Then $\muup$ is the pointwise limit of a sequence of strongly measurable functions, and therefore also (strongly)
measurable.
Moreover, the definition of $Z$ in
part \emph{vi)} gives
\begin{align*}
Z = \muup \circ Y\, ,
\end{align*}
and since $Z$ is a version of  $\E(X|Y)$, we see that $\muup$ is a version of the ``factorized'' conditional expectation $\E(X|Y=\mycdot):F\to E$.

Now note that $\muup$ carries all information required to reconstruct the Gaussian measures $\P_{X|Y}(\mycdot|y)$ for $y\in F\setminus N$:  Namely we have 
\begin{align*}
\P_{X|Y}(\mycdot|y) = \gaussm {\muup(y)}{\cov(X) - \cov( \muup  \circ Y)  }
\end{align*}
for all $y\in F\setminus N$,
where we note that $\E \P_{X|Y}(\mycdot|y) = \muup (y)$ directly follows from \emph{iii)}, while for the covariance we used 
\emph{vi)} and 
\begin{align*}
\cov \P_{X|Y}(\mycdot|y) =  \covup = \cov(X) - \cov(Z) = \cov(X) - \cov( \muup \circ Y)\, .
\end{align*}
%
This observation raises the question, if
we also obtain a version of $\P_{X|Y}(\mycdot|\mycdot)$ by setting 
\begin{align}\label{eq:rcp-from-fact-cond-exp}
\P_{X|Y}(\mycdot|y) := \gaussm {\mu(y)}{\cov(X) - \cov( \mu  \circ Y)  } \, , \myqquad y\in Y
\end{align}
for an \emph{arbitrary} version $\mu$ of $\E(X|Y=\mycdot):F\to E$. Since each such $\mu$ coincides $\P_Y$-almost surely with the above $\muup$
it is not overly surprising that the construction \eqref{eq:rcp-from-fact-cond-exp} is indeed possible. Yet it turns out in the proof of the following result that some work is 
required to ensure the measurability of $y\mapsto \P_{X|Y}(B|y)$ for all $B\in \sborelnormx E$. 


\begin{theorem}\label{thm:from-fce-to-rcp}
Let $(\Om,\sA,\P)$ be a probability space, $E$ and $F$ be separable Banach spaces, 
and $X:\Om\to E$ and $Y:\Om\to F$ be jointly  Gaussian random variables.
Moreover, let $\mu:F\to E$ be a measurable  such that 
\begin{align*}
Z := \mu  \circ Y 
\end{align*}
is a version of  $\E(X|Y)$. Then 
\eqref{eq:rcp-from-fact-cond-exp} defines  a version   of the regular conditional probability of $X$ given $Y$. 
\end{theorem}

%


We have seen around \eqref{eq:muup-on-F} that  a filtering sequence in combination with
Theorem \ref{thm:grv-conditioning-y-inf-dim} can be used to construct a version of $\E(X|Y=\mycdot):F\to E$. The next result shows that
we can refine this construction to have a better control over the nuisance $\P_Y$-zero set $N$.
Combining this with the construction \eqref{eq:rcp-from-fact-cond-exp} we can then also find a version $\P_{X|Y,\nfiltseq}(\mycdot|\mycdot)$
that we can control for \emph{all} $y\in F$. In particular,  this will generalize
 \cite[Corollary 3.10.3]{Bogachev98} mentioned in the introduction from finite dimensional to infinite dimensional $Y$.

\begin{theorem}\label{thm:canonical-rcp}
Let $(\Om,\sA,\P)$ be a probability space, $E$ and $F$ be separable Banach spaces,
and $X:\Om\to E$ and $Y:\Om\to F$ be jointly  Gaussian random variables.
Moreover, let $\nfiltseq:= (\nfilts n)$ be a filtering sequence for $F$ and
\begin{align*}
\Fconv := \bigl\{  y\in F: \exists  \lim_{n\to \infty}\cov ( X,Y_n) (\cov Y_n)\mpinv(\nfilts  n(y)- \E Y_n) \bigr\}\, ,
\end{align*}
where $Y_n := \nfilts n\circ Y $ for all $n\geq 1$. Then the following statements hold true:
\begin{enumerate}
\item We have $\Fconv \in \sborelEx F$ with $\P_Y(\Fconv) = 1$.
\item The set $F_{X|Y, \nfiltseq} := -\E Y + \Fconv$ is a measurable subspace of $F$ with $\ran \cov(Y,X)\subset F_{X|Y, \nfiltseq}$.
\item The map $\muupA:F\to E$ defined by
\begin{align*}
\muupA (y) := \E X +   \lim_{n\to \infty} \eins_{\Fconv}(y)   \cov ( X,Y_n) (\cov Y_n)\mpinv(\nfilts  n(y)- \E Y_n)\, , \qquad y\in Y\, ,
\end{align*}
is   measurable and $\muupA \circ Y$ is a version of $\E(X|Y)$.
\item The map $\P_{X|Y, \nfiltseq}(\mycdot|\mycdot) :\sborelnormx E\times \R^n\to[0,1]$ defined by 
\begin{align*}
\P_{X|Y, \nfiltseq}(\mycdot|y) := \gaussm {\muupA(y)}{\cov(X) - \cov( \muupA  \circ Y)  } \, , \myqquad y\in Y
\end{align*}
is a version of the regular conditional probability of $X$ given $Y$.
\end{enumerate}
\end{theorem}

Note that if 
 $X$ and $Y$ are centered, then $\Fconv = F_{X|Y, \nfiltseq}$ holds, and therefore $\Fconv$ is a measurable subspace of full measure $\P_Y$.
Moreover, for all $y\in \Fconv$ we have
\begin{align*}
\muupA (y) =  \lim_{n\to \infty}   \cov ( X,Y_n) (\cov Y_n)\mpinv \nfilts  n(y) \, .
\end{align*}
Consequently, the restriction of $\muupA$  to $\Fconv$ is linear and measurable. 
Note that in the Hilbert space case such linear maps also occurred in
\cite{Mandelbaum84a}. 

 Let us finally discuss how the
  theorems above  do provide some frequentist consistency and contraction results
  for
the Bayesian method given by the prior distribution $\P_X$ and the conditional distributions.  
  To this end, let us assume that we have an $E$-valued Gaussian random variable $X$ on some probability space $(\Om,\sA,\P)$
  and 
a filtering sequence $\nfiltseq :=(\nfilts n)$ for $E$. 
Moreover, assume that we cannot observe $X$ but only 
$X_n := \nfilts n\circ X$. This gives us versions $\P_{X|X_n}(\mycdot|\mycdot)$ of the regular conditional
probabilities of $X$ given $X_n$ with 
\begin{align*}
\P_{X|X_n}(\mycdot|\nfilts n (x)) = \gaussm {\muupxn(x)}{\covupxn} 
\end{align*}
for all $x\in E$,
where 
\begin{align*}
\muupxn(\nfilts n (x)) &= \E X + \cov ( X,X_n) (\cov X_n)\mpinv(\nfilts n (x)- \E X_n) \, ,\\
\covupxn     &= \cov ( X) - \cov ( X,X_n) (\cov X_n)\mpinv \cov (X_n, X) \, .
\end{align*}
To describe their behavior for $n\to \infty$, we set $Y:= X$. Then our results above 
give an $N\in \sA$ with $\P(N) = 0$ such that for all $\om \in \Om\setminus N$ we have 
\begin{align}\label{eq:conc-of-mean}
\muupxn(X_n(\om))  \to \muupx((X(\om)) =   X(\om)
\end{align}
with convergence in $E$,  where the last identity follows from  \emph{viii)} of Theorem \ref{thm:grv-conditioning-y-inf-dim}
in combination with
Hahn-Banach's theorem since 
  $\dualpair{x'}XE:\Om\to \R$ is $\s(X)$-measurable for all $x'\in E'$.
%
%
%
%
Similarly, 
part  \emph{viii)} of Theorem \ref{thm:grv-conditioning-y-inf-dim} also shows 
$\covupx x' = 0$ for all $x'\in E'$, that is $\covupx = 0$, and hence we find 
\begin{align}\label{eq:conc-of-cov}
\covupxn \to \covupx = 0
\end{align}
with respect to the nuclear norm. Part \emph{v)} of  Theorem \ref{thm:grv-conditioning-y-inf-dim}
thus yields the weak convergence 
\begin{align}\label{eq:conc-of-rcps}
\P_{X|X_n}(\mycdot|X_n(\om)) \to \gaussm {X(\om)}0
\end{align}
for all $\om \in \Om\setminus N$. In other words, with increasing $n$, the distributions 
$\P_{X|X_n}(\mycdot|X_n(\om))$ concentrate more and more around the observation $X(\om)$.
Note that these convergence results can also be interpreted from a frequentist point of view for
the Bayes method given by the posterior distributions $\P_{X|X_n}$ and the prior distribution $\P_X$, if it 
 is assumed that the ``true'' parameter is some $X(\om_0)$: Namely,  the convergence in
 \eqref{eq:conc-of-rcps} 
 provides a consistency result in the sense of
\cite[Chapter 6.2]{GhVa17}, while \eqref{eq:conc-of-mean} and \eqref{eq:conc-of-cov} describe 
a contraction result for the parameters in the sense of 
 \cite[Chapter 7.3.1]{GiNi16}.



%% file: filt-seq.tex
\section{Filtering Sequences}\label{seq:filt-seq}

Theorems \ref{thm:grv-conditioning-y-inf-dim} and \ref{thm:canonical-rcp} crucially depend on the existence  of
filtering sequences and our ability to construct and apply them. In this section, we therefore investigate filtering sequences in some detail. 
In particular, we show  that for separable
Banach spaces $F$ there always exists a filtering sequence. In addition, we  present
some natural and readily available filtering sequences for Hilbert spaces and reproducing kernel Hilbert spaces,
as well as for the spaces $\sC T$.
Finally, we also investigate  both $C^1([0,1])$  and general Banach spaces consisting of functions.

Our first result shows that a sufficiently rich sequence  $(y_i')\subset F'$ can be used to construct 
a filtering sequence. It thus translates
 the measure-theoretic notion
of filtering sequences into the language of functional analysis.

\begin{proposition}\label{prop:proper-filt-seq}
Let $F$ be a  separable Banach space and $(y_i')\subset F'$  with 
\begin{align}\label{prop:proper-filt-seq-h1}
\overline{\spann \{ y_i': i\geq 1  \}}^{\tauweaks} = F'\, ,
\end{align}
where on the left  side the $\tauweaks$-closure is considered. For $n\geq 1$ we
define $\nfilts n:F\to \R^n$ by
\begin{align}\label{prop:proper-filt-seq-h2}
\nfilts n(y) := \bigl(\dualpair {y_1'}{y}F,\dots, \dualpair {y_n'}{y}F\bigr)\, , \myqquad y\in F.
\end{align}
Then $(\nfilts n)$ is a filtering sequence for $F$.
\end{proposition}

Note that \eqref{prop:proper-filt-seq-h1} is, for example, satisfied, if we have a sequence $(y_i')\subset F'$, for which \eqref{prop:proper-filt-seq-h1} even holds for the
norm-closure, that is, 
\begin{align}\label{prop:proper-filt-seq-h1-norm}
\overline{\spann \{ y_i': i\geq 1  \}}^{\snorm\mycdot_{F'}} = F'\, .
\end{align}
Unfortunately, a  sequence $(y_i')\subset F'$ with \eqref{prop:proper-filt-seq-h1-norm} exists if and only if $F'$ is separable. As a consequence, 
spaces such as $F:=\sC{[0,1]^d}$  do not enjoy \eqref{prop:proper-filt-seq-h1-norm}, and for this reason we will mostly focus on \eqref{prop:proper-filt-seq-h1}.

Now, our first main result of this section shows the existence of   filtering sequences for generic, separable $F$, where
we recall from Theorem \ref{thm:alaoglu-and-more}  that for such $F$ we always have a countable and $\tauweaks$-dense $\denseblo_{F'} \subset B_{F'}$.

\begin{theorem}\label{thm:dense-gives-filter-seq}
Let $F$ be a separable Banach space and
$\denseblo_{F'} \subset B_{F'}$ be countable and $\tauweaks$-dense with enumeration $\denseblo_{F'} = \{y_i': i\geq 1\}$.
For $n\geq 1$ we
define $\nfilts n:F\to \R^n$ by \eqref{prop:proper-filt-seq-h2}.
Then $(\nfilts n)$ is a filtering sequence for $F$.
\end{theorem}

Combining Theorem \ref{thm:dense-gives-filter-seq} with Theorem \ref{thm:grv-conditioning-y-inf-dim} we see, for example, that the regular
conditional probabilities $\P_{X|Y}$ of jointly Gaussian random variables $X$ and $Y$ are always almost surely Gaussian measures.
Moreover, we can always approximate their means and covariances by the described finite dimensional projections.

Proposition \ref{prop:proper-filt-seq} shows that sequences $(y_i')\subset F'$ 
satisfying \eqref{prop:proper-filt-seq-h1}
can be used to construct filtering sequences. 
Unfortunately, working with the $\tauweaks$-closure in \eqref{prop:proper-filt-seq-h1} is often rather complicated. For this reason the 
 following result 
provides a characterization of  sequences satisfying \eqref{prop:proper-filt-seq-h1} that can be easier verified in some situations. 

\begin{proposition}\label{prop:sep-sequences}
Let $F$ be a  separable Banach space and $(y_i')\subset F'$ be a sequence. Then the following statements are equivalent:
\begin{enumerate}
\item We have $\overline{\spann \{ y_i': i\geq 1  \}}^{\tauweaks} = F'$.
\item The sequence $(y_i')\subset F'$ is separating, that is,  for all $y\in F\setminus \{0\}$   there is an $i\geq 1$ with $\dualpair{y'_i}yF \neq 0$.
\end{enumerate}
\end{proposition}

Let us now turn to more specific spaces. 
The following corollary considers
orthonormal bases (ONBs) in separable  Hilbert spaces.

\begin{corollary}\label{cor:filt-seq-HS}
Let $H$ be a separable Hilbert space with $\dim H = \infty$ and $(e_n)_{n\geq 1}$ be an ONB of $H$.
For $n\geq 1$ we
define $\nfilts n:H\to \R^n$ by
\begin{align*}
\nfilts n(y) := \bigl(\skprod {e_1}{y}_H,\dots, \skprod {e_n}{y}_H\bigr)\, , \myqquad y\in H.
\end{align*}
Then $(\nfilts n)$ is a filtering sequence for $H$.
\end{corollary}

Of course, it is possible to generalize Corollary \ref{cor:filt-seq-HS} to Banach spaces $F$ whose dual $F'$ have a Schauder base.
We skip the rather obvious details.

As described in the introduction  one of our goals is to
better understand GP4MLs, which are typically viewed to generate random functions.
To capture this in a general form, we say 
that a Banach space $F$ is 
a Banach spaces consisting of functions (BSF) on some set $T$, if $F$ consists of functions $T\to \R$, and the point evaluations $\diracf t:F\to \R$ 
given by $\diracf t f = f(t)$ are continuous for all $t\in T$.  


Before we show how these point evaluations can be used to construct filtering sequences, we note that 
 BSFs span a wide class of spaces. For example, the space $\ell_\infty(T)$ of all bounded functions $T\to \R$
equipped with the supremums norm $\inorm\cdot$ is a BSF. In addition, 
 the space $\sCb{T,\metric}$ of 
bounded, continuous functions  $T\to \R$ on some (pseudo) metric space $(T,\metric)$ is a BSF, if we again
equip $\sCb{T,\metric}$ with the supremums norm. Moreover, 
reproducing kernel Hilbert spaces (RKHSs) are covered, as they are by definition 
 both Hilbert spaces and BSFs, see 
 e.g.~\cite[Chapter 4.2]{StCh08} for details.
Furthermore, spaces of bounded, $\a$-H\"older continuous functions as well as spaces of bounded, continuously differentiable functions are  BSFs.
 In addition, Sobolev spaces can be viewed as BSFs if their smoothness is sufficiently large in the sense that Sobolev's embedding 
 theorem is satisfied.  BSFs have recently attracted some interest in the machine learning community, and we refer for more information on BSFs to e.g.~\cite{BaDVRoVi23a,ScSt25a}
 and the references mentioned therein. Finally, note that BSFs also play an important role for the PDE applications mentioned in the introduction, see e.g.~\cite{PfStHeWeXXa}.

 Now, if $F$ is a
  BSF  on $T$ and  $T_n:=\{t_1,\dots,t_n\}\subset T$, then 
%
%
 $\npeval {T_n} :F\to \R^n$ defined by
\begin{align}\label{eq:eval-op-bsf}
\npeval {T_n}(f) := \bigl(\dualpair {\diracf{t_1} }{f} F,\dots, \dualpair {\diracf{t_n} }{f} F\bigr)\, , \myqquad f\in F 
\end{align}
is a bounded linear operator.
The following result shows that in certain situations these operators can be used to construct filtering sequences.

\begin{theorem}\label{thm:bsf-in-cb}
Let $(T,\metric)$ be a separable pseudo-metric space and $F$ be a separable BSF on $T$ for which all $f\in F$ are continuous. 
Furthermore, let $T_\infty\subset T$ be a countable and dense subset with enumeration $T_\infty = \{t_i: i\geq 1\}$.
 Then $(\diracf {t_i})\subset F'$ is separating and
$(\npeval {T_n})$ is a filtering sequence for $F$, where
$\npeval {T_n}$ is defined by  \eqref{eq:eval-op-bsf}
for $T_n :=\{t_1,\dots,t_n\}$.
\end{theorem}

Recall that compact metric spaces $(T,\metric)$  are separable and that the corresponding space $\sC T$
is also  separable, see e.g.~\cite[Theorem V.6.6]{Conway90} or \cite[Corollary 11.2.5]{Dudley02}.
Therefore, we obtain the following corollary of Theorem \ref{thm:bsf-in-cb} without any further proof.

\begin{corollary}\label{cor:filt-seq-CT}
Let $(T,\metric)$ be a compact metric space and $T_\infty\subset T$ be a countable, dense subset with enumeration $T_\infty = \{t_i: i\geq 1\}$.
Then 
$(\npeval {T_n})$ is a filtering sequence for $\sC T$, where we have set  $T_n :=\{t_1,\dots,t_n\}$.
\end{corollary}

Our next goal is to apply Theorem \ref{thm:bsf-in-cb} to RKHSs. 
To this end, recall that every RKHS $H$ on $T$ has a unique kernel $k:T\times T\to \R$, which by the reproducing property 
can be used to represent the point evaluations. More precisely, we have 
\begin{align*}
\dualpair {\diracf{t} }{h} H = \skprod {k(t,\mycdot)}{h}_H\, , \myqquad t\in T,\, h\in H,
\end{align*}
see  e.g.~\cite[Chapter 4.2]{StCh08}. With these preparations, the application of Theorem \ref{thm:bsf-in-cb} to RKHSs reads as follows.

\begin{corollary}\label{cor:filt-seq-rkhs}
Let $(T,\metric)$ be a separable, pseudo-metric space and $H$ be an RKHS on $T$ with continuous kernel $k:T\times T\to \R$. 
Then $H$ is separable and for all countable  and dense subsets $T_\infty  = \{t_i: i\geq 1\}$ of $T$ the operators 
 $\npeval n:H\to \R^n$ defined by
\begin{align}\label{eq:eval-op-copy-rkhs}
\npeval n(h) := \bigl(\skprod {k(t_1,\mycdot)}{h}_H,\dots, \skprod {k(t_n,\mycdot)}{h}_H\bigr)  \, , \myqquad h\in H
\end{align}
give  a filtering sequence $(\npeval n)$   for $H$.
\end{corollary}

In the following, we show that for all separable BSFs $F$ on $T$ it is possible to construct 
a pseudo metric on $T$ such that the assumptions of Theorem \ref{thm:bsf-in-cb} are satisfied. 

We first consider the case of  RKHSs $H$ on some set $T$ since in this case, the construction is straightforward. 
To begin with, let 
 us recall, that
$\metric_k:T\times T\to [0,\infty)$ defined by
\begin{align*}
\metric_k(t_1,t_2)
:=
\snorm{k(t_1,\mycdot) - k(t_2,\mycdot)}_H
=
\sqrt{ k(t_1,t_1) - 2 k(t_1,t_2) + k(t_2,t_2)  } 
\end{align*}
for all $t_1,t_2\in T$
gives a pseudo-metric on $T$ called the kernel metric, see e.g.~\cite[Chapter 4.3]{StCh08}. With this preparation,
the RKHS case reads as follows.

\begin{proposition}\label{prop:filt-seq-RKHS}
Let $H$ be a separable RKHS on $T$ with kernel $k$. Then $(T,\metric_k)$ is a separable pseudo-metric space and 
all $h\in H$ are continuous with respect to $\metric_k$. 
\end{proposition}

Combining Proposition \ref{prop:filt-seq-RKHS} with Theorem \ref{thm:bsf-in-cb} we see that \eqref{eq:eval-op-copy-rkhs} 
defines a filtering sequence if we choose a countable and $\metric_k$-dense subset  $T_\infty$  of $T$.

Our next result will generalize Proposition \ref{prop:filt-seq-RKHS} to separable BSFs $F$. Here, a natural idea is to 
pull back a metric from $F'$ to $T$ 
with the help of the ``feature map'' $\Phi:T\to F'$
given by $t\mapsto \diracf t$
 as we have done for the kernel metric. Unfortunately, however, we can  no longer take the norm of $F'$ for this approach as 
the needed separability of $T$  would
require $F'$ to be separable. Moreover, simply taking a metric describing the $\tauweaks$-topology on a (scaled) ball
of $B_{F'}$ is not possible either, as the above $\Phi$ is bounded if and only if $F\subset \ell_\infty(T)$, and the latter is  
an  assumption we would like to avoid.
Nonetheless, the general idea is possible if we consider a  modified map and a different metric, as the proof 
of the following theorem reveals.

\begin{theorem}\label{thm:good-metric-for-bsf}
Let $F$ be a separable BSF on $T$. Then there exists a   pseudo-metric $\metric$ on $T$ 
such that $(T,\metric)$ is separable  and  all $f\in F$ are continuous with respect to $\metric$. In addition, there 
exists a countable subset $T_\infty  = \{t_i: i\geq 1\}$ of $T$ such that    $(\diracf{t_i})\subset F'$ is separating. 
\end{theorem}

To illustrate  the consequences of Theorem \ref{thm:good-metric-for-bsf}
we fix a  separable BSF $F$ on $T$ and a set
$T_\infty  = \{t_i: i\geq 1\}$   as in Theorem \ref{thm:good-metric-for-bsf}. Then Proposition
\ref{prop:sep-sequences} gives 
\begin{align*}
\overline{\spann \{ \diracf{t_i}': i\geq 1  \}}^{\tauweaks} = F'\, ,
\end{align*}
and Theorem \ref{thm:bsf-in-cb} ensures that 
$(\npeval {T_n})$ is a filtering sequence for $F$, where we have set  $T_n :=\{t_1,\dots,t_n\}$.
In other words, for separable BSFs it is always possible to apply our main  Theorems \ref{thm:grv-conditioning-y-inf-dim} and \ref{thm:canonical-rcp} 
with filtering sequences based upon point evaluations.

Finally we  present an alternative to point evaluations in the BSF context.
To keep the presentation short, we   focus on the space
 $C^1([0,1])$ of continuous functions $f:[0,1]\to \R$ that are continuously differentiable on $(0,1)$
and for which there exists a continuous $g:[0,1]\to \R$ with $g_{|(0,1)} = f'$. As usual, we equip this space
with the norm
\begin{align*}
\snorm f_{C^1([0,1])} := \inorm f + \inorm{f'}\, , \myqquad f\in C^1([0,1]).
\end{align*}
It is well-known that $C^1([0,1])$ is a separable  BSF. Now, the following simple lemma
provides a separating sequence for $C^1([0,1])$ and thus also a filtering sequence via Proposition \ref{prop:sep-sequences} and
Proposition \ref{prop:proper-filt-seq}.

\begin{lemma}\label{lem:c1-sep-seq}
Let $\{t_i: i\geq 1\}\subset (0,1)$ be dense in $[0,1]$ and $y_i':C^1([0,1])\to \R$ be defined  by
$y_i'(f) := f'(t_i)$ for all $f\in C^1([0,1])$.
If we further set $y_0' := \diracf 0$, then $(y_i')_{i\geq 0}$ is a separating sequence for $C^1([0,1])$.
\end{lemma}

Note that
Lemma \ref{lem:c1-sep-seq} actually
provides an alternative to  point evaluations since for $i\geq 1$ the functionals $y_i'$ do not access $f$ but its derivative.
Moreover, generalizations of Lemma \ref{lem:c1-sep-seq} to higher dimensions and derivatives are straightforward, and therefore
 our results  of Section \ref{sec:main-results} can indeed be applied to the PDE
applications mentioned in the introduction.


%% file: examples.tex
\section{Conditioning Continuous Gaussian Processes}\label{sec:examples}

In this section we illustrate how the general results of Section \ref{sec:main-results} can be applied to the GP4ML setting. Namely, we will 
consider the conditioning of Gaussian random variables for 
the spaces 
 $E: = \sC T$ and $F:=\sC S$, where $(T,\metric)$ is a compact metric space and   $S\subset T$ is closed.
 In addition, we will discuss how these results can be formulated from a 
 Gaussian process perspective. 

To apply our general results we first note that  $\sC T$ is a separable, see  e.g.~\cite[Theorem V.6.6]{Conway90} or \cite[Corollary 11.2.5]{Dudley02}.
Moreover, to describe the dual of  $\sC T$, we write
$\measpace T$ for the space  of all finite, signed Borel measures equipped with the
total variation norm. Riesz' representation theorem then shows that the map $\Cdualmap \cdot \, :\measpace T  \to \sC T'$ given by 
\begin{align}
\label{eq:riesz-represent}
\nu & \mapsto  \Cdualmap\nu := \biggl( f\mapsto 
\int_T f\intd \nu  \biggr)
\end{align}
is an isometric isomorphism,
see e.g.~\cite[Theorem 7.3.6]{Cohn13} and \cite[Corollary 29.13]{Bauer01}.
In addition,  note that \eqref{eq:riesz-represent} maps the Dirac measures $\dirac t$ to the point 
%
evaluation functionals $\diracf t$. 

To translate our main results to Gaussian processes, we also need 
a well-known, canonical one-to-one relationship between Gaussian processes with continuous paths and $\sC T$-valued Gaussian
random variables. To begin with  we fix a  probability space
$(\Om,\sA,\P)$ and a stochastic process $(X_t)_{t\in T}$ over $\Om$.
Then $(X_t)_{t\in T}$
%
is called a Gaussian process  if, for all
$n\geq 1$ and all  $t_1,\dots,t_n\in T$, the
random variable $(X_{t_1},\dots,X_{t_n}):\Om\to \R^n$ is   Gaussian.
Moreover, for
%
a fixed $\om \in \Om$, the map $\sppath X(\om): T\to \R$ given by 
\begin{align*}
t&\mapsto X_t(\om)
\end{align*}
is  called a  path.
%
We say that the process has continuous paths, if $\sppath X(\om) \in \sC T$ for all $\om \in \Om$.
Clearly, if only $\P$-almost all paths are continuous, then we can easily construct a version of  $(X_t)_{t\in T}$ having
 continuous paths. In the following, we assume that we have already chosen such a version.

To describe the above mentioned relationship, let us now fix a 
map $X:\Om\to \sC T$. Then we  obtain
a family $(X_t)_{t\in T}$ of functions $X_t:\Om\to \R$ by setting
\begin{align}\label{eq:rv-gives-sp}
X_t(\om) := \dualpair  {\diracf t} {X(\om)} {\sC T} = (X(\om))(t) \, , \myqquad \om\in \Om.
\end{align}
Obviously, this construction ensures that all paths $t\mapsto X_t(\om)$ are continuous.
Conversely, if we have a family $(X_t)_{t\in T}$ of functions $X_t:\Om\to \R$ such that for each $\om \in \Om$ the path $\sppath X(\om)$ is
continuous, then
\begin{align} 
\om & \mapsto \sppath X(\om)
\label{eq:sp-gives-rv}
\end{align}
defines a map $X:\Om \to \sC T $.
Clearly, the operations  \eqref{eq:rv-gives-sp} and \eqref{eq:sp-gives-rv} are inverse to each other.
The following folklore result shows that they map $\sC T$-valued Gaussian random variables to Gaussian processes with continuous paths and vice versa.

\begin{lemma}\label{lem:sp-gives-rv}
Let $(T,\metric)$ be a compact metric space, $(\Om, \sA, \P)$ be a probability space,
$X:\Om\to \sC T$ be a map, and
$(X_t)_{t\in T}$ be given by \eqref{eq:rv-gives-sp}. Then the following statements are equivalent:
\begin{enumerate}
\item $X$ is a $\sC T$-valued Gaussian random variable.
\item $(X_t)_{t\in T}$ is a Gaussian process with continuous paths.
\end{enumerate}
In this case, both the mean and the covariance function are continuous.
\end{lemma}

It is well-known that a Gaussian process $(X_t)_{t\in T}$ with continuous paths can be described by the mean  and covariance
functions given by \eqref{eq:gp4ml-mean} and \eqref{eq:gp4ml-cov}.
Moreover, if $X$ is the Gaussian random variable   associated to the process  via \eqref{eq:sp-gives-rv}, then
the mean $\E X$ and the covariance operator $\cov (X)$ can be used to compute the mean and covariance function, namely we have 
\begin{align} \label{eq:meanfct-by-mean}
m(t) & =  \E X_t =  \dualpair {\diracf t}{\E X}{\sC T}\, , \\  \label{eq:covfct-by-cov}
k(t_1,t_2) &= \cov(X_{t_1}, X_{t_2})  = \dualpair {\diracf {t_1}}{\cov (X) \diracf {t_2}}{\sC T}
\end{align}
for all $t,t_1,t_2\in T$, where we have used \eqref{eq:bochner-is-pettis}, respectively \eqref{eq:weak-cross-cov}.
Conversely, if we have the mean function $m$, then \eqref{eq:meanfct-by-mean} shows that we know the function $\E X \in \sC T$ pointwise, i.e.~we know $\E X$.
Moreover, the covariance function determines the covariance operator by
\begin{align*}
\dualpair {\Cdualmap\mu}{\cov (X) \Cdualmap\nu}{\sC T} = \int_T \int_T k(s,t) \intd\mu(s) \intd\nu(t) \, , \myqquad \mu,\nu \in \measpace T,
\end{align*}
see Lemma \ref{lem:cov-gp-comp} for details.
This justifies the notation
\begin{align*}
\gauss mk := \gaussm {\E X} {\cov X}
\end{align*}
for the distribution of a given Gaussian random variable $X:\Om\to \sC T$, where  $m$ and $k$ are the mean and covariance function of the associated
Gaussian process.

In summary, instead of  working with  Gaussian random variables $X:\Om\to \sC T$ we can alternatively consider
Gaussian processes $(X_t)_{t\in T}$ with continuous paths together with their mean and covariance function.
In the following we will   adopt this  point of view, which is common in the GP4ML
literature, see e.g.~\cite{Bishop06,RaWi06,KaHeSeSrXXa}.
In particular, we speak of conditioned processes, when
we condition the associated Gaussian random variables.

Let us now assume that we have a compact metric space  $(T,\metric)$, 
a Gaussian processes $(X_t)_{t\in T}$ with continuous paths, and  an observation
$(Y_s)_{s\in S}$
of this
process on a non-empty $S\subset T$. In other words, we have $Y_s(\om) = X_s(\om)$ for all $s\in S$ and $\om \in \Om$, where
$(\Om,\sA,\P)$ is the underlying probability space.
By the continuity of the paths we may assume without loss of generality 
 that $S$ is closed, and thus also a compact metric space. Now, 
to describe this setup by Gaussian random variables, we write
$E:=\sC T$ and $F:= \sC S$ and denote  the Gaussian random variables associated to the two processes above via
\eqref{eq:sp-gives-rv} by  $X:\Om\to \sC T$ and $Y:\Om\to \sC S$.
This gives
\begin{align}\label{eq:partial-observation}
Y(\om ) = (X(\om))_{|S}\, , \myqquad \om\in \Om,
\end{align}
that is $Y = \mycdot_{|S} \circ X$, where the restriction operator $\mycdot _{|S}:\sC T\to \sC S$ is bounded and linear.

We first consider the case of a finite set $S = \{s_1,\dots,s_n \}$.
Here we recall from \eqref{eq:eval-op-bsf} that the 
 bounded linear operator $\npeval S:\sC T\to \R^n$  is defined by
\begin{align}\label{eq:fin-eval-op}
\npeval S(f) := f(S) := \bigl(f(s_1),\dots,f(s_n)   \bigr)\, , \myqquad f\in \sC T\, .
\end{align}
In this case,  \eqref{eq:partial-observation} reduces to $Y = \npeval S\circ X$ as $\sC S$ can be canonically identified by $\R^n$.
Finally, note that if we have another finite set $R = \{r_1,\dots,r_m \}$, then the corresponding covariance operator
$\cov(\npeval {R} \circ X , \npeval S \circ X ): \R^n \to \R^m$ can be represented by the $m\times n$-matrix
\begin{align*}
K_{R,S}:= \bigl( \cov(X_{r_i}, X_{s_j})  \bigr)_{i,j} = \bigl( k(r_i,s_j)  \bigr)_{i,j}
\end{align*}
as discussed around \eqref{eq:cross-cov-matrix}. Furthermore, in the case $R=\{r\}$, we use the shorthand $K_{r,S} := K_{R,S}$, and an analogous
simplified notation is used in the case $S=\{s\}$.

With these preparations, the application of Theorem \ref{thm:grv-conditioning-y-fin-dim} to the specific
observational model \eqref{eq:partial-observation} with finite $S$ reads as follows.

\begin{theorem}\label{thm:grv-conditioning-y-fin-dim-CT}
Let $(\Om, \sA,\P)$ be a probability space, $(T,\metric)$ be a compact metric space, and $(X_t)_{t\in T}$ be a
Gaussian processes on $(\Om, \sA,\P)$ with continuous paths.
Moreover, let $X:\Om\to \sC T$ be the Gaussian random variable associated to the process via \eqref{eq:sp-gives-rv}
and $S := \{s_1,\dots,s_n \}\subset T$. Finally, let
\begin{align}\label{eq:CT-obs-mod}
Y := \npeval S\circ X\, ,
\end{align}
where $\npeval S:\sC T\to \R^n$ is defined by \eqref{eq:fin-eval-op}.
Then for every version
$\P_{X|Y}(\mycdot|\mycdot) :\sborelnormx {\sC T}\times \R^n\to[0,1]$ of the regular conditional probability of $X$ given $Y$
there exist an $N\in \sborel^n$ with
$\P_Y(N) = 0$ such that for all $y\in \R^n\setminus N$ we have
\begin{align*}
\P_{X|Y}(\mycdot|y) = \gaussm {\meanup}{\kfctup}\, ,
\end{align*}
where the mean and covariance functions are given by
\begin{align*}
\meanup(t) &= m(t) + K_{t,S}   K_{S,S}\mpinv(y- m(S)) \, ,\\
\kfctup(t_1,t_2)    &= k(t_1,t_2) -  K_{t_1,S}   K_{S,S}\mpinv K_{S,t_2} \, .
\end{align*}
\end{theorem}

Before we proceed with conditioning with respect to  observations on infinite $S$, we note that we can easily generalize the observational model
\eqref{eq:CT-obs-mod} to e.g.
\begin{align}\label{eq:obs-mod-plus-moise}
Y := \npeval S\circ X + \e\, ,
\end{align}
where $\e \sim \gauss 0 {\s^2 I_{n}}$ is independent of $X$, $\s\geq 0$, and $I_{n}$ denotes the $n$-dimensional identity matrix.
Indeed, in view of Theorem \ref{thm:grv-conditioning-y-fin-dim}, we only need to recompute the representing matrices of the operators $\cov(X,Y)$, $\cov(Y)$, and
$\cov(Y,X)$. However, we have, for example,
\begin{align*}
\cov(Y)
&= \cov\bigl( \npeval S\circ X + \e,  \npeval S\circ X + \e  \bigr) \\
&= \cov( \npeval S\circ X ,  \npeval S\circ X )  + \cov( \npeval S\circ X , \e )   + \cov( \e,  \npeval S\circ X  )   + \cov(  \e,   \e ) \\
&= \cov(Y) + \cov(\e)\, ,
\end{align*}
where in the last step we used the independence of $X$ and $\e$. Consequently, $\cov (Y)$ is represented by the matrix $K_{S,S} + \s^2 I_{n}$.
Some additional, but essentially
analogous calculations then show that the
corresponding mean and covariance functions can be computed by
\begin{align*}
\meanup(t) &= m(t) + K_{t,S}   (K_{S,S} + \s^2 I_{n})\mpinv(y- m(S)) \, ,\\
\kfctup(t_1,t_2)    &= k(t_1,t_2) -  K_{t_1,S}  (K_{S,S} + \s^2 I_{n})\mpinv K_{S,t_2}
\end{align*}
for all $t,t_1,t_2\in T$. Not surprisingly, these formulas are the well-known textbook formulas from the GP4MLs literature, see e.g.~\cite[Chapter 2.2]{RaWi06}. However, there is is subtle but important difference in their meaning: The mentioned textbook formulas do not ensure
that a Gaussian process with the above mean and covariance function has a version with continuous paths and even if this
can be derived by other means, they do not ensure that the resulting measure of such a version  equals the regular conditional probability on $\sC T$
we seek. In contrast, our approach establishes exactly these guarantees, so that the common functional interpretation of GP4MLs is indeed  justified.
This is similar in spirit to \cite[Corollary 4]{TrGi24a}, where more general functionals on $\sC T$ are considered 
in the noise free case $\e=0$.

Let us now turn to observations on infinite, closed $S\subset T$, where again $(T,\metric)$ is a compact metric space.  
Since in this case $(S,\metric)$ is also a compact metric space, there exists 
 a countable, dense $S_\infty\subset S$ with enumeration $S_\infty = \{s_i: i\geq 1\}$, and
we have seen in Corollary \ref{cor:filt-seq-CT} that 
the evaluation operators $\npeval {S_n}: \sC S\to \R^n$ for $S_n := \{s_1,\dots,s_n\}$
give  a filtering sequence   for $\sC S$.
Roughly speaking, this sequence provides us with  
   a mean to  approximate the infinite dimensional
observational random variable $Y := X_{|S}$  with the help of the sequence $(Y_n)$ of finite dimensional observations $Y_n := X_{|S_n}$ taken
at the points $S_n$.  
Applying Theorem \ref{thm:grv-conditioning-y-inf-dim} to this situation leads to:

\begin{theorem}\label{thm:grv-conditioning-y-inf-dim-CT}
Let $(\Om, \sA,\P)$ be a probability space, 
$(T,\metric)$ be a compact metric space, $(X_t)_{t\in T}$ be a
Gaussian process on  $(\Om, \sA,\P)$  with continuous paths, and 
 $X:\Om\to \sC T$ be the Gaussian random variable associated to the process via \eqref{eq:sp-gives-rv}. For  closed 
  $S \subset T$  let
\begin{align}\label{eq:CT-obs-mod-inf}
Y :=  X_{|S}\, .
\end{align}
Finally, let $S_\infty \subset S$ be countable and dense with enumeration $S_\infty = \{s_i: i\geq 1\}$. We write 
 $S_n := \{s_1,\dots,s_n\}$ for all $n\geq 1$.
Then for every version
$\P_{X|Y}(\mycdot|\mycdot) :\sborelnormx {\sC T}\times \sC S\to[0,1]$ of the regular conditional probability of $X$ given $Y$
there exist an $N\in \sborel(\sC S)$ with
$\P_Y(N) = 0$ such that for all $g\in \sC S\setminus N$ we have
\begin{align*}
\P_{X|Y}(\mycdot|g) = \gaussm {\meanup}{\kfctup}\, ,
\end{align*}
where the mean and covariance functions are given by
\begin{align}\label{thm:grv-conditioning-y-inf-dim-CT-mean}
\meanupg(t) &= m(t) + \lim_{n\to \infty}  K_{t,S_n}   K_{S_n,S_n}\mpinv(g(S_n)  - m(S_n)) \, ,\\ \label{thm:grv-conditioning-y-inf-dim-CT-cov}
\kfctup(t_1,t_2)    &= k(t_1,t_2) - \lim_{n\to \infty} K_{t_1,S_n}   K_{S_n,S_n}\mpinv K_{S_n,t_2} \,
\end{align}
and the convergence is uniform in $t\in T$, respectively in $t_1,t_2\in T$.
Moreover, for all $g\in \sC S\setminus N$, $s\in S$, and $t\in T$ we have
\begin{align}\label{thm:grv-conditioning-y-inf-dim-CT-mean-onS}
\meanupg(s) &= g(s) \, ,\\ \label{thm:grv-conditioning-y-inf-dim-CT-cov-onS}
\kfctup(s,t) & = 0 \, .
\end{align}
\end{theorem}

Of course, the remaining assertions of Theorem \ref{thm:grv-conditioning-y-inf-dim} as well as Theorems \ref{thm:from-fce-to-rcp} and \ref{thm:canonical-rcp} 
also hold in the   setup considered in
Theorem \ref{thm:grv-conditioning-y-inf-dim-CT}. 
For example,  we have the weak convergence  
\begin{align}\label{thm:grv-conditioning-y-inf-dim-CT-weakconv}
\P_{X| (X_{|S_n}) }(\mycdot| g(S_n)) \to \P_{X|(X_{|S})}(\mycdot|g)\, .
\end{align}
In addition, combining the approximations 
 \eqref{thm:grv-conditioning-y-inf-dim-CT-mean} and
\eqref{thm:grv-conditioning-y-inf-dim-CT-cov}   with  \eqref{thm:grv-conditioning-y-inf-dim-CT-mean-onS} and
\eqref{thm:grv-conditioning-y-inf-dim-CT-cov-onS} shows
%
%
\begin{align*}
 m(s) +   K_{s,S_n}   K_{S_n,S_n}\mpinv(g(S_n)  - m(S_n))  & \to g(s) \, , \\
 k(s,t) -   K_{s,S_n}   K_{S_n,S_n}\mpinv K_{S_n,t} &\to 0
\end{align*}
for $\P_Y$-almost all  observed paths  $g\in \sC S$ and
all $s\in S$ and $t\in T$, where again the convergence is  uniform in $s$, respectively $s$ and $t$.
Note that these findings substantially generalize the state-of-the-art contractions results established in \cite{KoPf21a}
 in the sense that neither $S=T$ nor additional assumptions on the observational model are needed.
Finally, to the best of our knowledge there is no known result ensuring \eqref{thm:grv-conditioning-y-inf-dim-CT-weakconv}
even under restricted assumptions.

%% file: proofs.tex
\section{Proofs of the Results of  Section \ref{sec:main-results}}\label{sec:proofs}

\subsection{Proof of Theorem \ref{thm:grv-conditioning-y-fin-dim}}


Before we can present the proof of Theorem \ref{thm:grv-conditioning-y-fin-dim} we need an auxiliary result. 
To motivate it, let $E$ be a Banach space, $T:E'\to E$ be bounded and linear, and  $x'\in E'$. If    $(x')':\R'\to E'$ denotes the dual of $x'$, then 
we can consider the 
bounded linear map $x' \circ T \circ (x')':\R'\to \R$. Since both $\R'$ and $\R$ are one-dimensional, this map 
can be described by a $1\times 1$-matrix, which is determined in the next lemma.

\begin{lemma}\label{lem:double-dual}
Let $E$ be a Banach space, $T:E'\to E$ be a bounded, linear operator, and $x'\in E'$. Then  we have
\begin{align*}
\bigl(x' \circ T \circ (x')'\bigr) (\skprod t\mycdot_\R)
=\dualpair {x'}{Tx'}E \cdot t\, , \myqquad   t\in \R,
\end{align*}
where $(x')':\R'\to E'$ denotes the dual of $x':E\to \R$ and
$\skprod \mycdot\mycdot_\R$ denotes the inner product of the Hilbert space $\R$.
\end{lemma}

\begin{proof}[Proof of Lemma \ref{lem:double-dual}]
Recall that 
the elements in $\R'$ can be represented  by $\R$ via
the canonical isomorphism $t\mapsto \skprod t\mycdot_\R$.
Now observe that for $t\in \R$ and $x\in E$ we have
\begin{align*}
\dualpairb {(x')'(\skprod t\mycdot_\R)} xE
= \dualpairb {\skprod t\mycdot_\R} {x'(x)}\R
= \skprod t{x'(x)}_\R
= t \cdot x'(x)
=   \dualpair {tx'}xE \, .
\end{align*}
In other words, we have
$(x')'(\skprod t\mycdot_\R) = tx'$ for all $t\in \R$. This in turn yields
\begin{align*}
\bigl(x' \circ T \circ (x')'\bigr) (\skprod 1\mycdot_\R) = (x' \circ T)\bigl((x')'(\skprod 1\mycdot_\R)\bigr) = (x'\circ T)(x') =   \dualpair {x'}{Tx'}E \, ,
\end{align*}
and by multiplying both sides with $t\in \R$ we obtain the assertion.
\end{proof}

\begin{proof}[Proof of Theorem \ref{thm:grv-conditioning-y-fin-dim}]
We begin by some preparations. To this end, we  fix a   version $\P_{X|Y}(\mycdot|\mycdot)$ of the regular conditional probability of   $X$ given $Y$.
Moreover, 
we choose a countable, $\tauweaks$-dense $\denseblo \subset B_{E'}$
with the help of Theorem \ref{thm:alaoglu-and-more}.
Finally, we fix an enumeration $(x'_i)$ of $\denseblo$, that is, $\denseblo = \{x_i': i\geq 1\}$.

For $i\geq 1$ we now define $X_i := \dualpair{x'_i}XE$.
Then $X_i:\Om\to \R$ and $Y:\Om\to \R^n$
are jointly Gaussian random variables with joint distribution $\P_{(X_i,Y)}$ on $\R\times \R^n$.
We 
fix a version $\P_{X_i|Y}(\mycdot|\mycdot)$ of the
regular
conditional probability of $X_i$ given $Y$.
By Theorem \ref{thm:grv-conditioning-fin-dim} there then exists 
 an $N_i'\in \sborel^n$ with 
$\P_Y(N_i') = 0$ such that for all $y\in \R^n\setminus N_i'$ we have 
\begin{align*}
\P_{X_i|Y}(\mycdot|y) = \gaussm {\muupi(y)}{\Kupi}\, ,
\end{align*}
where the expectation and covariance are given by 
\begin{align*}
\muupi(y) &:= \E X_i + \KXiY \KY\mpinv(y- \E Y) \, ,\\
\Kupi     &:= \KXi -\KXiY \KY\mpinv \KYXi \, .
\end{align*}
Using the definition of $X_i$, the expression for the expectation reads as  
\begin{align}\nonumber
\muupi(y) 
&=  \E X_i + \KXiY \KY\mpinv(y- \E Y) \\ \nonumber
&= \dualpair {x'_i} {\E X}E + \cov(x'_i\circ  X, Y)   \cov(Y)\mpinv (y- \E Y) \\ \nonumber
&= \dualpair {x'_i} {\E X}E +  \dualpairb {x'_i}  {\cov(X, Y)   \cov(Y)\mpinv (y- \E Y)}E \\ \label{thm:grv-conditioning-y-fin-dim-exp}
&= \dualpairb {x'_i} {\E X + \cov(X, Y)  \cov(Y)\mpinv (y- \E Y)}E  
\, ,
\end{align}
where we used the permutability of integration and operator application \eqref{eq:bochner-is-pettis}
and the cross covariance formula for linearly transformed random variables \eqref{eq:covs-of-compositions-new}.
To investigate $\Kupi$ in a similar fashion, we consider the bounded linear operator $\corXY:E'\to E$ defined by
 \begin{align*}
\corXY := \cov ( X,Y) (\cov Y)\mpinv \cov (Y, X) \, .
\end{align*}
By  $X_i = x'_i\circ X$ and \eqref{eq:covs-of-compositions-new} we then find 
\begin{align*}
\cov(X_i,Y)(\cov Y)\mpinv \cov(Y,X_i) 
= x'_i \circ  \corXY  \circ  (x'_i)' \, .
\end{align*}
Moreover, Lemma  \ref{lem:double-dual} shows that the operator $x'_i \circ  \corXY  \circ  (x'_i)':\R'\to \R$ is described by 
the $1\times 1$-matrix $\dualpair {x_i'} {\corXY x_i'}E$, while the definition of the covariance matrices $\KXiY$,  $\KY$, and $\KYXi$ 
shows that 
$\cov(X_i,Y)(\cov Y)\mpinv \cov(Y,X_i):\R'\to \R$ is described by the $1\times 1$-matrix $\KXiY \KY\mpinv \KYXi$.
Consequently, we have 
\begin{align}\label{thm:grv-conditioning-y-fin-dim-cov-half-new-first}
\KXiY \KY\mpinv \KYXi = \dualpair {x_i'} {\corXY x_i'}E \, .
\end{align}
In addition, \eqref{eq:weak-cov-pos} shows 
\begin{align*}
\KXi   = \var(X_i) = \var(\dualpair{x_i'}XE) = \dualpair {x_i'} {\cov (X) x_i'}E \, .
\end{align*}
Combining both considerations, we thus find 
\begin{align}\label{thm:grv-conditioning-y-fin-dim-cov-half-new}
\Kupi  
= \KXi -\KXiY \KY\mpinv \KYXi 
= \dualpair {x_i'} {\cov (X) x_i'}E -  \dualpair {x_i'} {\corXY x_i'}E \, .
\end{align}

%
Our next goal is to describe the relationship between 
$\P_{X|Y}(\mycdot|\mycdot)$ and the just considered $\P_{X_i|Y}(\mycdot|\mycdot)$ with the help of 
Theorem \ref{thm:rcp-under-trafo}, which describes regular conditional probabilities under transformations.
To this end, we set $T_0 := E$ and $T_1 := \R$, as well as  $U_0 := U_1 := \R^n$ and $\sA_0 := \sA_1 := \sborel^n$.
Moreover, we consider the maps $\Phi := x_i'$ and $\Psi := \id_{\R^n}$ as well as the measures
  $\P\hochkl 0 := \P_{(X,Y)}$ and $\P\hochkl 1 :=   \P_{(X_i,Y)} = \P\hochkl 1_{(\Phi,\Psi)}$.
By Theorem \ref{thm:rcp-under-trafo} there then exists
 an $N_i''\in \sborel^n$ with 
$\P_Y(N_i'') = \P_{\R^n}\hochkl 0 (N_i'')=  0$ such that for all $y\in \R^n\setminus N_i''$ we have
\begin{align*}
\P_{X|Y}\bigl( (x_i')^{-1}(B)|y\bigr) =    \P_{X_i|Y}( B|y)    \, , \myqquad B\in \sborel.
\end{align*}  
Consequently,  we find
\begin{align}\label{thm:grv-conditioning-h1}
\bigl(\P_{X|Y}( \mycdot  |y)\bigr)_{x'_i} = \P_{X_i|Y}(\mycdot |y) = \gaussm {\muupi(y)}{\Kupi}
\end{align}
for all   $y\in \R^n\setminus (N_i' \cup N_i'')$.
We define 
\begin{align*}
N :=  \bigcup_{i\geq 1} (N_i' \cup N_i'')\, .
\end{align*}
Obviously, this gives $N\in \sborel^n$ with $\P_Y(N) = 0$, and by construction 
\eqref{thm:grv-conditioning-h1} holds for all  $y\in \R^n\setminus N$ and  $x'_i\in \denseblo$.

Let us now fix a $y\in \R^n\setminus N$. Then we  have just seen in \eqref{thm:grv-conditioning-h1} that 
for all $x_i'\in \denseblo$ the 
image measure
$(\P_{X|Y}( \mycdot  |y))_{x'_i}$ of $\P_{X|Y}( \mycdot  |y)$ under $x_i'$ is
Gaussian, and by 
 Theorem \ref{thm:test-for-gms} we can thus conclude that $\P_{X|Y}( \mycdot  |y)$ is Gaussian.

Next we verify the formulas for the expectation and covariance.
Let us first consider the expectation:
For all $y\in \R^n\setminus N$ and  $x'_i\in\denseblo$, the identity \eqref{thm:grv-conditioning-h1}  shows 
\begin{align} \nonumber
\muupi(y)
=
\E \bigl(\bigl(\P_{X|Y}( \mycdot  |y)\bigr)_{x'_i}\bigr)
=
\int_\R \id_\R  \intd \bigl(\P_{X|Y}( \mycdot  |y)\bigr)_{x'_i}  
&=
\int_E (\id_\R \circ x_i') (x) \, \P_{X|Y}( \ind x  |y) \\ \nonumber
&=
\int_E \dualpair {x_i'} xE \, \P_{X|Y}( \ind x  |y) \\ \label{thm:grv-conditioning-y-fin-dim-exp-final}
&= 
\dualpairb {x_i'}{\muup(y)}E \, ,
\end{align}
where in the third step we used the transformation formula for image measures, and in the last step 
we used that $\id_E$ is Bochner-integrable with respect to $\P_{X|Y}( \mycdot  |y)$ since 
$\P_{X|Y}( \mycdot  |y)$ is a Gaussian measure, and thus \eqref{eq:Fernique's-theorem} holds true.
Using \eqref{thm:grv-conditioning-y-fin-dim-exp} we thus find 
\begin{align*}
\dualpairb {x_i'}{\muup(y)}E =  \dualpairb {x'_i} {\E X + \cov(X, Y)  \cov(Y)\mpinv (y- \E Y)}E  
\end{align*}
for all $x_i'\in \denseblo$. Since $\denseblo$ is $\tauweaks$-dense in $B_{E'}$, a simple limit argument then
shows that the latter identity  not only holds for all $x_i'$ but actually for all $x'\in B_{E'}$.
By Hahn-Banach's theorem we thus find 
\begin{align*}
\muup(y) = \E X + \cov(X, Y)  \cov(Y)\mpinv (y- \E Y)\, , \myqquad y\in \R^n\setminus N.
\end{align*}

To verify the formula for the covariance, we again fix an $y\in \R^n\setminus N$. For $x_i'\in \denseblo$,
the Identity \eqref{thm:grv-conditioning-h1} and the fact that $\bigl(\P_{X|Y}( \mycdot  |y)\bigr)_{x'_i}$ is a distribution on $\R$   then yield
\begin{align*}
\Kupi
= \var \Bigl(  \bigl(\P_{X|Y}( \mycdot  |y)\bigr)_{x'_i} \Bigr) 
&= \int_\R \bigl( \id_\R(t) - \muupi(y)  \bigr)^2 \intd \bigl(\P_{X|Y}( \mycdot  |y)\bigr)_{x'_i}(t) \\
&= \int_E \bigl(\id_\R \circ x_i' - \muupi(y) \bigr)^2 (x) \, \P_{X|Y}( \ind x  |y) \\
&= \int_E \bigl(\dualpair {x_i'} xE - \muupi(y) \bigr)^2   \, \P_{X|Y}( \ind x  |y) \\
&= \int_E \dualpairb {x_i'}  {x - \muup(y)}E^2   \, \P_{X|Y}( \ind x  |y) \\
& = \dualpairb {x_i'}{\cov \bigl(\P_{X|Y}( \mycdot  |y)\bigr) x_i' }E \, ,
\end{align*}
where we have also used a projection formula for covariances \eqref{eq:weak-cov-pos}, the basic transformation formula for $\R$-valued integrals, 
 \eqref{thm:grv-conditioning-y-fin-dim-exp-final}, and \eqref{eq:weak-cross-cov}.
Combining this with \eqref{thm:grv-conditioning-y-fin-dim-cov-half-new} we thus find
\begin{align*}
\dualpairb {x_i'}{\cov \bigl(\P_{X|Y}( \mycdot  |y)\bigr) x_i' }E
=
\dualpair {x_i'} {\cov (X) x_i'}E - \dualpairb {x_i'} {\corXY x_i'}E
\end{align*}
for all $x_i'\in \denseblo$ and $y\in \R^n\setminus N$.
Now the left hand side of this identity is $\tauweaks$-continuous on $B_{E'}$ by
Theorem \ref{thm:continuous-covariance-kernel}.
Moreover, $(\cov Y)\mpinv:\R^n\to (\R^n)'$ is symmetric and non-negative, see the discussion around \eqref{eq:psd-of-mpinv}.
Consequently, 
 the right-hand side  
is also $\tauweaks$-continuous on $B_{E'}$ by 
Theorem \ref{thm:continuous-covariance-kernel} and Lemma \ref{lem:cov-cross-co}. Since 
$\denseblo$ is $\tauweaks$-dense in $B_{E'}$, we thus find 
\begin{align*}
\dualpairb {x'}{\cov \bigl(\P_{X|Y}( \mycdot  |y)\bigr) x' }E
&=
\dualpair {x'} {\cov (X) x'}E - \dualpairb {x'} {\corXY x'}E \, .
\end{align*}
for all $x'\in B_{E'}$. By scaling this implies
\begin{align}\label{thm:grv-conditioning-y-fin-dim-hhh}
\dualpairb {x'}{\bigl(  \cov (\P_{X|Y}( \mycdot  |y) ) + \corXY\bigr)  x' }E
=
\dualpair {x'} {\cov (X) x'}E
\end{align}
for all $x'\in E'$. Moreover, since $(\cov Y)\mpinv:\R^n\to (\R^n)'$ is symmetric and non-negative, Lemma \ref{lem:cov-cross-co} shows 
that $\corXY$ is an abstract covariance operator, and this implies that
%
\begin{align*}
\cov (\P_{X|Y}( \mycdot  |y))  + \corXY :E'\to E
\end{align*}
is an abstract covariance operator.
In addition, $\cov (X):E'\to E$ is also an abstract covariance operator, and \eqref{thm:grv-conditioning-y-fin-dim-hhh} shows that
both abstract covariance operators coincide on the diagonal. Lemma \ref{lem:abstr-cov-on-diag}
thus yields $\cov (\P_{X|Y}( \mycdot  |y))  + \corXY = \cov (X)$, that is, we have found 
the formula for the covariance.

Let us now  prove the claims made for $Z$. 
To this end, we write $\coordproj E:E\times \R^n\to E$ and $\coordproj {\R^n}:E\times \R^n\to \R^n$ for  the coordinate 
projections onto $E$ and $\R^n$.
Moreover,
we define $N_0:= Y^{-1}(N)$, which
yields $N_0\in \sA$ with $\P(N_0) = 0$.
For $\om \in \Om\setminus N_0$ and $y:= Y(\om)$ we then have $y\in \R^n\setminus N$ and therefore
$\muup(y)$ is the mean of $\P_{X|Y}(\mycdot|y)$. 
Moreover, since $\P_{X|Y}(\mycdot|y)$ is Gaussian, we have $\id_E \in \sLx 1 {\P_{X|Y}(\mycdot|y),E}$.
For $g:= \id_E$
this gives $g\circ \coordproj E (\mycdot ,y) = g \in \sLx 1 {\P_{X|Y}(\mycdot|y),E}$ for all $y\in \R^n\setminus N$,
and therefore,  the random variable $Z_g:\R^n\to E$ defined by
\begin{align*}
Z_g(y)
:=
 \int_E \eins_{\R^n\setminus N} (y) \cdot  \bigl(  (g\circ \coordproj E) (x,y) \bigr) \intd \P_{X|Y}(\intd x|y) \, , \myqquad y\in \R^n,
\end{align*}
provides
  a version $Z_g\circ \coordproj {\R^n}$  of $\E( g\circ \coordproj E | \coordproj {\R^n}^{-1}(\sborel^n)) = \E( g\circ \coordproj E | E\times \sborel^n)$ 
  by Theorem \ref{thm:disintegration-bs} applied to the probability measure
$\P_{(X,Y)}$ on $\sborelnormx E \otimes \sborel^n = \sborelEx E \otimes \sborel^n$.
Hence, we have
\begin{align}\label{thm:grv-conditioning-y-fin-dim-h00}
 \int_{E\times \R^n}  \eins_{E\times A} \cdot  (Z_g\circ \coordproj {\R^n})  \intd \P_{(X,Y)}
 =
\int_{E\times \R^n}  \eins_{E\times A} \cdot (g\circ \coordproj E) \intd \P_{(X,Y)}
\end{align}
for all $A\in \sborel^n$.
Moreover, for $\om \in \Om\setminus N_0$ we find 
\begin{align*}
Z(\om) 
=
\muup(Y(\om))
&= \int_E x\,  \P_{X|Y}(\intd x|Y(\om)) \\
&= \int_E \eins_{\R^n\setminus N} (Y(\om)) \cdot \bigl(( g\circ \coordproj E) (x,Y(\om))\bigr) \intd \P_{X|Y}(\intd x|Y(\om)) \\
&= Z_g(Y(\om))\, ,
\end{align*}
which shows $\P(\{Z\neq Z_g\circ Y\}) = 0$. 
Since $Z$ is $\s(Y)$-measurable by construction, it thus 
suffices to show that
$Z_g\circ Y$ is a version of
 $\E(X|Y)$. Here the $\s(Y)$-measurability of $Z_g\circ Y$ is ensured by construction.
Moreover, for $A_0\in \s(Y)$ there exists an $A\in \sborel^n$ with $A_0 = Y^{-1}(A)$, and with this we obtain
\begin{align*}
\int_\Om \!\eins_{A_0} \cdot  (Z_g\circ Y) \intd \P
=
\int_\Om \!(\eins_{A}  \circ Y) \cdot  (Z_g\circ Y) \intd \P
&=
\int_{\R^n} \!\eins_A  Z_g \intd \P_Y \\
&=
\int_{E\times \R^n} \!\! (\eins_A \circ \coordproj {\R^n} ) \cdot (Z_g\circ \coordproj {\R^n}) \intd \P_{(X,Y)} \\
&=
\int_{E\times \R^n}  \eins_{E\times A}   \cdot (Z_g\circ \coordproj {\R^n}) \intd \P_{(X,Y)}
\end{align*}
as well as
\begin{align*}
\int_\Om \eins_{A_0} \cdot X\intd \P
=
\int_\Om (\eins_{A}  \circ Y) \cdot  X \intd \P
&=
\int_\Om \bigl(\eins_{A} \circ \coordproj {\R^n}(X,Y) \bigr)\cdot \bigl( g\circ \coordproj E(X,Y)  \bigr) \intd \P \\
&=
\int_{E\times \R^n}  \eins_{E\times A} \cdot  (g\circ \coordproj E) \intd \P_{(X,Y)} \, .
\end{align*}
Combining the latter two calculations with \eqref{thm:grv-conditioning-y-fin-dim-h00} we obtain
\begin{align*}
\int_\Om \eins_{A_0} \cdot  (Z_g\circ Y) \intd \P = \int_\Om \eins_{A_0} \cdot X\intd \P
\end{align*}
for all $A_0\in \s(Y)$. This shows that $Z_g\circ Y$ is a version of
 $\E(X|Y)$ and hence so is $Z$.
Moreover,  the map $T:\R^n\to E$ defined by
\begin{align}\label{thm:grv-conditioning-y-fin-dim-defT}
Ty :=  \E X + \cov ( X,Y) (\cov Y)\mpinv(y- \E Y)\, , \myqquad y\in \R^n,
\end{align}
 is affine linear and continuous
with $Z = T\circ Y$. Since $Y$ is Gaussian random variable, we then see that $Z$ is also a Gaussian random variable.
In addition, we have $\E Z = \E (\E(X|Y)) = \E X$ by Theorem \ref{thm:prop-cond-ex}.

To verify the formula for $\cov (Z)$, we first consider the case $E= \R$, in which we only need to determine $\var Z$.
Since we already know $\E Z  = \E X$, we then obtain
\begin{align}\label{thm:grv-conditioning-y-fin-dim-h0000}
\var Z = \E Z^2 - (\E Z)^2 = \E Z^2 - (\E X)^2\, .
\end{align}
Moreover, since for $y\in \R^n\setminus N$ we have $\muup(y) = Ty$ we find
\begin{align}\label{thm:grv-conditioning-y-fin-dim-h000}
\E Z^2
= \E \bigl( T\circ Y)^2
= \int_{\R^n} (Ty)^2 \intd \P_Y (y)
= \int_{\R^n} (\muup(y))^2 \intd \P_Y (y)  
\end{align}
by the definition \eqref{thm:grv-conditioning-y-fin-dim-defT} of $T$ and the already established formula for $\muup(y)$.
Since $E=\R$ we further have $\Kup = \var (\P_{X|Y}(\mycdot|y))$ for all $y\in \R^n\setminus N$ and combining this with
\eqref{thm:grv-conditioning-y-fin-dim-h000} and \eqref{thm:grv-conditioning-y-fin-dim-h0000} leads to
\begin{align*}
\Kup
= \int_{\R^n} \Kup \intd \P_Y(y)
&= \int_{\R^n} \biggl(  \int_\R x^2  \, \P_{X|Y}(\ind x|y) - \bigl(\muup(y)\bigr)^2   \biggr) \intd \P_Y (y) \\
&=  \int_{\R^n}   \int_\R x^2  \, \P_{X|Y}(\ind x|y)   \intd \P_Y (y) - \E Z^2 \\
&= \int_{\R\times \R^n} (\coordproj E)^2 \intd \P_{(X,Y)} -  \var Z -  (\E X)^2 \\
&=  \int_{\R }x^2 \intd \P_{X} -   \var Z -  (\E X)^2  \\
& = \var X  - \var Z\, ,
\end{align*}
where in the third to last step we used the disintegration formula \eqref{eq:disintegration}. 
In other words, we have found the desired covariance formula in the case $E=\R$.

To verify the formula for $\cov (Z)$ in the case of general $E$, we pick an $x'\in E'$
and define $X_0 := \dualpair{x'} XE$ as well as
\begin{align*}
Z_0:=  \E X_0 + \cov ( X_0,Y) (\cov Y)\mpinv(Y- \E Y) \, .
\end{align*}
Using \eqref{eq:covs-of-compositions-new} this definition can also be expressed as
\begin{align*}
\dualpair {x'}ZE
&= \dualpairb {x'}{ \E X + \cov ( X,Y) (\cov Y)\mpinv(Y- \E Y)}E \\
&= \E X_0 + \bigl(x' \circ \cov ( X,Y) \circ (\cov Y)\mpinv \bigr) (Y- \E Y) \\
&= \E X_0 + \cov ( X_0,Y) (\cov Y)\mpinv(Y- \E Y) \\
&= Z_0 \, .
\end{align*}
Since $X_0$ and $Y$ are jointly Gaussian random variables, the already treated
case $E=\R$ thus shows
\begin{align*}
\dualpair {x'} {\cov (Z) x'}E
=
\var(\dualpair{x'}ZE )
=
\var Z_0 
= \var (X_0) -  \KupO
&= \KXOY \KY\mpinv \KYXO  \\
&= \dualpair {x'} {\corXY x'}E
\, ,
\end{align*}
where in first step we used \eqref{eq:weak-cov-pos} and in 
the last step we repeated the calculations of \eqref{thm:grv-conditioning-y-fin-dim-cov-half-new-first}.
Since  $\cov (Z)$ is an abstract
covariance operator and we have seen around \eqref{thm:grv-conditioning-y-fin-dim-hhh} that  $\corXY$ is also an abstract
covariance operator,  Lemma \ref{lem:abstr-cov-on-diag}
 then  yields $\cov (Z) =  \corXY$.

Finally, for the proof of $\gaussspace Z \subset \gaussspace Y$, we consider the bounded linear operator  $A:=\cov ( X,Y) (\cov Y)\mpinv:\R^n\to E$.
The definition of $Z$ then gives
\begin{align*}
Z-\E Z = Z-\E X = \cov ( X,Y) (\cov Y)\mpinv(Y- \E Y) = A(Y- \E Y) \, .
\end{align*}
By Lemma \ref{lem:gaussspace-of-grv} we thus conclude $\gaussspace Z = \gaussspace{Z-\E Z} \subset  \gaussspace{Y-\E Y} = \gaussspace Y$.
%
\end{proof}

\subsection{Proof of Theorem \ref{thm:grv-conditioning-y-inf-dim}}

Before we can prove Theorem \ref{thm:grv-conditioning-y-inf-dim}, we again need some 
additional machinery.  


 We begin with a result showing that filtering sequences induce filtrations in a general setup.

\begin{lemma}\label{lem:filt-seq-gives-filtration}
Let $F$ be a separable Banach space 
and $(\nfilts n)$ be a filtering sequence for $F$. 
Moreover, let $(\Om, \sA, \P)$ be a probability space, $Y:\Om\to F$ be $\sborelEx F$- measurable, and $Y_n := \nfilts n\circ Y$. Then
$(\s(Y_n) )$ is a filtration in $\sA$ and
we have
\begin{align*}
\s(Y) = \s\biggl( \bigcup_{n\geq 1} \s(Y_n)   \biggr)\, .
\end{align*}
\end{lemma}

\begin{proof}[Proof of Lemma \ref{lem:filt-seq-gives-filtration}]
  We first note that we have
\begin{align*}
\s(Y_n) = \s(\nfilts n \circ Y) = (\nfilts n \circ Y)^{-1}(\sborel^n) = Y^{-1}(\nfilts n^{-1}(\sborel^n))
= Y^{-1}(\s(\nfilts n))\, ,
\end{align*}
where the second and the last step uses an elementary formula for  the $\s$-algebra generated by a single map,
see e.g.~\cite[Theorem 1.78]{Klenke14}. Since $(\s(\nfilts n))$ is a filtration, so is $(\s(Y_n))$.

To establish the formula for $\s(Y)$, we write
  $\ca E:= \bigcup_{n\geq 1} \s(\nfilts n)$. By assumption we know that
 $\s(\ca E ) = \sborelEx F$,  and hence we  obtain
\begin{align*}
\s(Y)
= Y^{-1}(\s(\ca E)) 
= \s(Y^{-1}(\ca E)) 
= \s  \biggl( \bigcup_{n\geq 1} Y^{-1}(\s(\nfilts n))  \biggr)
= \s  \biggl( \bigcup_{n\geq 1} \s(Y_n)  \biggr) \, ,
\end{align*}
where in the second step we used another well-known fact on generated $\s$-algebras, see e.g.~\cite[Theorem 1.81]{Klenke14} and its proof.
\end{proof}

The next result  translates martingale convergence to 
convergence of integrals taken with respect to regular conditional probabilities. Readers unfamiliar with the notion of uniform tightness are referred to 
Supplement \ref{app:cf}.

\begin{theorem}\label{thm:limit-conditioning}
Let $(\Om,\sA,\P)$ be a probability space, $E$ and $F$ be separable Banach spaces, 
and $X:\Om\to E$ and $Y:\Om\to F$ be random variables. 
Moreover, 
let $(\nfilts n)$ be a filtering sequence for $F$ and 
\begin{align*}
Y_n := \nfilts n\circ Y \, , \myqquad n\geq 1.
\end{align*}
Then for all versions  $\P_{X|Y_n}(\mycdot|\mycdot) :\sborelnormx E\times \R^n\to[0,1]$ of the regular conditional probabilities of $X$ given $Y_n$
there 
exists an $N\in \sborelEx F$ with $\P_Y(N) = 0$ such that  
\begin{align*}
\bigl\{\P_{X|Y_n}(\mycdot| \nfilts n(y)): n\geq 1\bigr\}
\end{align*}
 is uniformly tight for all $y\in F\setminus N$.

In addition, for all Banach spaces $G$ and 
all $\PX$-Bochner integrable $g:E\to G$ there exists an $N_g\in \sborelEx F$ with $\P_Y(N_g) = 0$ such that for all $y\in F\setminus N_g$ and $n\geq 1$ we have 
\begin{align}\label{thm:limit-conditioning-int}
g\in \sLx 1 {\P_{X|Y}(\mycdot|y),G} \myqquad \mbox{ and } \myqquad 
g\in \sLx 1 {\P_{X|Y_n}(\mycdot| \nfilts n(y)),G} 
\end{align}
as well as
\begin{align}\label{thm:limit-conditioning-cong}
\int_E  g(x) \, \P_{X|Y_n}(\ind x| \nfilts n(y)) \to  \int_E  g(x) \, \P_{X|Y}(\ind x|y) \, ,
\end{align}
where the convergence is with respect to $\snorm \cdot_G$.
\end{theorem}

\begin{proof}[Proof of Theorem \ref{thm:limit-conditioning}]
We begin with some preparations needed for both assertions. 
To this end, we write 
$\sA_n := \s(\nfilts n) \subset \s(F')$
and $\sA_\infty := \s(F')$. We further define $\nfilts \infty := \id_F$ and $Y_\infty := Y$, which gives $Y_\infty = \nfilts \infty \circ Y$. 
Since $(\nfilts n)$ is a filtering sequence we then have 
\begin{align*}
 \s(\nfilts \infty) = \sA_\infty =  \s\bigl( \nfilts n: n\geq 1 \bigr)\, .
\end{align*}
In addition, we denote 
the restriction of $\P_{(X,Y)}$ onto $\sborelnorm E\otimes \sA_n$ by $\Qm\hochkl n$, that is
\begin{align*}
\Qm\hochkl n := \bigl( \P_{(X,Y)}\bigr)_{|\sborelnorm E\otimes \sA_n} 
\end{align*}
for all $n\in \N\cup\{\infty\}$. For $n=\infty$, this gives $\Qm\hochkl \infty = \P_{(X,Y)}$ and for $n\in \N$, the identity
 $Y_n = \nfilts n\circ Y$ implies 
\begin{align*}
\P_{(X,Y_n)} = \bigl(  \P_{(X,Y)} \bigr)_{(\id_E, \nfilts n)} = \Qm\hochkl n_{(\id_E, \nfilts n)} \, .
\end{align*}
With the help of Theorem \ref{thm:ex-ei-reg-cond-prob}
we now fix a regular conditional probability $\P_{X|Y}(\mycdot|\mycdot)$ and write
 $\Qm\hochkl \infty (\mycdot|\mycdot):=\P_{X|Y}(\mycdot|\mycdot)$.
 Moreover, for $n\in \N$, 
we analogously fix a regular conditional probability
\begin{align*}
\Qm\hochkl n(\mycdot|\mycdot):\sborelnorm E\times F\to [0,1]
\end{align*}
 of $\Qm\hochkl n$.
Theorem \ref{thm:rcp-under-trafo} applied to $\Phi:= \id_E:E\to E$ and $\Psi:= \nfilts n:F\to \R^n$ and the probability measures $\P\hochkl 0 := \Qm\hochkl n$ and 
$\P\hochkl 1 := \P_{(X,Y_n)}$ 
then gives an $N\hochkl 1_n\in \sA_n$
with $\Qm\hochkl n_F(N\hochkl 1_n)) = 0$ such that for all $y\in F\setminus N\hochkl 1_n$ and all $B\in \sborelnorm E$
we have
\begin{align}\label{thm:limit-conditioning-h1}
\Qm\hochkl n(B|y) = \P_{X|Y_n}\bigl(B| \nfilts n(y)\bigr) \, .
\end{align}
Here we note that in the case $n=\infty$  our construction gives $\Qm\hochkl \infty = \P_{(X,Y)}$ as noted above 
and therefore     \eqref{thm:limit-conditioning-h1} holds in the case $n=\infty$  for all $y\in F$.
Let us now define
\begin{align*}
N_{\mathrm{trafo}} := \bigcup_{n\geq 1} N\hochkl 1_n \, .
\end{align*}
Since for all $n\geq 1$ we have $N\hochkl 1_n\in \sA_n\subset \sA_\infty = \s(F')$ with 
\begin{align}\label{thm:limit-conditioning-h666}
\P_Y(N_n\hochkl 1) = \P_{(X,Y)}(E\times N_n\hochkl 1) = \Qm\hochkl n (E\times N_n\hochkl 1)  = \Qm_F\hochkl n ( N_n\hochkl 1) = 0\, ,
\end{align}
we then find $N_{\mathrm{trafo}}\in   \sborelEx F$ with
$\P_Y(N_{\mathrm{trafo}}) = 0$. In addition, 
 \eqref{thm:limit-conditioning-h1} holds for all $y\in F\setminus N_{\mathrm{trafo}}$ and all $n\in \N\cup\{\infty\}$ by our construction.
Finally, our construction also ensures $\Qm\hochkl n_E = \P_{X}$ for all $n\in \N\cup\{\infty\}$.

 Our first goal is to establish \eqref{thm:limit-conditioning-int} and  \eqref{thm:limit-conditioning-cong}. To this end,
we fix a Banach space $G$ and 
a $\PX$-Bochner integrable $g:E\to G$. For $ n\in \N\cup\{\infty\}$
our construction then ensures 
\begin{align}\label{thm:martingale-convergence-int-cond}
\int_{E\times F} \snorm{ g\circ \coordproj E}_{G} \intd \Qm\hochkl n
= \int_{E} \snorm{  g}_{G}  \intd \Qm\hochkl n_E
= \int_{E} \snorm{  g}_{G}  \intd \P_{X}
< \infty\, . 
\end{align}
In other words, we have 
 $g\circ \coordproj E\in  \sLx 1 {\Qm\hochkl  n,G}$ for  all $n\in \N\cup\{\infty\}$.
 In the next step we will 
 apply the martingale convergence  Theorem \ref{thm:martingale-convergence}. To this end, 
we define the $\s$-algebras
\begin{align*}
\siC_n:= \coordproj F^{-1} (\sA_n) = \{E\times A : A \in \sA_n    \}\, , \myqquad  n\geq 1
\end{align*}
and we further write
\begin{align*}
\siC_\infty := \coordproj F^{-1} (\sborelEx F) = \{E\times A : A \in  \sborelEx F   \} =  \{E\times A : A \in \sA_\infty  \}\, .
\end{align*}
Since
$(\sA_n)$ is a filtration
 with $\sA_\infty = \s(\sA_n: n\geq 1)$, we quickly check 
that
$(\siC_n)$ is a filtration
 with $\siC_\infty = \s(\siC_n: n\geq 1)$.
For every fixed choice of versions $g_n$ of  $\E_{\P_{(X,Y)}} (g\circ \coordproj E |\siC_n)$
and $g_\infty$ of $\E_{\P_{(X,Y)}} (g\circ \coordproj E |\siC_\infty)$
Theorem \ref{thm:martingale-convergence}
applied to the probability space $(E\times F, \sborelnorm E\otimes \sborelEx F, \P_{(X,Y)})$
then gives an $N'\in \siC_\infty$ with $\P_{(X,Y)}(N') = 0$ and
\begin{align}\label{thm:limit-conditioning-hxx}
\snorm{g_n (x,y) - g_\infty (x,y)}_G\to 0   \, , \myqquad (x,y) \in (E\times F)\setminus N'\, .
\end{align}
Moreover,  note that $N'$ is of the form $N'= E\times N_{\mathrm{mart}}$ with $N_{\mathrm{mart}}\in \sborelEx F$ since $N'\in\siC_\infty$. In addition, we have
$\P_Y(N_{\mathrm{mart}}) = \P_{(X,Y)}(N') = 0$.

Let us now translate this convergence to regular conditional probabilities. To this end, we apply
Theorem \ref{thm:disintegration-bs} to $g\circ \coordproj E$ and the measures $\Qm\hochkl n$ defined on $\sborelnorm E\otimes \sA_n$:
Because of \eqref{thm:martingale-convergence-int-cond} we then obtain
an $N\hochkl 2_{n}\in \sA_n$ with $\Qm\hochkl n_F (N\hochkl 2_{n}) = 0$ and
\begin{align}\label{thm:limit-conditioning-h9999}
g = g\circ \coordproj E(\mycdot ,y) \in \sLx 1 {\Qm\hochkl n(\mycdot| y), G}  \, , \myqquad y\in F\setminus N\hochkl 2_n,\, n\in \N\cup\{\infty\}
\end{align}
such that 
$Z_n:F\to G$ defined by
\begin{align*}
Z_n(y)
:=
\int_E \eins_{F\setminus N\hochkl 2_{n}}(y) \cdot (g \circ \coordproj E)  (x,y) \, \Qm\hochkl n(\ind x|y)
=
\int_E \eins_{F\setminus N\hochkl 2_{n}}(y) \cdot  g(x) \, \Qm\hochkl n(\ind x|y)
\end{align*}
for all $y\in F$
is strongly  measurable, the disintegration formula \eqref{eq:disintegration} holds true, and
$Z_n\circ \coordproj F:E\times F\to G$ is a version of
$\E_{\Qm\hochkl n}(g\circ \coordproj E|\siC_n)$, that is, we have
\begin{align}\label{thm:martingale-convergence-conde-trans-1}
\int_{E\times F} \eins_C \cdot(Z_n\circ \coordproj F) \intd \Qm\hochkl n
=
\int_{E\times F} \eins_C \cdot(g\circ \coordproj E) \intd \Qm\hochkl n
\end{align}
for all $C\in \siC_n$.
Our next intermediate goal is to verify that
$Z_n\circ \coordproj F$ is also a version of $\E_{\P_{(X,Y)}} (g\circ \coordproj E |\siC_n)$.
Here, the strong $\siC_n$-measurability is obviously satisfied since $Z_n\circ \coordproj F$ is a version of $\E_{\Qm\hochkl n}(g\circ \coordproj E|\siC_n)$.
Moreover, for $C\in \siC_n$ we have
\begin{align}\label{thm:martingale-convergence-conde-trans-2}
\int_{E\times F} \eins_C \cdot(Z_n\circ \coordproj F) \intd \P_{(X,Y)}
=
\int_{E\times F} \eins_C \cdot(Z_n\circ \coordproj F) \intd \Qm\hochkl n\, ,
\end{align}
where  we used that $\eins_C \cdot(Z_n\circ \coordproj F)$ is strongly $\siC_n$-measurable and thus also strong
$\sborelnorm E\otimes \sA_n$-measurable
and that  $\Qm\hochkl n$ is the restriction of $\P_{(X,Y)}$ onto $\sborelnorm E\otimes \sA_n$.
In addition, $g\circ  \coordproj E$ is strongly $\{B\times F: B\in \sborelnorm E\}$-measurable
and thus also strongly $\sborelnorm E\otimes \sA_n$-measurable.
Consequently,
$\eins_C \cdot(g\circ \coordproj E)$ is also strongly $\sborelnorm E\otimes \sA_n$-measurable, and since
$\Qm\hochkl n$ is the restriction of $\P_{(X,Y)}$ onto $\sborelnorm E\otimes \sA_n$ we find
\begin{align}\label{thm:martingale-convergence-conde-trans-3}
\int_{E\times F} \eins_C \cdot(g\circ \coordproj E) \intd \Qm\hochkl n
=
\int_{E\times F} \eins_C \cdot(g\circ \coordproj E) \intd  \P_{(X,Y)}\, .
\end{align}
Combining \eqref{thm:martingale-convergence-conde-trans-2} with  \eqref{thm:martingale-convergence-conde-trans-1}   and
\eqref{thm:martingale-convergence-conde-trans-3}, we now find
\begin{align*}
\int_{E\times F} \eins_C \cdot(Z_n\circ \coordproj F) \intd \P_{(X,Y)}
=
\int_{E\times F} \eins_C \cdot(g\circ \coordproj E) \intd  \P_{(X,Y)} 
\end{align*}
for all $C\in \siC_n$. In other words, $Z_n\circ \coordproj F$ is indeed a version of $\E_{\P_{(X,Y)}} (g\circ \coordproj E |\siC_n)$.
For all $x\in E$ and 
$y\in F\setminus N_{\mathrm{mart}}$ the martingale convergence \eqref{thm:limit-conditioning-hxx} thus shows
\begin{align*}
Z_n(y) = Z_n\circ \coordproj F(x,y)  \to  Z_\infty\circ \coordproj F(x,y) =     Z_\infty(y) \, .
\end{align*}
Let us define $N_g := N_{\mathrm{trafo}} \cup N_{\mathrm{mart}}\cup N\hochkl 2_\infty \cup \bigcup_{n\geq 1} N\hochkl 2_n$.
Clearly, this definition gives $N_g\in \sborelEx F$ with $\P_Y(N_g) = 0$, see the calculation \eqref{thm:limit-conditioning-h666}.
For $y\in F\setminus N_g$ our considerations \eqref{thm:limit-conditioning-h9999}   and  \eqref{thm:limit-conditioning-h1}   further give
\begin{align*}
g  \in \sLx 1 {\Qm\hochkl n(\mycdot| y), G} = \sLx 1 {\P_{X|Y_n}(\mycdot| \nfilts n(y)),G}
\end{align*}
for all $n\in \N\cup\{\infty\}$ as well as
\begin{align*} 
\int_E \! g(x) \, \P_{X|Y_n}\bigl(\ind x| \nfilts n(y)\bigr)
= \int_E \! g(x) \, \Qm\hochkl n(\ind x|y)
= Z_n(y)
\to
Z_\infty(y)
&=
\int_E \! g(x) \, \Qm\hochkl \infty(\ind x|y) \\
&= \int_E \! g(x) \, \P_{X|Y}(\ind x|y)
\end{align*}
for $n\to \infty$, where  the convergence is with respect to the norm $\snorm\mycdot_G$.
In other words we have shown both \eqref{thm:limit-conditioning-int} and  \eqref{thm:limit-conditioning-cong}.

Let us finally establish the uniform tightness. To this end, we first note that $\{\PX\}$ is uniformly tight
as discussed around \eqref{def:radon-by-inner-reg},
and hence Lemma \ref{lem:uni-tight}
gives a $\sborelnorm E$-measurable $V:E\to [0,\infty]$ such that $\{V\leq c\}$ is compact for all $c\in [0,\infty)$ and $V\in \sLx 1 \PX$.
Consequently,   $N_{\mathrm{tight}}:= \{V=\infty \}\in \sborelnormx E$ satisfies $\PX(N_{\mathrm{tight}}) = 0$.
We write $G:= \R$ and $g:= \eins_{E\setminus N_{\mathrm{tight}}} V$, where we use the convention $0\cdot \infty := 0$.
Then we have $g\in \sLx 1 {\PX, \R}$, and therefore we already know that there 
exists an $N_g\in \sborelEx F$ with $\P_Y(N_g) = 0$ such that for all $y\in F\setminus N_g$ we have 
$g\in \sLx 1 {\P_{X|Y}(\mycdot|y),\R}$ and both \eqref{thm:limit-conditioning-int} and \eqref{thm:limit-conditioning-cong} hold true.
Consequently, there exists an $n_0\geq 1$ such that
\begin{align*}
\int_E  g(x) \, \P_{X|Y_n}(\ind x| \nfilts n(y))  \leq 1 +   \int_E  g(x) \, \P_{X|Y}(\ind x|y)  
\end{align*}
for all $n\geq n_0$ and all $y\in F\setminus N_g$. By applying \eqref{thm:limit-conditioning-int}  we conclude that
\begin{align}\label{thm:limit-conditioning-h999}
\sup_{n\geq 1} \int_E  g(x) \, \P_{X|Y_n}(\ind x| \nfilts n(y)) < \infty
\end{align}
for all $y\in F\setminus N_g$. Now note that
 $\{g\leq c\}$ may not be compact, and therefore we cannot apply
 Lemma \ref{lem:uni-tight} directly. For this reason, we
 translate \eqref{thm:limit-conditioning-h999} back to the function $V$. To this end, we note that we have
\begin{align*}
0 =  \PX(N_{\mathrm{tight}}) =   \P_{(X,Y)}(N_{\mathrm{tight}} \times F) = \Qm\hochkl n (N_{\mathrm{tight}}\times F) = \int_F \Qm\hochkl n (N_{\mathrm{tight}}| y) \intd \Qm\hochkl n _F(y)\, .
\end{align*}
Consequently, there exists an $N_n\hochkl 3 \in \sA_n$ with $\Qm\hochkl n_F (N_n\hochkl 3) = 0$ and 
\begin{align*}
\Qm\hochkl n (N_{\mathrm{tight}}| y)  = 0 \, , \myqquad y\in F\setminus N_n\hochkl 3.
\end{align*}
Repeating the calculation \eqref{thm:limit-conditioning-h666} then shows $\P_Y(N_n\hochkl 3)=0$. We define 
\begin{align*}
N:= N_g \cup \bigcup_{n\geq 1} (N_n\hochkl 1 \cup N_n\hochkl 3)\, .
\end{align*}
This gives $N\in \sborelEx F$ with $\P_Y(N) = 0$ and for $y\in F\setminus N$ we have both \eqref{thm:limit-conditioning-h999}
and 
\begin{align*}
\P_{X|Y_n}(N_{\mathrm{tight}}| \nfilts n(y)) = \Qm\hochkl n (N_{\mathrm{tight}}| y)  = 0 \, .
\end{align*}
Moreover, we have $N_{\mathrm{tight}}= \{V=\infty \} = \{V\neq g\}$ by the definition of $g$.
The previous identity thus implies
\begin{align*}
\int_E  V(x) \, \P_{X|Y_n}(\ind x| \nfilts n(y))  = \int_E  g(x) \, \P_{X|Y_n}(\ind x| \nfilts n(y))
\end{align*}
for all $y\in F\setminus N$.
Inserting this into  \eqref{thm:limit-conditioning-h999} then gives the uniform tightness by Lemma \ref{lem:uni-tight}.
\end{proof}


%
%

\begin{proof}[Proof of  Theorem \ref{thm:grv-conditioning-y-inf-dim}]
In the following, we sequentially prove all parts of Theorem \ref{thm:grv-conditioning-y-inf-dim}. To this end,
we construct for each part \emph{k)} an 
$N\hochkls k\in \sborelEx F$ with 
$\P_Y(N\hochkls k) = 0$ such that for all $y\in F\setminus N\hochkls k$ the claims of the parts \emph{i)} to \emph{k)} hold true. 
The desired $\P_Y$-zero set $N$ is then given by $N:= N\hochkls {viii}$.

\ada i Since $(\id_E,\nfilts n):E\times F\to E\times \R^n$ is a bounded and linear operator we quickly see that  
 $(X,Y_n)$ is a Gaussian random variable. For each $n\geq 1$,
Theorem \ref{thm:grv-conditioning-y-fin-dim} then gives an $N_n\hochkl 1\in \sborel^n$ with $\P_{Y_n}(N_n\hochkl 1) = 0$ such for all
$y_n\in \R^n\setminus N_n\hochkl 1$ we have
\begin{align}\label{thm:grv-conditioning-y-inf-dim-h0}
\P_{X|Y_n}(\mycdot|y_n) = \gaussm {\muupn(y_n)}{\covupn}\, ,
\end{align}
where $\muupn(y_n)\in E$ and $\covupn:E'\to E$ are described in Theorem \ref{thm:grv-conditioning-y-fin-dim}.
We define 
\begin{align*}
N\hochkls i := \bigcup_{n\geq 1} \nfilts n^{-1}(N_n\hochkl 1)\, .
\end{align*}
Note that we have $\nfilts n^{-1}(N_n\hochkl 1)\in \s(\nfilts n)\subset \sborelEx F$, and  $Y_n = \nfilts n\circ Y$ shows 
$\P_Y(\nfilts n^{-1}(N_n\hochkl 1)) = \P_{Y_n}(N_n\hochkl 1) = 0$. Consequently, we have $N\hochkls i\in \sborelEx F$ with $\P_Y(N\hochkls i) = 0$.
Moreover, for $y\in F\setminus N\hochkls i$ we have $y_n := \nfilts n(y) \in \R^n\setminus  N_n\hochkl 1$ and therefore, the claims follow from 
\eqref{thm:grv-conditioning-y-inf-dim-h0}.

\ada {ii}
We first consider the case $E=\R$. Here, we fix an enumeration $(t_m)_{m\geq 3}$ of $\Q$ and consider the functions
$g_m :  \R\to \Cfield$ defined by
\begin{align*}
g_m (x) := \exp(\imi t_m x) \, , \myqquad x\in \R
\end{align*}
as well as the two functions $g_1,g_2:\R\to \R$ defined by $g_i(x) := x^i$ for $i=1,2$.
Let $G:= \Cfield$ with $\Cfield$ being interpreted as the $\R$-vector space $\R^2$. Since for  $m\geq 3$
the function $g_m$ is bounded and continuous, it is $\PX$-Bochner integrable. 
For each $m\geq 3$ Theorem \ref{thm:limit-conditioning} thus
gives an $N_m\hochkl 2\in \sborelEx F$ with
$\P_Y(N_m\hochkl 2) = 0$ such that for all $y\in F\setminus N_m\hochkl 2$ we have
\begin{align*}
\p_{\P_{X|Y_n}(\mycdot|\nfilts n(y))}(t_m)
= \int_E  g_m(x) \,  \P_{X|Y_n}(\ind x| \nfilts n(y))
&\to
\int_E  g_m(x) \, \P_{X|Y}(\ind x|y) \\
&=
\p_{\P_{X|Y}(\mycdot|y)} (t_m) \, .
\end{align*}
In addition, for $i=1,2$ we have
\begin{align*}
\int_{\R } |g_i | \intd \P_{X} = \int_\R |x|^i \intd \P_X(x) < \infty
\end{align*}
since $\P_X$ is a Gaussian measure on $\R$. Consequently, $g_1$ and $g_2$ are also $\PX$-Bochner integrable, and therefore  Theorem \ref{thm:limit-conditioning}
 shows that there exist   $N_1\hochkl 2,N_2\hochkl 2\in \sborelEx F$ with
$\P_Y(N_i\hochkl 2) = 0$ such that for $i=1,2$  and all $y\in F\setminus N_i\hochkl 2$ we have
\begin{align*}
\mu_{n,y}\hochkl i :=  \int_E  x^i \, \P_{X|Y_n}(\ind x| \nfilts n(y))
\to
\int_E  x^i \, \P_{X|Y}(\ind x|y) =: \mu_{\infty,y}\hochkl i   \, .
\end{align*}
In other words, the first two raw moments $\mu_{n,y}\hochkl i$  of $\P_{X|Y_n}(\mycdot|\nfilts n(y))$
converge to the first two raw moments $\mu_{\infty,y}\hochkl i$ of $\P_{X|Y}(\mycdot|y)$. 
Now, the first raw moments are, of course, the expectations. Since the variances can be computed by the first and second raw moments, we further obtain 
\begin{align*}
\s_{n,y}^2:= \var \P_{X|Y_n}(\mycdot|\nfilts n(y)) &\to  \var  \P_{X|Y}(\mycdot|y) =: \s_{\infty,y}^2 \, .  
\end{align*}
Let us define  $N':= N\hochkls i \cup \bigcup_{m\geq 1}N_m\hochkl 2$. This implies  $N'\in \sborelEx F$ with $\P_Y(N') = 0$ and
\begin{align*}
\p_{\P_{X|Y_n}(\mycdot|\nfilts n(y))}(t)
\to
\p_{\P_{X|Y}(\mycdot|y)} (t)
\end{align*}
for all $t\in \Q$ and $y\in F\setminus N'$, and since $\P_{X|Y_n}(\mycdot|\nfilts n(y)) = \gauss {\mu_{n,y}\hochkl 1}{\s_{n,y}^2 }$ for such $y$ we   find
\begin{align*}
\p_{\P_{X|Y_n}(\mycdot|\nfilts n(y))}(t)
=  \eul^{\imi   t\mu_{n,y}\hochkl 1} \cdot \exp\biggl( -\frac{ t^2 \s_{n,y}^2 }{2} \biggr)
\to  
\p_{\gauss {\mu_{\infty,y}} {\s_{\infty,y}^2 } } (t)
\end{align*}
for all $t\in \R$ and $y\in F\setminus N'$. Consequently, for each fixed $y\in F\setminus N'$, we have
\begin{align}\label{thm:grv-conditioning-y-inf-dim-h4}
\p_{\P_{X|Y}(\mycdot|y)} (t) = \p_{\gauss {\mu_{\infty,y}} {\s_{\infty,y}^2 } } (t)
\end{align}
for all $t\in \Q$, and since both characteristic functions are continuous, we conclude that \eqref{thm:grv-conditioning-y-inf-dim-h4}
actually holds for all $t\in \R$. This shows that $\P_{X|Y}(\mycdot|y)$ is a Gaussian measure.

Let us now consider the case of a general separable Banach space $E$. Here we fix a countable
$\denseblo_{E'} \subset B_{E'}$ that is  $\tauweaks$-dense  in $B_{E'}$ and an enumeration $\denseblo_{E'} = \{x_i': i\geq 1\}$.
We write $X_i := x_i'\circ X$.
Our considerations in the case $E=\R$ 
then show that for each $i\geq 1$ there exists an $N_i'\in \sborelEx F$ with $\P_Y(N_i') = 0$ such that
$\P_{X_i|Y}(\mycdot|y)$ is a Gaussian measure
for all $y\in F\setminus N_i'$.

As in the proof of Theorem \ref{thm:grv-conditioning-y-fin-dim} we now translate this result to 
$\P_{X|Y}(\mycdot|y)$ with the help of Theorem \ref{thm:rcp-under-trafo}. To this end, 
we set $T_0 := E$ and $T_1 := \R$, as well as  $U_0 := U_1 := F$ and $\sA_0 := \sA_1 := \sborelEx F$.
Moreover, we consider the maps $\Phi := x_i'$ and $\Psi := \id_{F}$ as well as the measures
  $\P\hochkl 0 := \P_{(X,Y)}$ and $\P\hochkl 1 :=   \P_{(X_i,Y)} = \P\hochkl 1_{(\Phi,\Psi)}$.
By Theorem \ref{thm:rcp-under-trafo} there then 
exists an $N_i''\in \sborelEx F$ with 
$\P_Y(N_i'') = 0$ such that for all $y\in  F\setminus N_i''$ we have
\begin{align*}
\P_{X|Y}\bigl( (x_i')^{-1}(B)|y\bigr) =    \P_{X_i|Y}( B|y)    \, , \myqquad B\in \sborel.
\end{align*}  
Consequently,  we obtain
\begin{align*} 
\bigl(\P_{X|Y}( \mycdot  |y)\bigr)_{x'_i} = \P_{X_i|Y}(\mycdot |y)  
\end{align*}
for all  $y\in F\setminus N_i''$.
 Setting 
\begin{align*}
N\hochkls {ii}:= \bigcup_{i\geq 1} (N_i' \cup N_i'')
\end{align*}
we  find
$N\hochkls {ii}\in \sborelEx F$ with $\P_Y(N\hochkls {ii}) = 0$ and our construction ensures that 
$(\P_{X|Y}( \mycdot  |y))_{x'_i}$ is a Gaussian measure for all $i\geq 1$ and all $y\in F\setminus N\hochkls {ii}$.
Theorem \ref{thm:test-for-gms} thus shows  that $\P_{X|Y}( \mycdot  |y)$ is Gaussian for $y\in F\setminus N\hochkls {ii}$.

\ada {iii} Let us consider the map $g:= \id_E$. Since $E$ is separable, $g$ is strongly measurable. Moreover, 
since $X$ is a Gaussian random variable, Fernique's theorem, see \eqref{eq:Fernique's-theorem}, ensures
\begin{align*}
\int_{E } \snorm{g }_E \intd \P_{X} = \int_E \snorm x_E \intd \P_X(x) < \infty \, ,
\end{align*}
that is, $g:E\to E$ is $\PX$-Bochner integrable.
Therefore,  Theorem \ref{thm:limit-conditioning} shows that there exists an $N_g\in \sborelEx F$ with $\P_Y(N_g) = 0$
such that for all $y\in F\setminus (N\hochkls {ii}\cup N_g)$ we have
\begin{align*}
\muupn (\nfilts n(y)) = \int_E  x \, \P_{X|Y_n}(\ind x| \nfilts n(y)) \to  \int_E  x \, \P_{X|Y}(\ind x|y) =  \muup(y) \, .
\end{align*}
Setting $N\hochkls {iii} := N\hochkls {ii}\cup N_g$ thus shows \emph{iii)}.

\ada {iv} 
Let us consider the maps $h:E\to [0,\infty)$ and  $g:E \to \blonspace {E'}E$ defined by
\begin{align*}
h(x) := \snorm x_E^2\, , \\
g(x) := x\otimes x\, ,
\end{align*}
where 
the tensor products as well as the space of nuclear operators $\blonspace {E'}E$
are discussed  in front of Lemma \ref{lem:cross-cov-as-integral} and Lemma \ref{lem:elem-tens-conv}, respectively.
In addition, we consider 
  the probability space
$(E, \sborelnormx E, \P_X)$ and the function
$\tilde X:=   \id_E$. Because $E$ is separable, $\tilde X$  is
 strongly measurable. In addition, since $\P_X$ is a Gaussian measure, Fernique's theorem, see \eqref{eq:Fernique's-theorem},  shows
\begin{align*}
\int_{E } \snorm{\id_E}^2_E \intd \P_X
=
\int_E \snorm x_E^2 \intd \P_X(x) < \infty \, ,
\end{align*}
and hence we find $\tilde X\in \sLx 2 {\P_X, E}$.
Consequently, Lemma \ref{lem:cross-cov-as-integral} shows that the $\blonspace EE$-valued function
$g= \tilde X\otimes \tilde X$ is also strongly measurable
with $g\in \sLx 1 {\P_X, \blonspace{E'}E}$.
Therefore,  Theorem \ref{thm:limit-conditioning} applied with $G:= \blonspace {E'}E$ 
shows that there exists an $N_g\in \sborelEx F$ with $\P_Y(N_g) = 0$
such that for all $y\in F\setminus N_g$ we have 
\begin{align*}
g \in \sLx 1 {\P_{X|Y}(\mycdot|y),G} \myqquad \mbox{ and } \myqquad 
g \in \sLx 1 {\P_{X|Y_n}(\mycdot| \nfilts n(y)),G} 
\end{align*}
for all $n\geq 1$ as well as the $\blonspace {E'}E$-convergence
\begin{align}\label{thm:grv-conditioning-y-inf-dim-nuce-conv}
\int_E x \otimes x\, \P_{X|Y_n}(\ind x| \nfilts n(y)) 
 &\to  \int_E  x \otimes x \, \P_{X|Y}(\ind x|y)\, .
\end{align}
Moreover, Theorem \ref{thm:limit-conditioning} applied to $h = \snorm{\tilde X}_E^2$ with $G:= \R$ 
gives an $N_h\in \sborelEx F$ with $\P_Y(N_h) = 0$
such that for all $y\in F\setminus  N_h$ we have 
\begin{align*}
\tilde X \in \sLx 2 {\P_{X|Y}(\mycdot|y)} \myqquad \mbox{ and } \myqquad 
\tilde X \in \sLx 2 {\P_{X|Y_n}(\mycdot| \nfilts n(y))} 
\end{align*}
for all $n\geq 1$. 
We define $N\hochkls {iv} := N\hochkls {iii}\cup N_g \cup N_h$.
For $y\in F\setminus  N\hochkls {iv}$,  Lemma \ref{lem:cross-cov-as-integral} applied to $\tilde X\in  \sLx 2 {\P_{X|Y_n}(\mycdot| \nfilts n(y))} $ thus shows 
\begin{align*}
\covupn 
& = 
 \cov \P_{X|Y_n}(\mycdot|\nfilts n(y))\\
 &= \int_E (\tilde X \otimes \tilde X)(x) \, \P_{X|Y_n}(\ind x| \nfilts n(y)) \\
 &\quad - \biggl(\int_E  \tilde X(x) \, \P_{X|Y_n}(\ind x| \nfilts n(y)) \biggr) \otimes \biggl(\int_E  \tilde X(x) \, \P_{X|Y_n}(\ind x| \nfilts n(y)) \biggr) \\
 & = \int_E x \otimes x\, \P_{X|Y_n}(\ind x| \nfilts n(y))  - \muupn (\nfilts n(y)) \otimes \muupn (\nfilts n(y)) \\
 & \to  \int_E  x \otimes x \, \P_{X|Y}(\ind x|y) - \muup (y) \otimes \muup (y)
\end{align*}
with convergence in $\blonspace {E'}E$, where in the last step we used \eqref{thm:grv-conditioning-y-inf-dim-nuce-conv} and part \emph{iii)}
in combination with Lemma  \ref{lem:elem-tens-conv}.
Moreover,  Lemma \ref{lem:cross-cov-as-integral} applied to $\tilde X \in \sLx 2 {\P_{X|Y}(\mycdot|y)}$ also shows 
\begin{align*}
\cov \P_{X|Y}(\mycdot|y) =  \int_E  x \otimes x \, \P_{X|Y}(\ind x|y) - \muup (y) \otimes \muup (y)  
\end{align*}
for all $y\in F\setminus  N\hochkls {iv}$. Combining these considerations, we obtain the claimed  $\covupn \to \cov \P_{X|Y}(\mycdot|y)$ with  convergence in $\blonspace {E'}E$.
Furthermore, all $\covupn$ are independent of $y$, and hence so is their limit $\cov \P_{X|Y}(\mycdot|y)$, that is, the notation $\covup$ is indeed justified.

\ada {v} Let us fix a $y\in F\setminus N\hochkls {iv}$.
We then know from \emph{i)} and Proposition \ref{prop:cf-of-general-gm} that 
\begin{align}\label{thm:grv-conditioning-y-inf-dim-h00}
\p_{\P_{X|Y_n}(\mycdot| \nfilts n(y))}(x') = \eul^{\imi \dualpair{x'}{\muupn(\nfilts n(y))} E} \cdot \exp\biggl( -\frac{\dualpair {x'}{\covupn x'}E}{2} \biggr) 
\end{align}
holds for all $ x'\in E'$.
Moreover, note that  \emph{iv)}  implies 
\begin{align*}
\bigl| \dualpair {x'}{\covupn  x'}E  - \dualpair {x'}{ \covup x'}E \  \bigr| 
&= 
\bigl| \dualpair {x'}{(\covupn  - \covup)x'}E  \bigr| \\
&\leq 
 \snorm{\covupn  - \covup} \cdot \snorm{x'}_{E'}^2 \\
& \to 0\, .
\end{align*}
Combining this with 
 \emph{iii)} and \eqref{thm:grv-conditioning-y-inf-dim-h00} we thus obtain for all $x'\in E'$
 \begin{align}\label{thm:grv-conditioning-y-inf-dim-weak-conv-h1}
 \p_{\P_{X|Y_n}(\mycdot| \nfilts n(y))}(x') \to  \eul^{\imi \dualpair{x'}{\muup(y)} E} \cdot \exp\biggl( -\frac{\dualpair {x'}{\covup x'}E}{2} \biggr)\, .
 \end{align}
Furthermore, by Theorem \ref{thm:limit-conditioning}  there exists an $N^\star\in \sborelEx F$ with $\P_Y(N^\star) = 0$ such that  
\begin{align}\label{thm:grv-conditioning-y-inf-dim-weak-conv-h2}
\bigl\{\P_{X|Y_n}(\mycdot| \nfilts n(y)): n\geq 1\bigr\}
\end{align}
 is uniformly tight for all $y\in F\setminus N^\star$. We define $N\hochkls v := N\hochkls{iv}\cup N^\star$.
 For $y\in F\setminus N\hochkls v$ we then have both \eqref{thm:grv-conditioning-y-inf-dim-weak-conv-h1}
 and \eqref{thm:grv-conditioning-y-inf-dim-weak-conv-h2} and therefore
   Theorem \ref{thm:general-levy} shows that there then exists a probability measure $\nu_y$ on $X$ such that 
 $\P_{X|Y_n}(\mycdot| \nfilts n(y)) \to \nu_y$ weakly and 
 \begin{align*}
 \p_{\nu_y} (x')= \eul^{\imi \dualpair{x'}{\muup(y)} E} \cdot \exp\biggl( -\frac{\dualpair {x'}{\covup x'}E}{2} \biggr) \, , \myqquad x'\in E'\, . 
 \end{align*}
 Proposition \ref{prop:cf-of-general-gm} then shows that $\nu_y$ is a Gaussian measure with mean $\muup(y)$ and covariance operator $\covup$.
 By \emph{ii)}, \emph{iii)}, and \emph{iv)} we further know that $\P_{X|Y}(\mycdot |y)$ is also a 
 Gaussian measure with mean $\muup(y)$ and covariance operator $\covup$. As discussed in front of 
   \eqref{eq:gauss-meas-ex-cov} 
 we can then conclude that $\nu_y = \P_{X|Y}(\mycdot |y)$.

 \ada {vi}  We write $N:= N\hochkls {vi} := N\hochkls v$. Then
  $N_\Om := Y^{-1}(N)\in \sA$ satisfies $\P(N_\Om)=0$ since $N$ is a $\P_Y$-zero set.
 For $n\geq 1$ we now define the random variable $Z_n:\Om\to E$ by 
 \begin{align*}
 Z_n :=  \E X +  \cov ( X,Y_n) (\cov Y_n)\mpinv(Y_n- \E Y_n)  \, .
 \end{align*}
 Since for $\om \in \Om$ and $y:= Y(\om)$ we have $\om \in N_\Om$ if and only if $y\in N$, part \emph{iii)}
 together with $\nfilts n (y) = \nfilts n( Y(\om)) = Y_n(\om)$
 then shows that 
 the definition of $Z$ can be written as  
 \begin{align*}
 Z(\om) = 
 \begin{cases}
 \lim_{n\to \infty} Z_n(\om) \, ,& \mbox{ if } \om \in \Om\setminus N_\Om \\
 0 & \mbox { else.}
 \end{cases}
 \end{align*}
 Consequently, $Z$ as defined in \emph{vi)} is indeed a random variable and 
 we have the $\P$-almost sure convergence $Z_n\to Z$. Moreover, each $Z_n$ is $\s(Y_n) = \s(\nfilts n \circ Y)$-measurable and since 
 $N_\Om\in \s(Y)$, we conclude that $Z$ is actually $\s(Y)$-measurable.
 In addition,
  by Theorem \ref{thm:grv-conditioning-y-fin-dim}
 each $Z_n$ 
 is a 
 version of $\E(X|Y_n)$ and a
 Gaussian random variable with $\E Z_n = \E X$ and 
\begin{align}\nonumber
\cov Z_n & =  \cov ( X,Y_n) (\cov Y_n)\mpinv \cov (Y_n, X) \, , \\ \label{thm:grv-conditioning-y-inf-dim-cond-exp-0}
\gaussspace {Z_n} & \subset \gaussspace {Y_n}\, .
 \end{align}
Now, given an $x'\in E'$, we have   $\dualpair {x'}{Z_n(\om)}E \to \dualpair {x'}{Z(\om)}E$ for all $\om \in \Om\setminus N_\Om$, and 
therefore Theorem \ref{thm:levy-grv} shows that $\dualpair {x'}ZE$ is an $\R$-valued Gaussian random variable with 
\begin{align}\label{thm:grv-conditioning-y-inf-dim-cond-exp-1}
\E \dualpair {x'} {  Z_n}E &\to \E \dualpair {x'} {  Z}E \, , \\ \label{thm:grv-conditioning-y-inf-dim-cond-exp-2}
\var  \dualpair {x'} {  Z_n}E &\to \var\dualpair {x'} {  Z}E \, .
\end{align}
Now, since all $\dualpair {x'}ZE$ are  Gaussian,  $Z$ is  a Gaussian random variable.
Moreover, we have $\E \dualpair {x'} {  Z_n}E =  \dualpair {x'} { \E Z_n}E = \dualpair {x'} { \E X}E$, and hence we obtain 
\begin{align*}
 \dualpair {x'} { \E X}E =  \lim_{n\to \infty}  \dualpair {x'} { \E Z_n}E = \lim_{n\to \infty}  \E  \dualpair {x'} { Z_n}E =   \E  \dualpair {x'} { Z}E 
 = \dualpair {x'} { \E Z}E \, .
\end{align*}
%
%
%
Since this holds for all $x'\in E'$ we obtain 
  $\E X = \E Z$.
Similarly,  \eqref{eq:weak-cov-pos} yields 
\begin{align*}
\var  \dualpair {x'} {  Z_n}E 
=
\dualpair {x'}{\cov(Z_n) x'}E 
&= 
\dualpair {x'}{\cov ( X,Y_n) (\cov Y_n)\mpinv \cov (Y_n, X)  x'}E \\
&\to 
\dualpair {x'}{(\cov(X) - \covup) x'}E\, ,
\end{align*}
where in the last step we used that \emph{iv)}. 
Together with  \eqref{eq:weak-cov-pos} and \eqref{thm:grv-conditioning-y-inf-dim-cond-exp-2}
this gives 
\begin{align*}
\dualpair {x'}{\cov(Z) x'}E
= 
\var  \dualpair {x'} {  Z}E
= 
\dualpair {x'}{(\cov(X) - \covup) x'}E\, ,
\end{align*}
and since this  holds for all $x'\in E'$, Lemma \ref{lem:abstr-cov-on-diag} shows $\cov (Z) = \cov(X) - \covup$.

Our next step is to show that $Z$ is a version of $\E(X|Y)$. To this end,  
we write $\ca F_n := \s(Y_n)$ and $\ca F_\infty := \s(Y)$. Then we already know that $Z_n$ is a version of  $\E(X|\ca F_n)$.
Moreover,
Lemma \ref{lem:filt-seq-gives-filtration} shows that
$(\ca F_n)$ is   a filtration with
\begin{align*}
\ca F_\infty
= \s  \biggl( \bigcup_{n\geq 1} \ca F_n  \biggr)\, .
\end{align*}
By the martingale convergence Theorem
\ref{thm:martingale-convergence} we thus find
\begin{align*}
Z_n   \to \E(X|Y)
\end{align*}
in $\sLx 2 {\P, E}$ and $\P$-almost surely. Since we already have the $\P$-almost sure convergence $Z_n\to Z$, we conclude that
$\P(\{Z= \E(X|Y)\}) = 1$, and this in turn implies that $Z$ is indeed a version of $\E(X|Y)$ and 
$Z_n\to Z$ in $\sLx 2 {\P, E}$.

\ada {vii} We set $N\hochkls {vii} := N\hochkls {vi}$ and fix
 an $x'\in E'$. 
Since $Y_n - \E Y_n = \nfilts n\circ Y - \E (\nfilts n\circ Y) = \nfilts n\circ (Y-\E Y)$ and $\nfilts n$ is bounded and linear, 
Lemma \ref{lem:gaussspace-of-grv}  shows $ \gaussspace {Y_n  } \subset \gaussspace {Y  }$.
Combining this with  \eqref{thm:grv-conditioning-y-inf-dim-cond-exp-0} yields
\begin{align*}
\dualpair {x'}  {Z_n - \E Z_n}E \in  \gaussspace {Z_n} \subset    \gaussspace {Y_n  } \subset \gaussspace {Y  } \, .
\end{align*}
Moreover, $Z_n\to Z$ in $\sLx 2 {\P, E}$ implies 
$\dualpair {x'}  {Z_n}E\to \dualpair {x'}  {Z}E$ in $\Lx 2 \P$. By $\E Z_n = \E X = \E Z$ we conclude 
$\dualpair {x'}  {Z_n - \E Z_n}E\to \dualpair {x'}  {Z - \E Z}E$ in $\Lx 2 \P$.
Since $\gaussspace {Y  }$
is a closed subspace of $\Lx 2 \P$ by definition, we thus find $\dualpair {x'}  {Z-\E Z}E \in \gaussspace {Y}$.
Taking the $\Lx 2 \P$-closure then yields $\gaussspace {Z} \subset  \gaussspace {Y}$.

To verify the formula for the orthogonal projection, we
recall that a bounded linear operator on a Hilbert space
is an orthogonal projection if and only if it is idempotent and self-adjoint, see e.g.~\cite[Proposition II.3.3]{Conway90}.
Moreover, $\E(\mycdot |Y):\Lx 2\P\to \Lx 2 \P$ is known to be an orthogonal projection, where we note that this result is rarely stated
in the formulation above, so we quickly prove it here using well-known properties of conditional expectations, see e.g.~\cite[Theorem 8.14]{Klenke14}.
First, $\E(\mycdot |Y)$ is idempotent since for $f\in \Lx 2 \P$ we have 
\begin{align}\label{eq:idempot-cond-exp}
 \E (\E (f|Y)| Y) = \E(f|Y)\, .
\end{align}
Moreover, for $f,g\in \Lx 2 \P$ we have
\begin{align*}
\skprod f {\E(g|Y)}_{\Lx 2 \P}
= \E \bigl( f\cdot {\E(g|Y)}  \bigr)
= \E \bigl( \E(  f\cdot {\E(g|Y)}|Y)  \bigr)
&= \E \bigl( {\E(f|Y)}  \cdot   {\E(g|Y)}  \bigr) \, ,
\end{align*}
and since we analogously obtain $\skprod g {\E(f|Y)}_{\Lx 2 \P} =   \E \bigl( {\E(f|Y)}  \cdot   {\E(g|Y)}  \bigr) $, we  find 
\begin{align} \label{eq:selfadj-cond-exp}
\skprod f {\E(g|Y)}_{\Lx 2 \P} = \skprod  {\E(f|Y)} {g}_{\Lx 2 \P}\, ,
\end{align}
that is, $\E(\mycdot |Y)$ is also self-adjoint. 
Our next goal is to show
\begin{align}\label{eq:restr-cond-exp}
\E(f|Y) \in \gaussspace Y\, , \myqquad f \in \gaussspace{X,Y}.
\end{align}
To this end, we fix 
 an $f\in \gaussspace {X,Y}$. By the definition of $\gaussspace {X,Y}$, there
 then exists a sequence $(f_n)$ of the form $f_n = \dualpair {x_n'}XE + \dualpair{y_n'}Y E$
with $f_n \to f$ in $\Lx 2 \P$, where $x_n'\in E'$ and $y_n'\in F'$ are suitable functionals.
Now, Theorem \ref{thm:prop-cond-ex} together with $\gaussspace Z \subset \gaussspace Y$ implies 
\begin{align*}
\E\bigl(\dualpair {x_n'}XE|Y\bigr) =  \dualpairb {x_n'}{\E(X|Y)}E = \dualpair{x_n'}ZE \in \gaussspace Y\, .
\end{align*}
Moreover, since $\dualpair {y_n'}YE$ is $\s(Y)$-measurable, we have
\begin{align*}
\E\bigl(\dualpair {y_n'}YE|Y\bigr) = \dualpair {y_n'}YE \in \gaussspace Y\, .
\end{align*}
Combining both considerations shows $\E(f_n|Y)\in \gaussspace Y$ for all $n\geq 1$. Since 
$\E(\mycdot |Y):\Lx 2\P\to \Lx 2 \P$ is continuous, we further have the $\Lx 2 \P$-convergence
\begin{align*}
\E(f_n|Y) \to \E(f|Y)\, ,
\end{align*}
and since $\gaussspace Y$ is a closed subspace of $\Lx 2 \P$ we conclude $\E(f|Y)\in \gaussspace Y$.
This finishes the proof of \eqref{eq:restr-cond-exp}. Since $\gaussspace Y\subset \gaussspace{X,Y}$ we 
can now consider the operator 
\begin{align*}
\orthproj :\gaussspace{X,Y} &\to \gaussspace{X,Y} \\
f&\mapsto \E(f|Y)\, .
\end{align*}
The operator $\orthproj$ is bounded and linear as it is the restriction of $\E(\mycdot|Y)$ to $\gaussspace{X,Y}$
Moreover, \eqref{eq:idempot-cond-exp} shows that $\orthproj$ is idempotent
and since \eqref{eq:selfadj-cond-exp} holds for all $f,g\in \Lx 2 \P$ it also holds for all $f,g\in \gaussspace{X,Y}$,
and hence $\orthproj$ is self-adjoint. In other words, $\orthproj$ is an orthogonal projection.

It remains to show $\ran \orthproj = \gaussspace Y$. Here, we note that the inclusion ``$\subset$''
has already been established in \eqref{eq:restr-cond-exp}. For the proof of the converse inclusion, we pick
an $f\in \gaussspace Y$. By Lemma \ref{lem:gaussspace-of-grv} we know that the $\P$-equivalence class of $f$ contains
a $\s(Y)$-measurable representative, so we may assume without loss of generality that $f$ is $\s(Y)$-measurable. This leads to
\begin{align*}
f = \E(f|Y)  = \orthproj f  \in  \ran \orthproj\, .
\end{align*}

\ada {viii}
We fix an $x'\in E'$, for which
 $\dualpair{x'}XE$ is $\s(Y)$-measurable. Using Theorem \ref{thm:prop-cond-ex} we then obtain
\begin{align*}
\dualpairb{x'} {\muup \circ Y}E = \dualpair{x'} {Z}E = \dualpairb{x'} {\E(X|Y)}E = \E\bigl(\dualpair {x'}XE\bigl|Y\bigr) = \dualpair {x'}XE\, ,
\end{align*}
where the first identity holds up to the $\P$-zero set $Y^{-1}(N\hochkls {vii})\in \sA$
and
in the second identity we simply considered the version $Z$ of $\E(X|Y)$. Moreover,
 the third and forth identities are meant in the sense that
both $\dualpair{x'} {\E(X|Y)}E$ and $\dualpair {x'}XE$ are versions of $\E(\dualpair {x'}XE|Y)$.
Consequently, the set
\begin{align*}
N_{x'} := \bigl\{\dualpair{x'} {\muup \circ Y}E \neq  \dualpair {x'}XE\bigr\}
\end{align*}
is a $\P$-zero set with $N_{x'}\in \s(Y)$. Let us now consider the set 
\begin{align*}
\ca F := \bigl\{x'\in B_{E'}: \dualpair{x'}XE:\Om\to \R \mbox{ is $\s(Y)$-measurable }\bigr\}
\end{align*}
as well as a metric $\metric_{\tauweaks}$ describing the relative $\tauweaks$-topology of $B_{E'}$. Since
subsets of separable metric spaces are separable, and $(B_{E'},\metric_{\tauweaks})$ is separable by Theorem \ref{thm:alaoglu-and-more},
we conclude that there exists a countable $\ca F_0\subset \ca F$ that is $\metric_{\tauweaks}$-dense in $\ca F$.
We define 
\begin{align*}
N_{\ca F_0} := \bigcup_{x'\in \ca F_0} N_{x'}\, .
\end{align*}
By construction, we have $N_{\ca F_0}\in \s(Y)$ with $\P(N_{\ca F_0}) = 0$. Moreover, $N_{\ca F_0}\in \s(Y)$ gives an 
$N_{\ca F_0}'\in   \sborelEx F$ with $Y^{-1}(N_{\ca F_0}') = N_{\ca F_0}$ and $\P_Y(N_{\ca F_0}') = 0$. We define 
$N\hochkls {viii} := N\hochkls {vii} \cup N_{\ca F_0}'$. By construction, we then have 
\begin{align*}
\dualpairb {x'} {\muup (Y(\om))}E = \dualpair{x'}{X(\om)}E\, , \myqquad x'\in \ca F_0, \om \not \in Y^{-1}(N\hochkls {viii})\, .
\end{align*}
Since $\ca F_0$ is $\metric_{\tauweaks}$-dense in $\ca F$, we then find the first assertion for all $x'\in \ca F$, and a simple rescaling argument
then gives the desired identity in the general case.


Finally,  by \eqref{eq:weak-cov-pos}, part \emph{vi)}, Theorem \ref{thm:prop-cond-ex}, and the assumed $\s$-measurability we obtain
\begin{align*}
\dualpair {x'}{\cov(Z)x'}E 
= 
\var(\dualpair {x'}{Z}E )
=
\var\bigl(\dualpair {x'}{\E(X|Y)}E \bigr)
&=
\var\bigl(\E(\dualpair {x'} XE|Y)\bigr) \\
&=
\var\bigl(\dualpair {x'} XE\bigr) \\
&=
\dualpair {x'}{\cov(X)  x'}E\, .
\end{align*}
Using the identity $\covup = \cov(X) - \cov(Z)$ established in  \emph{vi)}, we then find
\begin{align*}
\dualpair {x'}{\covup x'}E 
=
\dualpair {x'}{\cov(X)  x'}E - \dualpairb {x'}{ \cov(Z) x'}E 
= 0 \, .
\end{align*}
By \eqref{eq:csu-for-abstr-cov} this implies  $\covup x' = 0$ as discussed there.
\end{proof}


\subsection{Proofs of Theorems \ref{thm:from-fce-to-rcp} and \ref{thm:canonical-rcp}}

In this subsection we present the proofs of the remaining two theorems of Section \ref{sec:main-results}.

\begin{proof}[Proof of Theorem \ref{thm:from-fce-to-rcp}]
In the following we write  $\covup :=  \cov(X) - \cov( \mu \circ Y)$ and 
 verify the three conditions of Definition \ref{def:reg-cond-prob}.

We begin by observing that
$\P_{X|Y}(\mycdot|y)$ is, of course,  a probability measure on $\sborel(E)$ for every $y\in F$ by our construction  \eqref{eq:rcp-from-fact-cond-exp}.

To verify the measurability condition \emph{ii)}, we define
\begin{align*}
\ca E &:= \bigcup_{x'\in E'} \s(x') \, , \\
\ca B &:= \bigl\{ B\in \sborelEx E:  y\mapsto \P_{X|Y}(B|y)  \mbox{ is measurable }\bigr\} \, .
\end{align*}
We first show that $\ca E \subset \ca B$. To this end, we choose a $B\in \ca E$. Then there exists an $x'\in E'$ and an $A\in \sborel$ such that $B=(x')^{-1}(A)$.
For fixed $y\in F$ we then obtain
\begin{align*}
\P_{X|Y}(B|y)
=
\gaussm {\mu(y)}{\covup} \bigl((x')^{-1}(A)\bigr)
=
\Bigl(\gaussm {\mu(y)}{\covup} \Bigr)_{x'}(A)
=
\gauss{m(y)}{\s^2}(A)\, ,
\end{align*}
where we have set $m(y) := \dualpair{x'}{\mu(y)}E$ and $\s^2$ denotes the $1\times 1$-matrix representing the linear operator $x'\circ \covup \circ (x')':\R'\to \R$.
Note the  map $y\mapsto m(y)$ is $(\sborelEx F,\sborel)$-measurable  since $\mu$ is $(\sborelEx F, \sborelEx E)$-measurable and $x'$ is $(\sborelEx E,\sborel)$-measurable.

It remains to verify the measurability of the map $y\mapsto \gauss{m(y)}{\s^2}(A)$. In the case $\s^2= 0$ this measurability immediately  follows from
\begin{align*}
 \gaussm{m(y)}{\s^2}(A) = \dirac{m(y)}(A) = \eins_{A}(m(y))
\end{align*}
and the measurability of $m$ and $A$. Moreover, in the case $\s^2 >0$ we have
\begin{align*}
\gauss{m(y)}{\s^2}(A)
= \frac 1{\sqrt{2\pi \s^2}} \int_\R \eins_A(t) \exp \Bigl( -\frac{(t-m(y))^2}{2\s^2} \Bigr) \intd t \, .
\end{align*}
Now, the integrand $f(t,y) :=  \eins_A(t) \exp \bigl( -\frac{(t-m(y))^2}{2\s^2} \bigr)$ is a non-negative $(\sborel\otimes \sborelEx F, \sborel)$-measurable function,
and therefore Fubini-Tonelli's theorem, see e.g.~\cite[Theorem 14.16]{Klenke14}, yields the measurability of $y\mapsto \gauss{m(y)}{\s^2}(A)$.
In both cases we have thus found $B\in \ca B$, and hence we have established  $\ca E \subset \ca B$.

Let us now show that $\ca B$ is a $\s$-algebra. Here, $\emptyset \in  \ca B$ is obvious since $y\mapsto \P_{X|Y}(\emptyset|y) = 0$ for all $y\in F$.
Moreover, if we have a $B\in \ca B$, then
\begin{align*}
\P_{X|Y}(E\setminus B|y) = 1 - \P_{X|Y}(B|y) \, , \myqquad y\in F
\end{align*}
quickly shows that $E\setminus B\in \ca B$. Similarly, if we have mutually disjoint $B_1,B_2,\dots\in \ca B$ and we set $B:= B_1\cup B_2\cup \dots$, then we have
\begin{align*}
\P_{X|Y}(B|y) = \sum_{i=1}^\infty \P_{X|Y}(B_i|y) \, , \myqquad y\in F
\end{align*}
and since each $y\mapsto \P_{X|Y}(B_i|y)$ is measurable, we conclude that $B\in \ca B$.

With these preparations we obtain $\s(\ca E) \subset \s(\ca B) = \ca B \subset \sborelEx E$, and since we also have $\s(\ca E) = \sborelEx E$, we conclude 
$\ca B = \sborelEx E = \sborelnormx E$. This shows condition \emph{ii)} of  Definition \ref{def:reg-cond-prob}.

Let us finally verify condition \emph{iii)} of  Definition \ref{def:reg-cond-prob}. To this end, we fix an arbitrary version
$\tilde \P_{X|Y}(\mycdot|\mycdot)$ of the regular conditional probability of $X$ given $Y$ and an $N\in \sborelEx F$ with $\P_Y(N)=0$ according to
Theorem \ref{thm:grv-conditioning-y-inf-dim}. We define $\muup :F\to E$ by \eqref{eq:muup-on-F} and write
 \begin{align*}
 \tilde Z := \muup \circ Y\, .
 \end{align*}
Then part \emph{vi)} of Theorem \ref{thm:grv-conditioning-y-inf-dim} shows that $\tilde Z$ is a version of $\E(X|Y)$, and hence we have
\begin{align*}
0 = \P (\{ Z \neq \tilde Z\}) = \P(\{ \mu \circ Y \neq \muup \circ Y\})  = \P\bigl (Y^{-1}(\{\mu\neq \muup\})     \bigr) = \P_Y(\{\mu\neq \muup\}) \, ,
\end{align*}
where in the first step we used the almost sure uniqueness of conditional expectations in the Banach space valued case, which is discussed
after Definition \ref{def:con-exp}.
In other words,  the set $N_1 := \{\mu\neq \muup\}$ satisfies
$N_1\in \sborelEx F = \sborel(F)$ and $\P_Y(N_1) = 0$.  For the $\P_Y$-zero set $N_2 := N \cup N_1$ and $y\in F\setminus N_2$ we then have
\begin{align*}
\P_{X|Y}(\mycdot|y) = \gaussm {\mu(y)}{\covup  } = \gaussm {\muup(y)}{\covup  } = \tilde \P_{X|Y}(\mycdot|y) \, .
\end{align*}
For $B\in \sborel(E)$ and $A\in \sborel(F)$ we thus find
\begin{align*}
\P_{(X,Y)}(B\times A) = \int_F \eins_A(y) \tilde \P(B|y) \intd \P_Y(y) = \int_F \eins_A(y) \P(B|y) \intd \P_Y(y)\, ,
\end{align*}
i.e.~we have verified  condition \emph{iii)} of  Definition \ref{def:reg-cond-prob}.
\end{proof}

\begin{proof}[Proof of Theorem \ref{thm:canonical-rcp}]
Let $B_n :F\to E$ be the bounded linear operator given by
\begin{align*}
B_n := \cov ( X,Y_n) (\cov Y_n)\mpinv \nfilts n\, .
\end{align*}
Since $\E Y_n = \nfilts n(\E Y)$, we then find
\begin{align*}
\Fconv = \bigl\{  y\in F: \exists  \lim_{n\to \infty} B_n(y-\E Y)\bigr\}\, .
\end{align*}
We now pick a version $\P_{X|Y}(\mycdot|\mycdot)$
of the regular conditional probability of $X$ given $Y$ and a $\P_Y$-zero set $N\in \sborelEx F$
according to Theorem \ref{thm:grv-conditioning-y-inf-dim}. In addition, we define $\muup:F\to E$
with the help of part \emph{iii)} of Theorem \ref{thm:grv-conditioning-y-inf-dim} and \eqref{eq:muup-on-F}.
Note that part \emph{iii)} of Theorem \ref{thm:grv-conditioning-y-inf-dim} shows $F\setminus N \subset \Fconv$.

\ada i Since $E$ is complete, we have
\begin{align*}
\Fconv
&=
\bigl\{  y\in F: (B_n(y-\E Y)) \mbox{ is Cauchy sequence in $E$}\bigr\} \\
&=
\bigcap_{k\geq 1} \bigcup_{n_0\geq 1} \bigcap_{n,m\geq n_0}\bigl\{ y\in F: \snorm{B_n(y-\E Y) - B_m(y-\E Y)}_E \leq 1/k  \bigr\} \, .
\end{align*}
By the continuity of the map $y\mapsto \snorm{B_n(y-\E Y) - B_m(y-\E Y)}_E$, we further see that
\begin{align*}
\bigl\{ y\in F: \snorm{B_n(y-\E Y) - B_m(y-\E Y)}_E \leq 1/k  \bigr\} \in \sborelnormx F
\end{align*}
for all $k,n,m\in \N$. Combining both observations yields $\Fconv \in \sborelnormx F = \sborelEx F$.

For the proof of $\P_Y(\Fconv) = 1$ we simply combine the already observed inclusion
$F\setminus N \subset \Fconv$ with $\P_Y(F\setminus N)= 1$.

\ada{ii} We first note that $B_n(0) = 0$ shows  $\E Y \in \Fconv$, and hence $0\in F_{X|Y, \nfiltseq}$.
Moreover, for $y_1,y_2\in F_{X|Y, \nfiltseq}$ and $\a_1,\a_2\in \R$
we have $\tilde y_i := y_i + \E Y \in \Fconv$ and
\begin{align*}
B_n\bigl((\a_1 y_1+\a_2 y_2 + \E Y) - \E Y \bigr)
&= \a_1 B_n y_1 + \a_2 B_n y_2 \\
&=  \a_1 B_n  (\tilde y_1  - \E Y )  + \a_2 B_n  (\tilde y_2  - \E Y ) \, .
\end{align*}
From this we easily conclude that $\a_1 y_1+\a_2 y_2 + \E Y \in \Fconv$, that is $\a_1 y_1+\a_2 y_2\in F_{X|Y, \nfiltseq}$.
Finally, $\Fconv \in \sborelEx F = \sborelnormx F$ implies $F_{X|Y, \nfiltseq}\in\sborelnormx F  = \sborelEx F$ since  translations $F\to F$
are continuous and thus $(\sborelnormx F, \sborelnormx F)$-measurable.

Let us now show $\ran \cov(Y,X)\subset F_{X|Y, \nfiltseq}$.
To this end, we fix an $x'\in E'$ and set $y:= \cov(Y,X)(x')$.
This gives $\nfilts n y = (\nfilts n \circ \cov(Y,X))(x') = \cov (Y_n, X) (x')$ and
by part \emph{iv)} of Theorem \ref{thm:grv-conditioning-y-inf-dim}
we   know that
\begin{align*}
\lim_{n\to \infty} \bigl( \cov ( X,Y_n) (\cov Y_n)\mpinv \cov (Y_n, X) \bigr) (x')
&=
\lim_{n\to \infty} \bigl( \cov ( X,Y_n) (\cov Y_n)\mpinv   \bigr) (\nfilts n y) \\
&=
\lim_{n\to \infty} B_n y
\end{align*}
exists. In other words, we have  $y + \E Y  \in \Fconv$, and hence $y \in -\E Y + \Fconv =  F_{X|Y, \nfiltseq}$.
%

\ada{iii} The definition of $\muupA:F\to E$ can be rewritten as
\begin{align*}
\muupA (y) = \E X + \lim_{n\to \infty} \eins_{\Fconv} (y)\cdot  B_n(y-\E Y)\, .
\end{align*}
In particular, $\muupA $ is the pointwise limit of a sequence of strongly measurable functions $F\to E$, and hence also strongly measurable.
Moreover, our construction ensures $\Fconv  \cap F\setminus N \subset \{\muup = \muupA\}$, and since we already know 
$F\setminus N \subset \Fconv$ we conclude that $\{\muup \neq  \muupA\} \subset N$.
%
%
This shows
\begin{align*}
\P\bigl(\{ \muup\circ Y \neq \muupA \circ Y \}   \bigr)
=
\P_Y\bigl( \{\muup \neq \muupA\}  \bigr)
\leq  \P_Y(N)  = 0\, .
\end{align*}
Since $\muup\circ Y$ is a version of
$\E(X|Y)$ by part \emph{iv)} of Theorem \ref{thm:grv-conditioning-y-inf-dim}, we conclude that $\muupA \circ Y$ is also a version of  $\E(X|Y)$.

\ada{iv} This is a direct consequence of \emph{iii)} and Theorem \ref{thm:from-fce-to-rcp}.
\end{proof}

%% file: proofs-filt-seq.tex
\section{Proofs Related to Filtering Sequences in Section \ref{seq:filt-seq}}\label{sec:proof-filt-subeqnarray}

This section contains all proofs for the results of Section \ref{seq:filt-seq} as well as some 
auxiliary results.

We begin by investigating 
the $\s$-algebra generated by some bounded linear  $A:F\to \R^n$.
To this end,
let $e_1,\dots,e_n$ be the standard ONB of $\R^n$.
We then write
\begin{align*}
\crdfct A := \bigl\{ a_i:F\to \R \,\bigl|\,  a_i := \skprod {e_i}A , i=1,\dots, n\bigr\}
\end{align*}
for the set of its coordinate functions $a_1,\dots,a_n$. Clearly, we have $\crdfct A\subset F'$.
To relate $\s(A)$ with $\crdfct A$ we further need to recall that 
 the annihilator of a non-empty $B\subset F$ is defined by
\begin{align*}
B^\annil &:= \bigl\{ y'\in F': \dualpair {y'}yF = 0 \mbox{ for all } y\in B   \bigr\}\, .
\end{align*}
Analogously, the  annihilator
  of a non-empty set $B\subset F'$ is defined by
  \begin{align*}
  B_\annil :=  \bigl\{ y\in F: \dualpair {y'}yF = 0 \mbox{ for all } y'\in B   \bigr\}\, .
  \end{align*}
With these preparations, our next intermediate result reads as follows.

\begin{lemma}\label{lem:fin-dim-ops-essentials}
Let $F$ be a Banach space and $A:F\to \R^n$ be a bounded linear operator. Then the following statements hold true:
\begin{enumerate}
\item We have $\s(A) = \s(\crdfct A) = \s(\spann \crdfct A)$.
\item We have $\spann \crdfct A = \ran A' = (\ker A)^\annil$.
\end{enumerate}
\end{lemma}

\begin{proof}[Proof of Lemma \ref{lem:fin-dim-ops-essentials}]
\ada i We first note that in view of Lemma \ref{lem:sub-ws-sig-alg} we only need to show $\s(A) = \s(\crdfct A)$.
To establish this identity, let
 $a_1',\dots,a_n':F\to \R$ be  the coordinate functions of $A$, that is $\crdfct A = \{a_1',\dots,a_n'\}$.
Then, \cite[Theorem 1.90]{Klenke14}
shows that for a given $\s$-algebra $\sA'$ on $F$,
the map $A:F\to \R^n$ is $(\sA', \sborel^n)$-measurable if and only if all $a_1',\dots, a_n'$ are
$(\sA', \sborel)$-measurable.

If we now consider $\sA':= \s(A)$, then $A$ is
$(\sA', \sborel^n)$-measurable, and hence
we conclude that $\s(a_i') \subset  \sA'= \s(A)$ for all $i=1,\dots,n$. This gives the inclusion
 $\s(a_1',\dots,a_n') \subset   \s(A)$.
Conversely, for $\sA' :=\s(a_1',\dots,a_n')$, we know that
all $a_1',\dots, a_n'$ are
$(\sA', \sborel)$-measurable, and hence $A$ is $\sA'$-measurable, that is
 $\s(A) \subset \s(a_1',\dots,a_n')$.

 \ada {ii} We first note that $A$ is given by $A= \sum_{i=1}^n \dualpair{a_i'}\mycdot F e_i =  \sum_{i=1}^n a_i' \otimes e_i$, 
 see \eqref{eq:def-tens-prod}, where  $e_1,\dots,e_n$ is the standard ONB of $\R^n$.
By
 \eqref{eq:adjoint-elem-tens}  its adjoint $A'$ is thus given by
 \begin{align*}
 A' t = \biggl(\sum_{i=1}^n e_i \otimes a_i' \biggr) (t) = \sum_{i=1}^n \skprod {e_i}t \cdot a_i' = \sum_{i=1}^n t_i a_i' \, ,
 \myqquad t=(t_1,\dots,t_n)\in \R^n,
 \end{align*}
 where we use the canonical identification of $\R^n$ with both $(\R^n)'$ and $(\R^n)''$. This shows $\spann \crdfct A = \ran A'$.

 For the proof of  the second identity, we   note that 
the linearity and continuity of $A$ gives
 \begin{align}\label{lem:fin-dim-ops-essentials-h1}
(\ker A)^\annil =  \bigl( (\ran A')_\annil \bigr)^\annil    = \overline{\ran A'}^{\tauweaks}\, ,
 \end{align}
 where the first identity can be found in e.g.~\cite[Lemma 3.1.16]{Megginson98} or \cite[Proposition VI.1.8]{Conway90},
 and the second identity can be found in
 e.g.~\cite[Proposition 2.6.6]{Megginson98}. Moreover, since the subspace $\ran A'$ of $F'$ is finite dimensional, it is $\tauweaks$-closed, see e.g.~\cite[Corollary 2.2.32]{Megginson98} or \cite[Theorem 1.21]{Rudin91}. Combining the latter with
 \eqref{lem:fin-dim-ops-essentials-h1} yields $(\ker A)^\annil = \ran A'$.
\end{proof}

The next proposition provides a general sufficient condition for filtering sequences. It will be at the core of the proof of Proposition \ref{prop:proper-filt-seq}.

\begin{proposition}\label{prop:suff-filt-seq}
Let $F$ be a separable Banach space and $(\nfilts n)$ be a sequence of bounded linear operators $\nfilts n:F\to \R^n$.
Then, $(\nfilts n)$ is a filtering sequence for $F$ if we have both $\spann \crdfct {\nfilts n} \subset \spann \crdfct {A_{n+1}}$ for all $n\geq 1$ and
\begin{align}\label{prop:suff-filt-seq-h1}
\overline{\bigcup_{n\geq 1} \spann \crdfct {\nfilts n}}^{\tauweaks} = F'\, .
\end{align}
\end{proposition}

\begin{proof}[Proof of Proposition \ref{prop:suff-filt-seq}]
By Lemma \ref{lem:fin-dim-ops-essentials} we find 
\begin{align*}
\s(\nfilts n) =  \s(\spann \crdfct {\nfilts n}) \subset  \s(\spann \crdfct {A_{n+1}}) = \s(\nfilts {n+1})\, .
\end{align*}
To verify the condition $\sborelEx F =  \s\bigl( \nfilts n: n\geq 1 \bigr)$, we write $F_n' : = \spann \crdfct {\nfilts n}$
and
\begin{align*}
F_\infty' := \bigcup_{n\geq 1} \spann \crdfct {\nfilts n} = \bigcup_{n\geq 1} F_n'\, .
\end{align*}
Since $F_n'\subset F_{n+1}'$ for all $n\geq 1$, we quickly see that $F_\infty'$ is a subspace of $F'$. Using our assumption and
Lemma
\ref{lem:sub-ws-sig-alg} we further have
\begin{align*}
\sborelEx F=   \s(\overline{F_\infty'}^{\tauweaks})  = \s(F_\infty')
= \s\biggl( \bigcup_{y'\in F_\infty'} \!\s(y') \!\biggr)
= \s\biggl( \bigcup_{n\geq 1}\bigcup_{y'\in F_n'} \!\s(y') \!\biggr)
&\subset
\s\biggl( \bigcup_{n\geq 1} \s(  F_n') \!\biggr) \\
&=
\s\biggl( \bigcup_{n\geq 1} \s(\nfilts n) \!\biggr) ,
\end{align*}
where in the last step we used the already observed $\s(  F_n') =  \s(\spann \crdfct {\nfilts n}) = \s(\nfilts n)$.
Hence we have found $\sborelEx F \subset \s\bigl( \nfilts n: n\geq 1 \bigr)$. The converse inclusion is trivial.
\end{proof}

\begin{proof}[Proof of Proposition \ref{prop:proper-filt-seq}]
Our construction gives $\crdfct {\nfilts n} = \{y_1',\dots,y_n'\}$. 
This identity implies $\spann \crdfct {\nfilts n} \subset \spann \crdfct {\nfilts{n+1}}$ for all $n\geq 1$ as well as 
\begin{align*}
\spann \{ y_i': i\geq 1  \} = \bigcup_{n\geq 1} \spann \crdfct {\nfilts n}\, .
\end{align*}
Applying  Proposition \ref{prop:suff-filt-seq} then yields the assertion.
\end{proof}

\begin{proof}[Proof of Theorem \ref{thm:dense-gives-filter-seq}]
We will apply Proposition \ref{prop:proper-filt-seq}. 
To verify \eqref{prop:proper-filt-seq-h1} we pick a $y'\in F'$. Without loss of generality we may assume $y'\neq 0$.
We define $y_\infty' := \snorm{y'}^{-1} y'$.
Since $y_\infty'\in B_{F'}$  and the relative  $\tauweaks$-topology of $B_{E'}$ is metrizable, see Theorem \ref{thm:alaoglu-and-more}, there then exists a sequence
$(y_{i_k}')\subset \denseblo_{F'}$ with $y_{i_k}' \to y_\infty'$ in the $\tauweaks$-topology. This shows
\begin{align*}
y_\infty' \in \overline{\spann \{ y_i': i\geq 1  \}}^{\tauweaks}\, ,
\end{align*}
and rescaling gives the inclusion ``$\supset$'' in \eqref{prop:proper-filt-seq-h1}. The converse inclusion 
is trivially satisfied. 
\end{proof}

\begin{proof}[Proof of Proposition \ref{prop:sep-sequences}]
For $n\geq 1$ we
define $\nfilts n:F\to \R^n$ by \eqref{prop:proper-filt-seq-h2}. This definition ensures 
%
$\spann \crdfct {\nfilts n} \subset \spann \crdfct {\nfilts{n+1}}$ for all $n\geq 1$ and hence we have 
\begin{align*}
\spann \{ y_i': i\geq 1  \} = \bigcup_{n\geq 1} \spann \crdfct {\nfilts n}\, .
\end{align*}
Moreover, by Lemma \ref{lem:fin-dim-ops-essentials} we know $\spann \crdfct {\nfilts n} =  (\ker  {\nfilts n})^\annil$, and hence we find 
\begin{align*}
 \biggl(\bigcup_{n\geq 1}\spann \crdfct {\nfilts n} \biggr)_{\!\annil} 
 = \bigcap_{n\geq 1} (\spann \crdfct {\nfilts n})_\annil
 = \bigcap_{n\geq 1} \bigl((\ker \nfilts n )^\annil\bigr)_\annil 
 =
 \bigcap_{n\geq 1}  \overline{\ker \nfilts n}^{\snorm \cdot_F}
 = \bigcap_{n\geq 1} \ker \nfilts n  ,
\end{align*}
where the first identity almost directly follows from the definition of  annihilators, and the third identity can be found in e.g.~\cite[Proposition 1.10.15]{Megginson98}.
We thus obtain
\begin{align*}
\overline{\spann \{ y_i': i\geq 1  \}}^{\tauweaks}
=
\overline{\bigcup_{n\geq 1} \spann \crdfct {\nfilts n}}^{\tauweaks} 
=  \biggl( \biggl(\bigcup_{n\geq 1}\spann \crdfct {\nfilts n} \biggr)_\annil \biggr)^\annil 
= \biggl( \bigcap_{n\geq 1} \ker \nfilts n\biggr)^\annil \, ,
\end{align*}
where 
the second identity can be found in e.g.~\cite[Proposition 2.6.6]{Megginson98}. 
Moreover, we have $\nfilts n(y) = 0$ if and only if $\dualpair {y_1'}{y}F=\dots=\dualpair {y_n'}{y}F = 0$. In other words,
the identity $\ker \nfilts n = \ker y_1' \cap \dots \cap \ker y_n'$ holds. Combining this with our previous considerations yields
\begin{align*}
\overline{\spann \{ y_i': i\geq 1  \}}^{\tauweaks} = \biggl( \bigcap_{i\geq 1} \ker y_i'\biggr)^\annil \, .
\end{align*}
Finally note that  
$\{y_i': i\geq 1\}\subset F'$ is separating if and only if we have 
\begin{align*}
\bigcap_{i\geq 1} \ker y_i' = \{0\}\, .
\end{align*}
Now the equivalence follows from the well-known fact that $\{0\}\subset F$ is the only non-empty subset $B\subset F$
with $B^\annil = F'$.
\end{proof}

\begin{proof}[Proof of Corollary \ref{cor:filt-seq-HS}]
Let us identify $H$ with $H'$ via the canonical and isometric Fr\'echet-Riesz
isomorphism $h\mapsto \hat h := \skprod h\mycdot_H$. Then $(\hat e_n)_{n\geq 1}$ is an ONB of $H'$, and hence the assertion  follows from 
the discussion around \eqref{prop:proper-filt-seq-h1-norm} and 
Proposition \ref{prop:proper-filt-seq}.
\end{proof}

\begin{proof}[Proof of Theorem \ref{thm:bsf-in-cb}]
We begin by verifying that $(\diracf{t_i}) \subset F'$ is separating. To this end, let us fix an $f\in F$ with $f\neq 0$.
Then there exists a $t\in T$ with $f(t) \neq 0$ and then  a  sequence $(t_{i_k})\subset T_\infty$ with $\metric(t_{i_k}, t)\to 0$.
This gives $f(t_{i_k})\to f(t)$ and hence there exists an $i_k$ with $\dualpair {\diracf {t_{i_k}}} fF =  f(t_{i_k}) \neq 0$.
In other words, $(\diracf{t_i}) \subset F'$ is separating.
Applying Proposition  
\ref{prop:sep-sequences}
and then Proposition \ref{prop:proper-filt-seq} thus gives the assertions.
\end{proof}

\begin{proof}[Proof of Corollary \ref{cor:filt-seq-rkhs}]
The separability can be found in e.g.~\cite[Lemma 4.33]{StCh08}. 

Our next goal is to show that each $h\in H$ is continuous. To this end, let 
$\Phi:T\to H$ be the ``canonical feature map''  given by $\Phi(t) := k(t,\mycdot)$. Since $k$ is continuous, $\Phi$ is also continuous, see
e.g.~\cite[Lemma 4.29]{StCh08}. For $t\in T$ and $(t_n)\subset T$ with $t_n\to t$, we thus find 
$h(t_n) = \skprod {\Phi(t_n)} h_H \to  \skprod {\Phi(t)} h_H = h(t)$ by the reproducing property of $k$. 

Now,  Theorem \ref{thm:bsf-in-cb} in combination with the reproducing property    gives the assertion.
\end{proof}

\begin{proof}[Proof of Proposition \ref{prop:filt-seq-RKHS}]
We first show the existence of a countable $\metric_k$-dense   $T_\infty\subset T$. 
To this end,   recall that subsets of separable metric spaces are separable, and hence the set
\begin{align*}
K_T := \bigl\{ k(t,\mycdot): t\in T  \bigr\}
\end{align*}
is $\snorm\cdot_H$-separable. Consequently, there exists a countable $T_\infty \subset T$ such that
\begin{align*}
K_0 := \bigl\{ k(t,\mycdot): t\in T_\infty  \bigr\}
\end{align*}
is $\snorm\cdot_H$-dense in $K_T$. For a fixed $t\in T$ we thus find a sequence
$(t_n)\subset T_\infty$ with $k(t_n,\mycdot) \to k(t,\mycdot)$ in $H$. The definition of $\metric_k$ then gives $\metric_k(t,t_n) \to 0$,
i.e.~$T_\infty$ is $\metric_k$-dense in $T$.

For the proof of the second assertion, we first note that  for $h\in H$ and $t,t'\in T$ the reproducing property of $k$ and the Cauchy-Schwarz inequality yield
\begin{align*}
|h(t)-h(t')| = \bigl|\skprod h {k(t,\mycdot) - k(t',\mycdot)}_H\bigr| \leq \snorm h_H \cdot \snorm{k(t,\mycdot) - k(t',\mycdot)}_H = \snorm h_H \cdot \metric_k(t,t')\, .
\end{align*}
Consequently, every $h\in H$ is $\metric_k$-continuous. 
\end{proof}

\begin{proof}[Proof of Theorem \ref{thm:good-metric-for-bsf}]
Since $F$ is separable, Theorem \ref{thm:alaoglu-and-more} shows that 
there exists a metric $\metric_{\tauweaks}$ on $B_{F'}$ that metrizes the relative $\tauweaks$-topology of $B_{F'}$. In the following, we fix such a metric 
$\metric_{\tauweaks}$. Moreover, we consider the    function $m:T\to [0,\infty)$ given by 
 \begin{align*}
 m(t) 
 := 
 \begin{cases}
 \snorm{\diracf t}_{F'}^{-1}\, , & \mbox{ if } \diracf t \neq 0\, , \\
 0 \, , &\mbox{ else. }
 \end{cases}
 \end{align*}
 Our first goal is to show that for all $f\in F$ and $t\in T$ we have 
 \begin{align}\label{thm:good-metric-for-bsf-h1}
 \dualpair {\diracf t} fF = \snorm{\diracf t}_{F'}\cdot m(t) \cdot  \dualpair {\diracf t} fF\, .
 \end{align}
Obviously,  if  $\diracf t \neq 0$ this identity directly follows from the definition of $m$. Moreover, in the case 
 $\diracf t = 0$ we have $ \dualpair {\diracf t} fF = 0$ for all $f\in F$, and therefore both sides of \eqref{thm:good-metric-for-bsf-h1} equal $0$.
 
 Now, from \eqref{thm:good-metric-for-bsf-h1} we conclude that $\snorm {\diracf t}_{F'} =  \snorm{\diracf t}_{F'} \cdot \snorm{m(t) \diracf t}_{F'}$,
 which in turn implies $\snorm{m(t) \diracf t}_{F'} \leq 1$ for all $t\in T$.
 This makes it possible to consider the map 
 \begin{align*}
 \Phi: T & \to B_{F'} \times [0,\infty) \\
 t & \mapsto \bigl(m(t) \diracf t, \snorm{\diracf t}_{F'}\bigr)\, .
 \end{align*}
 In the following, we equip $B_{F'} \times [0,\infty)$ with the metric 
 \begin{align*}
 \metric_{\tauweaks}^+\bigl( y_1',r_1), (y_2',r_2) \bigr) := \metric_{\tauweaks}(y_1',y_2') + |r_1-r_2|\, , \myqquad y_1',y_2'\in B_{F'}, r_1,r_2\in [0,\infty).
 \end{align*}
 Since $(B_{F'},  \metric_{\tauweaks})$ is a separable metric space, we quickly verify that $(B_{F'} \times [0,\infty),  \metric_{\tauweaks}^+)$ 
 is also a separable metric space. Since subsets of separable metric spaces are separable metric spaces, we then see that $(\Phi(T),  \metric_{\tauweaks}^+)$
 is also a separable metric space. Consequently, if we define $\metric :T\times T\to [0,\infty)$ by 
 \begin{align*}
 \metric(s,t) :=  \metric_{\tauweaks}^+(\Phi(s), \Phi(t))\, , \myqquad s,t\in T,
 \end{align*}
 then $(T, \metric)$ becomes a separable pseudo-metric space.
 
 It remains to show that all $f\in F$ are continuous with respect to $\metric$. To this end, we fix some $f\in F$ and  $t\in T$, as well as a sequence $(t_n)\subset T$ with 
 $\metric (t_n,t)\to 0$. By the definition of $\metric$, $ \metric_{\tauweaks}^+$, and $\Phi$ the latter ensures 
 \begin{align*}
 \metric_{\tauweaks}\bigl( m(t_n) \diracf {t_n}, m(t) \diracf t  \bigr) + \bigl| \snorm{\diracf {t_n}}_{F'} - \snorm{\diracf t}_{F'}  \bigr|  
 = 
  \metric_{\tauweaks}^+(\Phi(t_n),\Phi(t))
 \to 0\, .
 \end{align*}
 This in turn ensures $\snorm{\diracf {t_n}}_{F'} \to \snorm{\diracf t}_{F'}$ and 
 \begin{align*}
 \dualpairb {m(t_n)\diracf {t_n}} fF \to \dualpairb {m(t)\diracf {t}} fF \, .
 \end{align*}
 Combining these two limits with \eqref{thm:good-metric-for-bsf-h1} we   find $f(t_n) =  \dualpair {\diracf {t_n}} fF \to  \dualpair { \diracf {t}} fF = f(t)$.
\end{proof}

\begin{proof}[Proof of Lemma of \ref{lem:c1-sep-seq}]
Let us fix an $f\in C^1([0,1])$ with $f\neq 0$. Then, if $f$ is constant, we find $y_0'(f) = f(0) \neq 0$, and  if $f$ is not constant,
there is an $t\in (0,1)$ with $f'(t) \neq 0$. Since $f'$ is continuous, we then find a $t_i$ with $y_i'(f) = f'(t_i) \neq 0$
as in the proof of Theorem \ref{thm:bsf-in-cb}.
\end{proof}

%% file: proofs-examples.tex
\section{Proofs for Section \ref{sec:examples}}\label{sec:example-proofs}

In this section we present the proofs related to the results of  Section \ref{sec:examples}.

\begin{proof}[Proof of Lemma \ref{lem:sp-gives-rv}]
The equivalence directly follows from Lemma \ref{lem:sp-gives-rv-app}, and the continuity of the mean and covariance functions can be found in 
Lemma \ref{lem:cont-mean+cov-4-cont-GP}.
\end{proof}

For the proof of Theorem \ref{thm:grv-conditioning-y-fin-dim-CT} and Theorem \ref{thm:grv-conditioning-y-inf-dim-CT} we need the following 
intermediate lemma that computes the update formulas for finite $S\subset T$.

\begin{lemma}\label{thm:grv-conditioning-CT-bs-2-fct}
Let $(\Om, \sA,\P)$ be a probability space, $(T,\metric)$ be a compact metric space, and $(X_t)_{t\in T}$ be a
Gaussian processes on $(\Om, \sA,\P)$ with continuous paths, mean function $m$, and covariance function $k$.
Moreover, let $X:\Om\to \sC T$ be the Gaussian random variable associated to the process via \eqref{eq:sp-gives-rv}
and $S := \{s_1,\dots,s_n \}\subset T$. Finally, let 
\begin{align*}
Y := \npeval S\circ X 
\end{align*}
and $m(S) := (m(s_1),\dots,m(s_n))$,
where $\npeval S:\sC T\to \R^n$ is defined by \eqref{eq:fin-eval-op}.
Then for every $y\in \R^n$ and all $t,t_1,t_2\in T$  we have 
\begin{align*}
\dualpairb {\diracf t}{\E X + \cov ( X,Y) (\cov Y)\mpinv(y- \E Y)}{\sC T}  = m(t) + K_{t,S}   K_{S,S}\mpinv(y- m(S))
\end {align*}
as well as
\begin{align*}
\dualpairb {\diracf {t_1}}{ (\cov ( X) \!-\! \cov ( X,Y) (\cov Y)\mpinv \cov (Y, X) ) \diracf  {t_2}}{\sC T} = k(t_1,t_2) \!- \! K_{t_1,S}   K_{S,S}\mpinv K_{S,t_2} .
\end{align*}
\end{lemma}

\begin{proof}[Proof of Lemma \ref{thm:grv-conditioning-CT-bs-2-fct}]
Note that our notations give $\E Y = (\E X_{s_1},\dots,\E X_{s_n}) = m(S)$.
Using  \eqref{eq:covs-of-compositions-new}
the first identity then follows from
\begin{align*}
&\dualpairb {\diracf {t}}{ \E X + \cov ( X,Y) (\cov Y)\mpinv(y- \E Y)}{\sC T}\\
&=\E X_t + \bigl( \diracf {t} \circ \cov ( X,Y) \circ (\cov Y)\mpinv \bigr) (y- \E Y) \\
& = m(t) + \bigl(\cov (  \diracf {t} \circ X,Y) \circ (\cov Y)\mpinv \bigr) (y- m(S)) \\
& = m(t) +  \cov (  \diracf {t} \circ X,Y) \bigl(K_{S,S}\mpinv(y- m(S)) \bigr) \\
& =  m(t) + K_{t,S}   K_{S,S}\mpinv(y- m(S)) \, .
\end{align*}
For the proof of the second identity, we first observe that
\begin{align*}
& \dualpairB {\diracf {t_1}}{\bigl(\cov ( X) - \cov ( X,Y) (\cov Y)\mpinv \cov (Y, X) \bigr) (\diracf {t_2})}{\sC T} \\
&= \dualpairb {\diracf {t_1}}{\cov ( X) (\diracf {t_2})}{\sC T}  - \dualpairb {\diracf {t_1}}{ \cov ( X,Y) (\cov Y)\mpinv \cov (Y, X)   (\diracf {t_2})}{\sC T} \, .
\end{align*}
Moreover, by \eqref{eq:weak-cross-cov} and $\dualpair {\diracf {t_i}}{X}{\sC T} = X_{t_i}$  we obtain
\begin{align*}
\dualpairb {\diracf {t_1}}{\cov ( X) (\diracf {t_2})}{\sC T}
=
\cov\bigl( \dualpair {\diracf {t_1}}{X}{\sC T}, \dualpair {\diracf {t_2}}{X}{\sC T} \bigr)
=
k(t_1,t_2) \, .
\end{align*}
In addition, the definition of cross covariance operators gives
\begin{align*}
 \cov (Y, X) (\diracf {t_2} )= \int_\Om (X_{t_2}-\E X_{t_2}) (Y-\E Y) \intd \P = \cov(Y, X_{t_2})(1) = K_{S,t_2} \cdot 1 = K_{S,t_2}\, ,
\end{align*}
and hence we find
\begin{align*}
&
\dualpairb {\diracf {t_1}}{ \cov ( X,Y) (\cov Y)\mpinv \cov (Y, X)   (\diracf {t_2})}{\sC T} \\
&= \bigl(  \cov ( \diracf {t_1} \circ X,Y) \circ  (\cov Y)\mpinv    \bigr) \bigl(\cov (Y, X) (\diracf {t_2})\bigr) \\
& = \bigl(  \cov ( \diracf {t_1} \circ X,Y) \circ  (\cov Y)\mpinv   \bigr) (K_{S,t_2}) \\
&=  \cov ( \diracf {t_1} \circ X,Y) \bigl( K_{S,S}\mpinv  K_{S,t_2} \bigr) \\
&=  K_{t_1,S}   K_{S,S}\mpinv K_{S,t_2}\, .
\end{align*}
Combining these considerations yields the second identity.
\end{proof}

\begin{proof}[Proof of Theorem \ref{thm:grv-conditioning-y-fin-dim-CT}]
In view of Theorem \ref{thm:grv-conditioning-y-fin-dim} it suffices to translate the formulas for $\muup$ and   $\covup$ to the claimed formulas
for the mean and covariance function via \eqref{eq:meanfct-by-mean}, respectively \eqref{eq:covfct-by-cov}.
This, however, is straightforward by using Lemma \ref{thm:grv-conditioning-CT-bs-2-fct}.
\end{proof}


\begin{proof}[Proof of Theorem \ref{thm:grv-conditioning-y-inf-dim-CT}]
By Corollary \ref{cor:filt-seq-CT}  the operators $\npeval {S_n}: \sC S\to \R^n$
give  a filtering sequence   for $F:= \sC S$.
In addition, the definition of $Y$ ensures that  $(X,Y)$ are jointly Gaussian.
Applying 
Theorem \ref{thm:grv-conditioning-y-inf-dim}  with  $E:= \sC T$ and $Y_n :  =  \npeval {S_n} \circ Y$ gives us a $\P_Y$-zero set $N\in \sborel(\sC S)$ such 
that the assertions of parts \emph{i)} to \emph{viii)} hold true.
In particular, we know that
 $\P_{X|Y}(\mycdot|g)$ is a Gaussian measure on $\sC T$ for all $g\in \sC S\setminus N$.

To compute its mean and covariance function we pick a 
 $g\in \sC S\setminus N$ and write $\npeval n := \npeval {S_n}$. Then,   part \emph{i)} of Theorem \ref{thm:grv-conditioning-y-inf-dim}  together with Lemma \ref{thm:grv-conditioning-CT-bs-2-fct}
 shows
\begin{align*}
&
\dualpairb {\diracf t}{\E X +   \cov ( X,Y_n) (\cov Y_n)\mpinv(\npeval n(g)- \E Y_n)}{\sC T} \\
&= m(t) +  K_{t,S_n}   K_{S_n,S_n}\mpinv(g(S_n)  - m(S_n))
\end{align*}
and
\begin{align*}
&
\dualpairb {\diracf {t_1}}{\bigr(\cov ( X) - \cov ( X,Y_n) (\cov Y_n)\mpinv \cov (Y_n, X)\bigr)( \diracf  {t_2})}{\sC T} \\
&=  k(t_1,t_2)-  K_{t_1,S_n}   K_{S_n,S_n}\mpinv K_{S_n,t_2}
\end{align*}
for all $t,t_1,t_2\in T$.
Combining part \emph{iii)} of Theorem \ref{thm:grv-conditioning-y-inf-dim} with \eqref{eq:meanfct-by-mean} we thus obtain
\begin{align*}
&
\sup_{t\in T} \Bigl|   m(t) +  K_{t,S_n}   K_{S_n,S_n}\mpinv\bigl(g(S_n)  - m(S_n)\bigr)   -  \meanupg(t)  \Bigr| \\
&=
\sup_{t\in T} \Bigl| \dualpairb {\diracf t}{\E X \!+ \!  \cov ( X,Y_n) (\cov Y_n)\mpinv(\npeval n(g)- \E Y_n)}{\sC T}  -  \dualpair  {\diracf {t}}{  \muup(g)}{\sC T} \Bigr| \\
&=
\sup_{t\in T} \Bigl| \dualpairb {\diracf t}{\E X +   \cov ( X,Y_n) (\cov Y_n)\mpinv(\npeval n(g)- \E Y_n) -  \muup(g) }{\sC T} \Bigr| \\
& = \mnorm{\E X +   \cov ( X,Y_n) (\cov Y_n)\mpinv(\npeval n(g)- \E Y_n) -  \muup(g) }_{\sC T} \\
& \to 0\, ,
\end{align*}
where in the second to last step we used that the norm of $\sC T$ is the supremums norm.

The formula for the covariance function can be derived similarly: Indeed, combining
part \emph{iv)} of Theorem \ref{thm:grv-conditioning-y-inf-dim} with \eqref{eq:covfct-by-cov}, we obtain
\begin{align*}
&
\sup_{t_1,t_2\in T} \Bigl|  k(t_1,t_2) -  K_{t_1,S_n}   K_{S_n,S_n}\mpinv K_{S_n,t_2}  - \kfctup (t_1,t_2)    \Bigr| \\
&=
\sup_{t_1,t_2\in T} \Bigl| \dualpairB {\diracf {t_1}}{\bigl(\cov ( X) -\cov ( X,Y_n) (\cov Y_n)\mpinv \cov (Y_n, X) \bigr) (\diracf {t_2})}{\sC T}  \\
&\qquad \qquad\qquad
-  \dualpair {\diracf {t_1}}{\covup \diracf {t_2}}{\sC T}   \Bigr| \\
&=
\sup_{t_1,t_2\in T} \Bigl| \dualpairB {\diracf {t_1}}{\bigl(\cov ( X) -\cov ( X,Y_n) (\cov Y_n)\mpinv \cov (Y_n, X) - \covup\bigr) (\diracf {t_2})}{\sC T}    \Bigr| \\
& =
\sup_{t_2\in T}  \mnorm{\bigl(\cov ( X) -\cov ( X,Y_n) (\cov Y_n)\mpinv \cov (Y_n, X) - \covup\bigr) (\diracf {t_2})}_{\sC T} \\
&\leq
\mnorm{\cov ( X) -\cov ( X,Y_n) (\cov Y_n)\mpinv \cov (Y_n, X) - \covup} \\
&\to 0\, ,
\end{align*}
where we used that
$\diracf {t_2}\in B_{\sC T'}$ for all $t_2\in T$.

To prove the last assertion, we note that $\dualpair {\diracf s}X{\sC T} = X_s = \dualpair {\diracf s}Y{\sC S}$.
Consequently, 
the random variable $\dualpair {\diracf s}X{\sC T}$ is $\s(Y)$-measurable.
If we now  pick an $\om \in Y^{-1}(\{g\})$, then we have $\om \not \in Y^{-1}(N)$ and $g = Y(\om) = X_{|S}(\om)$, and therefore
part \emph{viii)} of Theorem \ref{thm:grv-conditioning-y-inf-dim}
gives
\begin{align*}
\meanupg(s)
=
\dualpairb {\diracf s}{ \muup(Y(\om))}{\sC T}
=
\dualpairb {\diracf s}{X(\om) }{\sC T}
= g(s)\, .
\end{align*}
Analogously,
part \emph{viii)} of Theorem \ref{thm:grv-conditioning-y-inf-dim} shows
\begin{align*}
\kfctup(s,t)  =  \kfctup(t,s) = \dualpair {\diracf t}{\covup \diracf s}{\sC T} = 0
\end{align*}
 for all $t\in T$.
\end{proof}

%% file: appendix-integration.tex
\section{Measurability in Banach spaces and Banach Space Valued Integration}\label{app:bs-integration}

In this supplement we recall some basic measurability results in Banach spaces $E$ as well as the Bochner integral.

We begin, however, with a result on the
$\tauweaks$-topology on $E'$ for separable Banach spaces $E$. 
Its first two assertions can be found
in \cite[Theorem 2.6.18 and Theorem 2.6.23]{Megginson98}
or \cite[Theorem V.3.1 and Theorem V.5.1]{Conway90}, and the last assertion follows
from the fact that compact metric spaces are separable.

\begin{theorem}\label{thm:alaoglu-and-more} 
Let $E$ be a separable Banach space. Then $B_{E'}$ is $\tauweaks$-compact,
the relative $\tauweaks$-topology of $B_{E'}$ is metrizable, and
there exists a countable $\denseblo \subset B_{E'}$
that is 
$\tauweaks$-dense in $B_{E'}$.
\end{theorem}

Note that for the proofs of our main results we need the existence of such sets $\denseblo$, and
another key result, namely Theorem \ref{thm:dense-gives-filter-seq} that ensures the existence of our approximation scheme, 
also needs such sets in its formulation.
In addition,
the fact that the $\tauweaks$-topology of $B_{E'}$ is metrizable is often used in the the proofs as it allows to work with sequences 
rather than nets and this in turn opens the door for limit theorems of integrals. Finally, working with $\denseblo$ often makes it possible to 
 better control unions of
sets with measure zero.

The following theorem, which can be found in e.g.~\cite[Proposition 1.1.1]{HyvNVeWe16} and which also follows from a version of
 Pettis' measurability theorem \cite[Theorem E.9]{Cohn13}, shows that for separable Banach spaces the $\s$-algebras 
 $\sborelEx E$ and $\sborelnorm E$ coincide.


\begin{theorem}\label{thm:pettis-var}
Let $E$ be a separable Banach space.  Then we have $\sborelEx E =  \sborelnorm E$.
\end{theorem}

The next
lemma  considers the behavior of $\s$-algebras on $F$ that are generated by some $B\subset F'$. A somewhat similar yet less general result can be found in \cite[Proposition 1.1.1]{HyvNVeWe16}.

\begin{lemma}\label{lem:sub-ws-sig-alg}
Let $F$ be a Banach space and $B\subset F'$ be non-empty. Then we have
\begin{align*}
\s(B) = \s(\spann B) = \s(\overline B^{\tauweaks}) = \s(\overline {\spann B}^{\tauweaks}) \, .
\end{align*}
\end{lemma}

\begin{proof}[Proof of Lemma \ref{lem:sub-ws-sig-alg}]
We first prove the equation $\s(B) = \s(\spann B)$. Here, the inclusion ``$\subset$'' is obvious. To establish the converse, we fix an
$y'\in \spann B$. Then there exist an $n\geq 1$ as well as $\a_1,\dots,\a_n\in \R$ and $y_1',\dots,y_n'\in B$ with
$y' = \a_1 y_1'+\dots+\a_n y_n'$. Clearly, each $y_i':F\to \R$ is $(\s(B),\sborel)$-measurable, and hence $y'$ is also
$(\s(B),\sborel)$-measurable. 
This gives $\s(y') \subset \s(B)$, and since $y'\in \spann B$ was chosen arbitrarily,  we find $\s(\spann B) \subset \s(B)$. 

For the proof of $\s(\spann B) = \s(\overline {\spann B}^{\tauweaks})$ we write $V:= \spann B$ and 
\begin{align*}
W:= \bigl\{ y'\in F': y':F\to \R \mbox{ is $\s(V)$-measurable} \bigr\}\, .
\end{align*}
Then $W$ is a subspace of $F'$ with $V\subset W$. Moreover, if we have a sequence $(y_i')\subset W$ that converges to an $y'\in E'$ in
 the  $\tauweaks$-topology, then we have 
\begin{align*}
 \dualpair{ y'}yF =  \lim_{i\to \infty} \dualpair{y_{i}'}yF \, , \myqquad y\in F.
\end{align*}
Since each $y_i':F\to \R$ is  $(\s(V),\sborel)$-measurable, we then conclude that their pointwise limit $y'$
is also $(\s(v),\sborel)$-measurable, that is $y'\in W$. In other words, $W$ is $\tauweaks$-sequentially closed.
Now a consequence of Krein-\u{S}mulian's theorem, see e.g.~\cite[Corollary 2.7.13]{Megginson98}, shows that $W$ is even $\tauweaks$-closed,
and hence we find 
\begin{align*}
\overline V^{\tauweaks} \subset \overline W^{\tauweaks} = W\, .
\end{align*}
By the definition of $W$ we conclude that every $y'\in \overline V^{\tauweaks}$ is $\s(V)$-measurable, and this in turn shows 
$\s(\overline {\spann B}^{\tauweaks}) \subset \s(\spann B) $. The converse inclusion is trivial.

The proof of $\s(B) = \s(\overline B^{\tauweaks})$ follows from the already established identities and trivial inclusions, namely
$\s(B) \subset \s(\overline B^{\tauweaks}) \subset \s(\overline {\spann B}^{\tauweaks}) = \s(\spann B) = \s(B)$.
%
%
\end{proof}

Let us now turn to the Bochner integral. To this end, 
let $(\Om,\sA)$ be a measurable space, $E$ be a Banach space, and $X:\Om\to E$.
It is well known, that $X$ is strongly measurable if and only if it is weakly measurable and
$X(\Om)$ is separable, see e.g.~\cite[Theorem E.9]{Cohn13} or \cite[Theorem 1.1.6]{HyvNVeWe16}.
In addition,
%
if   $X$ is strongly measurable, then
 there exists a sequence $(X_n)$ of (strongly) measurable step functions $X_n:\Om\to E$ with
\begin{align}\label{eq:step-fct-approx-1}
\lim_{n\to \infty} X_n(\om) &= X(\om)   \, , \\ \label{eq:step-fct-approx-2}
\snorm{X_n(\om)}_E &\leq \snorm {X(\om)}_E
\end{align}
for all all $\om \in \Om$ and $n\geq 1$, see e.g.~\cite[Proposition E.2]{Cohn13}.
Here we note that step functions are also called simple functions. In addition,
for step functions  all three notions of measurability coincide, and for this reason, we usually
speak of measurable step functions.
Finally, every pointwise limit $X$ of a sequence $(X_n)$ of strongly measurable functions is strongly measurable,
see e.g.~\cite[Proposition E.1]{Cohn13}. Consequently, we can characterize strong measurability by pointwise approximations 
with measurable step functions in the sense of \eqref{eq:step-fct-approx-1}. In particular, the finite sum of strongly measurable functions is strongly
measurable and if $f:\Om\to \R$ is measurable and $X:\Om\to E$ is strongly measurable, then $f\cdot X:\Om\to E$ is 
strongly measurable.

If we have a probability measure $\P$ on $(\Om,\sA)$ and a
measurable step function $X:\Om\to E$ with some fixed representation $X= \sum_{i=1}^m \eins_{A_i} x_i$ with $A_i\in \sA$ and $x_i \in E$, then
\begin{align}\label{def:step-fct-bochner-int}
\int_\Om X\intd \P := \sum_{i=1}^m \P(A_i) x_i
\end{align}
is actually independent of the chosen representation, where the proof is analogous to the $\R$-valued case. To extend this integral
we say that a strongly measurable $X:\Om\to E$ is Bochner integrable with respect to  $\P$, if
\begin{align*}
\int_\Om \snorm {X(\om)}_E  \intd \P(\om) < \infty\, .
\end{align*}
In this case,
using   approximating sequences described in \eqref{eq:step-fct-approx-1} and \eqref{eq:step-fct-approx-2}, one can define the Bochner integral by
\begin{align}\label{def:bochner-integral}
\E X := \int_\Om X\intd \P := \lim_{n\to \infty} \int_\Om X_n \intd \P\, ,
\end{align}
where  the integral in the middle is independent of the chosen approximation $(X_n)$, see  e.g.~\cite[page 399]{Cohn13}.
Moreover, the limit is, of course, with respect to the norm $\snorm\cdot_E$.
The set $\sLx 1 {\P,E}$ of all Bochner integrable $X:\Om\to E$ is a vector space, see e.g.~\cite[Proposition E.4]{Cohn13},
which obviously contains all measurable step functions. 
Moreover, it is easy to see that
\begin{align}\label{eq:bochner-1-norm}
\snorm X_{\sLx 1 {\P,E}} :=  \int_\Om \snorm{X(\om )}_E \intd \P(\om)
\end{align}
defines a semi-norm on $\sLx 1 {\P,E}$.
In addition, the Bochner integral defines
a linear map $\sLx 1 {\P,E} \to E$ that is also bounded since we have
\begin{align}\label{eq:bochner-int-cont}
\bnorm {\int_\Om X\intd \P }_E \leq \int_\Om \snorm{X(\om )}_E \intd \P(\om) \, , \myqquad X\in \sLx 1 {\P,E}\, ,
\end{align}
see e.g.~\cite[Proposition E.5]{Cohn13}.
As usual, we write $\Lx 1 {\P,E}$ for the space of $\P$-equivalence classes, which becomes a Banach space
if we modify the norm above in the usual way, see e.g.~\cite[page 401]{Cohn13}.

The Bochner integral also enjoys a form of the dominated convergence theorem. Namely, if we have a sequence $(X_n)$ of strongly
measurable $X_n:\Om\to E$ for which there exist both an $X:\Om\to E$ and a $\P$-integrable $g:\Om\to [0,\infty]$ such that
\begin{align} \label{eq:meas-fct-approx-1}
 \lim_{n\to \infty} X_n(\om) &= X(\om)\, ,\\ \label{eq:meas-fct-approx-2}
\snorm{X_n(\om)}_E &\leq \snorm {g(\om)}_E
\end{align}
for all all $\om \in \Om$ and $n\geq 1$, then $X$ and all $X_n$ are Bochner $\P$-integrable, and we have
\begin{align}\label{eq:dominated-conv-bnochner}
 \lim_{n\to \infty} \int_\Om X_n \intd \P = \int_\Om X\intd \P \, .
\end{align}
For a proof, we refer to e.g.~\cite[Proposition E.6]{Cohn13}, where we note that the strong measurability of $X$ follows from
 e.g.~\cite[Proposition E.1]{Cohn13}.

Moreover, if $T:E\to F$ is a bounded linear operator mapping into another Banach space $F$ and $X\in \sLx 1 {\P,E}$, then the map $T\circ X:\Om\to F$ satisfies  both
$T\circ X \in \sLx 1 {\P,F}$ and
\begin{align}\label{eq:transform-bochner}
\int_\Om T\circ X\intd \P = T\biggl(   \int_\Om X\intd \P  \biggr)\, ,
\end{align}
see e.g.~\cite[Theorem 6 in Chapter II.2]{DiUh77} or \cite[Theorem 1.2.4]{HyvNVeWe16}
for a stronger result in this direction that goes back to E.~Hille.
In particular,  for all  $X\in \sLx 1 {\P,E}$, we have both $\dualpair {x'}XE \in \sLx 1 \P$ and the ``Pettis property''
\begin{align}\label{eq:bochner-is-pettis}
\int_\Om  \dualpair {x'}XE  \intd \P = \dualpairbb {x'}  {\int_\Om X\intd \P}E\, , \myqquad x'\in E'\, .
\end{align}
This shows that for $X\in \sLx 1 {\P,E}$ the Bochner integral coincides with the so-called Pettis integral, which is
discussed in e.g.~\cite[Chapter 1.2.c]{HyvNVeWe16}.

Similarly to $\sLx 1 {\P,E}$ one can define, for $p\in (1,\infty)$, the vector space
$\sLx p {\P,E}$ as the set of all strongly measurable
$X:\Om\to E$   with
\begin{align*}
\snorm X_{\sLx p {\P,E}}^p :=  \int_\Om \snorm{X(\om)}_E^p \intd \P(\om) < \infty.
\end{align*}
Again this defines  a semi-norm on $\sLx p {\P,E}$ and $\Lx p {\P,E}$ denotes the Banach space of corresponding  $\P$-equivalence classes.

%% file: appendix-covariances.tex
\section{Covariance Operators}\label{app:covariances}

In this supplement properties of covariance operators of Banach space valued random variables are collected from the literature. In addition, some 
basic and probably well-known results we could not find the literature are presented and proved.

To begin with, 
let us fix a probability space $(\Om,\sA,\P)$ and Banach spaces $E$ and $F$. 
Moreover, let 
  us  fix some $X\in \sLx 2 {\P,E}$ and $Y\in \sLx 2 {\P,F}$, as well as a $y'\in F'$. 
 Since $X$ and $Y$ are strongly measurable, the map 
\begin{align*}
\dualpair {y'}{Y-\E Y} F \cdot (X-\E X):\Om \to E
\end{align*}
  is also strongly measurable, see the discussion following \eqref{eq:step-fct-approx-2}.
  In addition,
the Cauchy-Schwarz inequality shows that it is also Bochner integrable.
Moreover, the resulting linear operator $\cov(X,Y):F'\to E$ defined by
\begin{align*}
\cov(X,Y)y'= \int_\Om \dualpair {y'} {Y-\E Y} F \, (X-\E X) \intd \P \, , \myqquad y'\in F'
\end{align*}
is bounded by virtue of the norm-inequality for Bochner integrals  \eqref{eq:bochner-int-cont} and another application of the Cauchy-Schwarz inequality.
Applying the ``Pettis property'' \eqref{eq:bochner-is-pettis} then gives
\begin{align}\label{eq:weak-cross-cov}
\dualpair {x'}{\cov(X,Y)y'}E
&= \int_\Om \dualpair {y'} {Y-\E Y} F \, \dualpair{x'}{X-\E X}E \intd \P \\ \nonumber
&= \cov\bigl(  \dualpair {y'} {Y} F , \dualpair{x'}{X}E \bigr) \, .
\end{align}
In other words, the cross covariance operator provides the covariances between all pairs of one-dimensional projections of $X$ and $Y$.
In particular, if $X$ and $Y$ are independent, then all pairs $\dualpair{x'}{X}E$ and $\dualpair {y'} {Y} F$ are independent, and hence we find 
$\cov(X,Y) = 0$.
Another routine check shows that the dual $\cov (X, Y)':E'\to F''$ of $\cov(X,Y)$ is given by 
\begin{align}\label{eq:dual-of-cross-cov}
\cov (X, Y)' = \bidualmap_F \circ  \cov(Y,X) \, ,
\end{align}
where $\bidualmap_F:F\to F''$ denotes the canonical embedding of $F$ into its bidual $F''$.
In addition, if $S:E\to \tilde E$ and $T:F\to \tilde F$ are bounded linear 
 operators and we have $x\in   E$ and $ y\in F$, then the following identities hold true:
 \begin{align}\label{eq:covs-of-compositions-new}
 \cov( S\circ X,  T\circ Y  ) &= S \circ \cov(X,Y) \circ T' \, , \\ \label{eq:covs-of-trala}
  \cov( x + X, y+  Y  )  &=  \cov(X,Y)  \, . 
 \end{align}
 Indeed,  \eqref{eq:covs-of-compositions-new} follows by applying \eqref{eq:transform-bochner} with the operator $S$ and using the definition of $T'$. Moreover, 
 \eqref{eq:covs-of-trala} is a direct consequence of the definition of $\cov(X,Y)$.

In the case $X=Y$, the covariance operator $\cov(X) := \cov(X,X)$ is an operator $\cov (X):E'\to E$ and with \eqref{eq:weak-cross-cov} we find 
\begin{align}\label{eq:weak-cov-sym}
\dualpair {x'_1}{\cov(X)x'_2}E &= \dualpair {x'_2}{\cov(X)x'_1}E  \, , \\ \label{eq:weak-cov-pos}
\dualpair {x'}{\cov(X)x'}E &= \var \bigl(   \dualpair{x'}{X}E \bigr) \geq 0
\end{align}
for all $x',x_1',x_2'\in E'$, where $\var(Z)$ denotes the variance of an $\R$-valued random variable $Z\in \sLx 2 \P$.
This motivates the following definition.

\begin{definition}\label{def:abstr-cov}
Let $E$ be a Banach space. A bounded linear operator $T:E'\to E$ is called an abstract covariance operator if it is
symmetric and non-negative, that is,
if
 for all $x',x_1',x_2'\in E'$, we have
\begin{align*}
\dualpair {x_1'} {Tx_2'}E = \dualpair {x_2'} {Tx_1'}E  \myqquad \mbox{ and } \myqquad \dualpair {x'} {Tx'}E \geq 0\, .
\end{align*}
Analogously,   a bounded linear operator $S:E\to E'$ is symmetric and non-negative, if for all $x,x_1,x_2\in E$ we have
\begin{align*}
\dualpair {Sx_1}{x_2}E = \dualpair {Sx_2}{x_1}E  \myqquad \mbox{ and } \myqquad \dualpair {Sx} {x}E \geq 0\, .
\end{align*}
\end{definition}

As discussed above, every covariance operator $\cov (X)$ is an abstract covariance operator.
In addition, an elementary calculation shows that 
the sum $T_1+T_2:E'\to E$ of two abstract covariance operators $T_1,T_2:E'\to E$ is an abstract covariance operator.
Moreover,   every abstract covariance operator $T:E'\to E$ satisfies
\begin{align*}
\dualpair{x_1' + x_2'}{T (x_1' + x_2')}E
=
\dualpair{x_1'}{T x_1'}E
+
2\dualpair{x_1'}{T x_2'}E
 +
\dualpair{x_2'}{T x_2'}E\, ,
\end{align*}
which in turn implies
\begin{align}\label{eq:cov-kernel-diagonal-new}
\dualpair{x_1'}{T x_2'}E
=
\frac{\dualpair{x_1' + x_2'}{T (x_1' + x_2')}E - \dualpair{x_1'}{T x_1'}E - \dualpair{x_2'}{T x_2'}E} 2
\end{align}
for all $x_1',x_2'\in E'$.
The following lemma uses this observation to show
that abstract covariance operators are determined by the behavior of $x'\mapsto \dualpair {x'} {Tx'}E$, that is, on the diagonal.

\begin{lemma}\label{lem:abstr-cov-on-diag}
Let $E$ be a Banach space and $S,T:E'\to E$ be two abstract covariance operators with
\begin{align*}
\dualpair{x'}{T x'}E = \dualpair{x'}{S x'}E\, , \myqquad x'\in E'.
\end{align*}
Then we have $S=T$.
\end{lemma}

\begin{proof}[Proof of Lemma \ref{lem:abstr-cov-on-diag}]
Let us fix some $x_1',x_2'\in E'$. By \eqref{eq:cov-kernel-diagonal-new} and the assumed identity on the diagonal we then know
 $\dualpair{x_1'}{S x_2'}E   = \dualpair{x_1'}{T x_2'}E$.
This implies $Sx_2' = Tx_2'$ for all $x'_2\in E'$ by Hahn-Banach's Theorem.
Consequently we have $S= T$.
\end{proof}

In the finite dimensional case $E=\R^n$, one can quickly check that a linear operator $T:(\R^n)'\to \R^n$ is an abstract covariance operator, if and only if 
its representing $n\times n$-matrix $A$ with respect to the canonical bases of $(\R^n)'$ and $\R^n$ is symmetric and positive semi-definite.
In this case, the spectral theorem for matrices shows $A=S\Lambda S\transpose$, where 
$\Lambda$ is a diagonal $n\times n$-matrix given by the eigenvalues $\lb_1,\dots,\lb_n$ of $A$
and $S$ is an orthogonal 
$n\times n$-matrix $S$. Moreover, we have 
\begin{align}\label{eq:psd-of-mpinv}
A\mpinv = (S\Lambda S\transpose)\mpinv = S\Lambda\mpinv S\transpose\, , 
\end{align}
where the second identity can be found in e.g.~\cite[Theorem 1.1.2]{WaWeQi18}. Moreover, a straightforward calculation shows that 
$\Lambda\mpinv$ is the diagonal $n\times n$-matrix having the entries $\lb_1^{-1},\dots,\lb_n^{-1}$ with $0^{-1} := 0$. 
Consequently, $A\mpinv$ is also symmetric and positive semi-definite. Moreover, this matrix describes 
the operator $T\mpinv:\R^n\to (\R^n)'$, and therefore $T\mpinv$ is symmetric and non-negative.


It can be quickly verified that an abstract covariance operator $T:E'\to E$ defines a (reproducing) kernel $k_T:E'\times E'\to \R$  by
\begin{align*}
k_T(x_1',x_2') := \dualpair {x_1'} {Tx_2'}E\, , \myqquad x_1',x_2'\in E'\, .
\end{align*}
Lemma \ref{lem:abstr-cov-on-diag} shows that this kernel is determined by its behaviour on the diagonal. Moreover, the 
Cauchy-Schwarz inequality for kernels, see e.g.~\cite[(4.14)]{StCh08} yields
\begin{align}\nonumber
|\dualpair {x_1'} {Tx_2'}E|^2
= |k_T(x_1',x_2')|^2 
&\leq k_T(x_1',x_1') \cdot k_T(x_2',x_2') \\ \label{eq:csu-for-abstr-cov}
&= \dualpair {x_1'} {Tx_1'}E \cdot \dualpair {x_2'} {Tx_2'}E \, .
\end{align}
Consequently, for every $x_2'\in E'$ with $\dualpair {x_2'} {Tx_2'}E = 0$ we find 
 $Tx_2' = 0$ by a simple application of Hahn-Banach's theorem.

The next lemma collects some additional information
about this kernel $k_T$ in the case of  covariance operators $T=\cov(X)$.

\begin{theorem}\label{thm:continuous-covariance-kernel}
Let $(\Om, \sA, \P)$ be a probability space, $E$ be a separable Banach space,
$X\in \sLx 2 {\P,E}$, and $k_X := k_{\cov (X)}:E'\times E'\to \R$ be the resulting kernel. Then the
restriction $(k_X)_{|B_{E'}\times B_{E'}}$ of $k_X$ onto $B_{E'}\times B_{E'}$ is $\tauweaks$-continuous.
\end{theorem}

\begin{proof}[Proof of Theorem \ref{thm:continuous-covariance-kernel}]
Since $(B_{E'}, \tauweaks)$ is metrizable, see Theorem \ref{thm:alaoglu-and-more},
it suffices to show that the restricted kernel is $\tauweaks$-sequentially continuous.
Moreover, by \eqref{eq:cov-kernel-diagonal-new} it suffices to show that the kernel restricted to $2B_{E'}$  is
$\tauweaks$-sequentially continuous on the diagonal.
Finally, by a simple scaling argument, it actually suffices to verify that
  $x'\mapsto k_X(x',x')$ is $\tauweaks$-sequentially continuous on $B_{E'}$.

To prove the latter, we fix some $x'_0 \in B_{E'}$ as well as a sequence $(x_n') \subset B_{E'}$ such that
$x_n'\to x'_0$  with respect to $\tauweaks$.
Moreover,   we write    $Z:= X-\E X$. By \eqref{eq:covs-of-trala} and \eqref{eq:weak-cross-cov}
we then find
\begin{align}
\scriptonlyraw{&}
\dualpair {x_n'}{\cov (X) x_n'}E \scriptonlyraw{\\ \nonumber}
=  \dualpair  {x_n' }{\cov (Z) x_n'}E
&=
 \int_\Om  \dualpair {x_n'}ZE \cdot   \dualpair {x_n'}ZE    \intd\P        \label{thm:continuous-covariance-kernel-h1}
\end{align}
for all $n\geq 0$.
For $n\geq 0$ we define $g_n:\Om \to [0,\infty)$ by
\begin{align*}
g_n(\om) :=   \dualpair {x_n'}{Z(\om)}E^2\, , \qquad \qquad \om \in \Om.
\end{align*}
By the assumed $\tauweaks$-convergence $x_n'\to x'_0$ we then see that $g_n(\om)\to g_0(\om)$ for all $\om \in \Om$.
Moreover, a simple estimate yields
\begin{align*}
|g_n(\om)|
\leq \snorm{x_n'}_{E'}^2 \cdot  \snorm{Z(\om)}_E^2
\leq  \snorm{Z(\om)}^2_E \, , \myqquad \om \in \Om.
\end{align*}
Using $Z\in \sLx 2 {\P,E}$,
Lebesgue's dominated convergence theorem, and \eqref{thm:continuous-covariance-kernel-h1}, we then find
\begin{align*}
\dualpair {x_n' }{\cov (X) x_n' }E    \to \dualpair {x_0' }{\cov (X) x_0' }E \, ,
\end{align*}
that is, we have shown the assertion.
\end{proof}


\begin{lemma}\label{lem:cov-cross-co}
Let $E$ and $F$ be separable Banach spaces, $(\Om, \sA, \P)$ be a probability space, and  
$X\in \sLx 2 {\P,E}$ and $Y\in \sLx 2 {\P,F}$ be centered.
Moreover, let 
%
%
$S:F\to F'$ be a bounded,  linear,  symmetric, and non-negative operator.
Then the operator $K:E'\to E$ defined by
\begin{align*}
K := \cov(X,Y) S \cov(Y,X)\,.
\end{align*}
is an abstract covariance operator, and for all $x_1',x_2'\in E'$ we have
\begin{align*}
\dualpair {x_2'}{Kx_1'}E
\scriptonlyraw{&}= \!\int\limits_{\Om\times \Om} \!\!\!\dualpair{x_1'}{X(\om_1)}E \cdot \dualpair{x_2'}{X(\om_2)}E \scriptonlyraw{\\}
\scriptonlyraw{&\qquad\qquad}
\cdot \dualpair{SY(\om_1)}{Y(\om_2)}F \intd \P^2(\om_1,\om_2) .
\end{align*}
Moreover, the map $B_{E'}\to \R$ defined by $x'\mapsto  \dualpair{ x'} {Kx'}E$ is $\tauweaks$-continuous.
\end{lemma}

\begin{proof}[Proof of Lemma \ref{lem:cov-cross-co}]
Obviously, $K$ is bounded and linear. To verify that $K$ is an abstract covariance operator, we fix some
$x_1',x_2'\in E'$ and write $y_i := \cov(Y,X)x_i' \in F$. With the help of \eqref{eq:dual-of-cross-cov} we then obtain 
\begin{align*}
\dualpair {x_1'} {Kx_2'}E 
= \dualpair {x_1'} {\cov(X,Y) S y_2}E 
&= \dualpair {\cov(X,Y)'x_1'}{Sy_2}{F'}\\
&=\dualpair {\bidualmap_F\cov(Y,X)x_1'}{Sy_2}{F'} \\
&= \dualpair {Sy_2}{y_1}F \, .
\end{align*}
Since we assumed that $S$ is symmetric and non-negative, we can now quickly verify that this is also true for $K$.

To establish the integral formula, we again
 fix some  $x_1',x_2'\in E'$.
Then the definition of $ \cov(Y,X)$ gives
\begin{align*}
y'_1:=
S \cov(Y,X) x'_1
\scriptonlyraw{&}= S \biggl( \int_\Om  \dualpair{x_1'}XE \cdot Y \intd \P \biggr) \scriptonlyraw{\\}
\scriptonlyraw{&}= \int_\Om  \dualpair{x_1'}XE \cdot (S\circ  Y) \intd \P\, ,
\end{align*}
where in the last step we used  \eqref{eq:transform-bochner}.
Another routine calculation for $y\in F$ thus yields
\begin{align*}
\dualpair {y'_1} yF
=
\dualpair{ \bidualmap_F y}{y'_1}{F'}
&=
\dualpairbb{  \bidualmap_F y}{ \int_\Om  \dualpair{x_1'}XE \cdot (S\circ  Y) \intd \P }{F'} \\
&=
\int_\Om  \dualpair{x_1'}XE \cdot \dualpair{\bidualmap_F y}{S\circ  Y}{F'} \intd \P \\
&=
\int_\Om  \dualpair{x_1'}{X(\om_1)}E \cdot \dualpair{S\circ  Y(\om_1)}{y}F \intd \P(\om_1)\, ,
\end{align*}
where we used the ``Pettis property'' \eqref{eq:bochner-is-pettis} in the third step.
Combining the last identity with the definition of $\cov(X,Y)$ and \eqref{eq:bochner-is-pettis}, we then find
\begin{align} \nonumber
&\dualpair {x_2'}{Kx_1'}E \\ \nonumber
&=
\dualpair {x_2'}{\cov(X,Y)y_1'}E \\ \nonumber
&=
\dualpairbb{ x_2'}{ \int_\Om  \dualpair{y'_1}{Y(\om_2)}F \cdot X(\om_2) \intd \P(\om_2)}E \\ \nonumber
&=
\int_\Om  \dualpair{y'_1}{Y(\om_2)}F \cdot \dualpair{x_2'}{X(\om_2)}E \intd \P(\om_2)  \\ \label{lem:cov-cross-co-h666}
 &=
  \int_\Om  \int_\Om  \dualpair{x_1'}{X(\om_1)}E \cdot \dualpair{ S\circ  Y(\om_1)}{Y(\om_2)}F   
  \cdot \dualpair{x_2'}{X(\om_2)}E \intd \P(\om_1)   \intd \P(\om_2)\, .
\end{align}
Moreover, we have
\begin{align}\nonumber
g(\om_1,\om_2) &:= \Bigl| \dualpair{x_1'}{X(\om_1)}E \cdot \dualpair{ S\circ  Y(\om_1)}{Y(\om_2)}F \cdot \dualpair{x_2'}{X(\om_2)}E  \Bigr| \\ \label{lem:cov-cross-co-h1}
&\leq
\snorm{x_1'}\cdot \snorm{x_2'} \cdot \snorm S \cdot \snorm{X(\om_1)}_E \cdot \snorm{Y(\om_1)}_F \scriptonlyraw{\\ \nonumber}
\scriptonlyraw{&\qquad\qquad}   \cdot \snorm{Y(\om_2)}_F \cdot \snorm{X(\om_2)}_E 
\end{align}
for all $\om_1,\om_2\in \Om$. Now, by assumption and the Cauchy-Schwarz inequality 
we know 
$\snorm{X}_E\cdot\snorm {Y}_F \in \sLx 1 \P$. 
This gives
\begin{align} \nonumber 
&\int_\Om \int_\Om |g(\om_1,\om_2) | \intd \P(\om_1)   \intd \P(\om_2)  \\ \nonumber
&\leq 
\snorm{x_1'}\cdot \snorm{x_2'} \cdot \snorm S  \cdot \int_\Om   \snorm{X(\om_1)}_E \cdot \snorm{Y(\om_1)}_F  \intd \P(\om_1)  
\\ \nonumber  & \qquad \qquad
\cdot \int_\Om    \snorm{X(\om_2)}_E \cdot \snorm{Y(\om_2)}_F \intd \P(\om_2) \\ \label{lem:cov-cross-co-h2}
&< \infty \, ,
\end{align}
and therefore 
%
Fubini's theorem applied to \eqref{lem:cov-cross-co-h666} yields the formula for $\dualpair {x_2'}{Kx_1'}E$.

Let us now recall from   Theorem \ref{thm:alaoglu-and-more}  that $(B_{E'}, \tauweaks)$ is a compact metrizable space since we assumed that
$E$ is separable.
Consequently, it suffices to show that $x'\mapsto  \dualpair{ x'}{ Kx'}E$ is $\tauweaks$-sequentially continuous. To this end,
we fix an $x'_0 \in B_{E'}$ and a 
sequence $(x_n')\subset B_{E'}$ that converges to  $x'_0$ with respect to $\tauweaks$, that is,
$\dualpair {x_n'}xE \to \dualpair {x_0'}xE$ for all $x\in E$. 
We write
\begin{align*}
\scriptonlyraw{&}
f_n(\om_1,\om_2) \scriptonlyraw{\\}
\scriptonlyraw{&}:=  \dualpair{x_n'}{X(\om_1)}E \cdot \dualpair{ S\circ  Y(\om_1)}{Y(\om_2)}F \cdot \dualpair{x_n'}{X(\om_2)}E
\end{align*}
for all $n\geq 0$ and $\om_1,\om_2\in \Om$. Then we have $f_n \to f_0$ pointwise and \eqref{lem:cov-cross-co-h1}
shows
\begin{align*}
\scriptonlyraw{&}|f_n(\om_1,\om_2)| \scriptonlyraw{\\}
\scriptonlyraw{&}\leq \snorm S \cdot \snorm{X(\om_1)} \cdot \snorm{Y(\om_1)} \cdot \snorm{Y(\om_2)}  \cdot \snorm{X(\om_2)}  
\end{align*}
for all $n\geq 0$ and $\om_1,\om_2\in \Om$.
Since this upper bound of $f_n$ is $\P\otimes\P$-integrable by \eqref{lem:cov-cross-co-h2} and Fubini's theorem, 
we then find 
\begin{align*}
\dualpair {x_n'}{Kx_n'}E  = \int_{\Om\times \Om} f_n\intd \P^2 \to \int_{\Om\times \Om} f_0\intd \P^2 =  \dualpair {x_0'}{Kx_0'}E
\end{align*}
%
by the already established formula  for $\langle x_n', Kx_n'\rangle$ and Lebesgue's dominated convergence theorem.
\end{proof}


So far, we have defined the cross covariance operator $\cov(Y,X):E'\to F$ with the help of an
$F$-valued Bochner integral that is defined pointwise, i.e.~for all $x'\in E'$.
Our next goal is to show that the cross covariance operator can also be expressed by an operator valued Bochner integral.
Here we note that in the Hilbert space case such a representation can sometimes be found in the literature, while in the general Banach space 
case we could not find such representation. For this reason, we explain its construction in some detail.

Let us begin by  quickly recalling some basic facts about tensor products:
Let us fix
some Banach spaces $E$ and $F$. For $x'\in E'$ and $y\in F$, we  define
the bounded linear operator $x'\otimes y:E\to F$ by
\begin{align}\label{eq:def-tens-prod}
(x'\otimes y)(x) := \dualpair{x'}xE \cdot  y\, , \myqquad x\in E\, .
\end{align}
A simple calculation then shows that
\begin{align}\label{eq:norm-elem-tensor}
\snorm{x'\otimes y} = \sup_{x\in B_E} \snorm{\dualpair{x'}xE \cdot  y}_F = \sup_{x\in B_E}  |\dualpair{x'}xE| \cdot \snorm{y}_F
= \snorm  {x'}_{E'} \cdot \snorm y_F\, .
\end{align}
As usual, $x'\otimes y$ is called an elementary tensor. Elementary tensors can be added and a 
quick calculation  shows bilinearity, or more precisely
\begin{align*}
x'_1\otimes y + x'_2\otimes y &= (x'_1+x'_2)\otimes y \,, &  \a \cdot (x'\otimes y) &= (\a x')\otimes y \, ,\\
x'\otimes y_1 + x'\otimes y_2 &= x'\otimes (y_1+y_2) \, , &  \b \cdot (x'\otimes y) &= x'\otimes (\b y)
\end{align*}
for all $x',x'_1,x'_2\in E'$,  $y,y_1,y_2\in F$,  and $\a,\b\in \R$, where, as usual, we assumed that $\otimes$ has a higher priority than $+$.
We define the algebraic tensor product space by 
\begin{align*}
E'\otimes F := \spann\{ x'\otimes y: x'\in E' , y\in F \} \, .
\end{align*}
It is easy to see that
this space equals the space of all bounded linear operators $A:E\to F$ with finite rank.
Moreover, for $x'\otimes y:E\to F$, its adjoint $(x'\otimes y)':F'\to E'$ is given by
\begin{align}\label{eq:adjoint-elem-tens}
(x'\otimes y)' = (\bidualmap_F y)\otimes x' \, .
\end{align}
Indeed, for $x\in E$ and $y'\in F'$ we have
$((\bidualmap_F y)\otimes x')(y') = \dualpair{\bidualmap_F y}{ y'}{F'} \cdot x' = \dualpair{y'}yF \cdot x'$, and hence we find
\begin{align*}
\dualpairb {((\bidualmap_F y)\otimes x')(y')} x E
=
\dualpairb { \dualpair{y'}yF \cdot x'} x E
&=
\dualpair {y'}yF \cdot \dualpair {x'}xE\\
&=
\dualpairb {y'} {\dualpair {x'}xE \cdot y}F \\
&=
\dualpair{y'}{(x'\otimes y)(x)}F \, .
\end{align*}
A bounded linear operator $A:E\to F$ is called nuclear, if there exist sequences $(x_n')\subset E'$ and $(y_n)\subset F$ with
\begin{align}\label{def:nuc-op}
\sum_{n=1}^\infty \snorm{x_n'}_{E'} \cdot \snorm{y_n}_F < \infty
\myqquad \mbox{ and } \myqquad
Tx = \sum_{n=1}^\infty \dualpair {x_n'}{x}{E} \cdot y_n
\end{align}
for all $x\in E$, where the convergence in the second series is in $F$. 
Note that by \eqref{eq:norm-elem-tensor} the first condition is equivalent to $(\snorm{x'_n\otimes y_n})\in \ell_1$, and therefore we have
$T = \sum_{n=1}^\infty x_n'\otimes y_n$, where the series is absolutely convergent with respect to the operator norm.
Finally, the nuclear norm of $T$ is defined by
\begin{align}\label{eq:nuc-norm}
\nucnorm T := \inf\biggl\{  \sum_{n=1}^\infty \snorm{x_n'}_{E'} \cdot \snorm{y_n}_F : (x_n')\subset E' \mbox{ and } (y_n)\subset F \mbox{ satisfy }
\eqref{def:nuc-op}\biggr\} \, .
\end{align}
It is well-known that the set $\blonspace EF$ of nuclear operators $T:E\to F$ equipped with $\nucnorm\mycdot$ is a Banach space
and
we have $\snorm T\leq \nucnorm T$ for all $T\in \blonspace EF$, see e.g.~\cite[Satz VI.5.3]{Werner11} or \cite[Theorem 5.25]{DiJaTo95}.
Moreover,
 every finite rank operator is nuclear, and in particular all elementary tensors  are nuclear.
In addition, for $x'\in E'$ and $y\in F$ we have
 \begin{align*}
\snorm  {x'}_{E'} \cdot \snorm y_F = \snorm{x'\otimes y} \leq \nucnorm{x'\otimes y} \leq \snorm  {x'}_{E'} \cdot \snorm y_F
 \end{align*}
where in the last step we used that $x'\otimes y$ is a possible representation  \eqref{def:nuc-op} of $T:= x'\otimes y$. Consequently, we have
\begin{align}\label{eq:elem-tens-nuc-norm}
\snorm{x'\otimes y} = \nucnorm{x'\otimes y} = \snorm  {x'}_{E'} \cdot \snorm y_F
\end{align}
for all $x'\in E'$ and $y\in F$. In particular, for $x_1',x_2'\in E'$ and $y_1,y_2\in F$ with $\snorm{x_1'}_{E'}\leq \snorm{x_2'}_{E'}$
and $\snorm{y_1}_F\leq \snorm{y_2}_F$ we have
\begin{align}\label{eq:elem-tens-dom-conv}
\nucnorm{x_1'\otimes y_1}  \leq \nucnorm{x_2'\otimes y_2} \, .
\end{align}
The next lemma establishes a simple yet useful convergence result for elementary tensors.

\begin{lemma}\label{lem:elem-tens-conv}
Let $E$ and $F$ be Banach spaces, $x'\in E'$, $y\in F$, and $(x_n')\subset E'$ and $(y_n)\subset F$ be sequences
with $x_n'\to x'$ and $y_n\to y$. 
Then we have
\begin{align*}
\nucnorm{x'_n\otimes y_n - x'\otimes y} \to 0 \, . 
\end{align*}
\end{lemma}

\begin{proof}[Proof of Lemma \ref{lem:elem-tens-conv}]
For the nuclear operator $T_n:= x'_n\otimes y_n - x'\otimes y$ we have the alternative representation
$T_n= (x'_n-x')\otimes y_n - x'\otimes (y-y_n)$ in the sense of \eqref{def:nuc-op}. By the definition of the nuclear norm we hence obtain
\begin{align*}
\nucnorm{x'_n\otimes y_n - x'\otimes y} \leq \snorm{x'_n-x'}_{E'} \cdot \snorm {y_n}_F + \snorm{x'}_{E'}\cdot \snorm{y-y_n}_F \to 0\, ,
\end{align*}
since the convergence of $(y_n)$ ensures that the sequence $(\snorm{y_n}_F)$ is bounded.
\end{proof}

Given two functions  $X':\Om\to E'$ and $Y:\Om\to  F$ we define the $E'\otimes F$-valued function $X'\otimes Y:\Om\to E'\otimes F$  pointwise by elementary tensors, that is 
\begin{align*}
(X'\otimes Y)(\om) := (X'(\om))\otimes (Y(\om))\, , \myqquad \om \in \Om.
\end{align*}
Under suitable assumptions, it turns out that the Bochner integral of tensor space valued functions can be used to compute cross covariances.
To establish this result we need the following elementary lemma, which investigates such functions in the case of $X$ or $Y$ being constant.

\begin{lemma}\label{lem:bochner-int-for-tensors}
Let $(\Om,\sA,\P)$ be a probability space, $E$ and $F$ be   Banach spaces, and $X'\in \sLx 1 {\P,E'}$ and $Y\in \sLx 1 {\P,F}$. 
Then for all $x'\in E'$ and $y\in F$ we have $X'\otimes y\in \sLx 1 {\P,\blonspace {E}F}$ and $x'\otimes Y\in \sLx 1 {\P,\blonspace {E}F}$ 
with 
\begin{align*}
\int_\Om X'\otimes y \intd \P =\biggr(\int_\Om X'\intd \P \biggl) \otimes y 
\myqquad \mbox{and} \myqquad 
\int_\Om x'\otimes Y \intd \P =  x' \otimes \biggr(\int_\Om Y\intd \P \biggl) \, .
\end{align*}
Here, the integrals of $X'\otimes y$ and $x'\otimes Y$ are $\blonspace {E}F$-valued Bochner integrals and the integrals of $X'$ and $Y$ are $E'$-valued, respectively $F$-valued Bochner integrals.
\end{lemma}

\begin{proof}[Proof of Lemma \ref{lem:bochner-int-for-tensors}]
We only prove the assertion for $X'\otimes y$, since the proof for $x'\otimes Y$ is analogous.

Let us first assume that  $X'$ is a measurable step function. We fix a representation
$X=\sum_{i=1}^m \eins_{A_i} x_i'$, where $A_i \in \sA$ and $x_i'\in E'$.
For $\om \in \Om$ the bilinearity of elementary tensors then gives
\begin{align*}
(X'\otimes y)(\om)
=  (X'(\om))\otimes y
= \biggl(  \sum_{i=1}^m \eins_{A_i}(\om) x_i' \biggr)  \otimes  y
=  \sum_{i=1}^m  \eins_{A_i}(\om) (x_i'\otimes y) \, .
\end{align*}
Consequently, $X'\otimes y:\Om \to \blonspace {E}F$
is a measurable step function, and therefore we  have  $X'\otimes y \in \sLx 1 {\P, \blonspace {E}F}$, see also the discussion in front of \eqref{eq:bochner-1-norm}.
Moreover, the construction of the Bochner integral for step functions \eqref{def:step-fct-bochner-int} shows
\begin{align*}
 \int_\Om X'\otimes y \intd \P
=  \sum_{i=1}^m   \P({A_i}) \cdot   (x_i'\otimes y)   
=  \biggl(\sum_{i=1}^m   \P({A_i}) \cdot   x_i' \biggr) \otimes y
=\biggr(\int_\Om X'\intd \P \biggl) \otimes y \, .
\end{align*}
Let us now consider the general case $X'\in \sLx 1 {\P,E'}$. Since $X'$ is strongly measurable, 
 there exists a sequence $(X'_n)$  of measurable step functions approximating
$X'$ in the sense of \eqref{eq:step-fct-approx-1} and \eqref{eq:step-fct-approx-2}.  
 For each fixed $\om\in \Om$, Lemma \ref{lem:elem-tens-conv} in combination with \eqref{eq:step-fct-approx-1} then
 shows
\begin{align}\label{lem:bochner-int-for-tensors-h666}
\nucnorm{(X_n'\otimes y)(\om)   -(X'\otimes y)(\om)} \to 0 \, .
\end{align}
Furthermore, \eqref{eq:step-fct-approx-2} and \eqref{eq:elem-tens-dom-conv} give $\nucnorm{(X'_n\otimes y)(\om)}\leq   \nucnorm{(X'\otimes y)(\om)}$, and  
 we already know that each $X'_n\otimes y:\Om\to \blonspace {E}F$ is a measurable step function.
 In summary, $(X_n'\otimes y)$ is a sequence approximating $X'\otimes y$ in $\blonspace EF$
 in the sense of \eqref{eq:step-fct-approx-1} and \eqref{eq:step-fct-approx-2}.
Consequently, $X'\otimes y:\Om\to \blonspace {E}F$ is strongly measurable as discussed around \eqref{eq:step-fct-approx-1} and \eqref{eq:step-fct-approx-2}.
Moreover, with the help of \eqref{eq:elem-tens-nuc-norm} we obtain
\begin{align*}
\int_\Om \nucnorm{(X'\otimes y)(\om)} \intd \P(\om)
=  \int_\Om \snorm{X'(\om)}_{E} \cdot \snorm {y}_F \intd \P(\om) < \infty\, .
\end{align*}
In other words we have found $X'\otimes y \in \sLx 1 {\P, \blonspace {E}F}$.
Finally, with convergence in  $\blonspace {E}F$  we have 
\begin{align*} 
\int_\Om X'\otimes y \intd \P  
=
\lim_{n\to \infty}  \int_\Om X'_n\otimes y \intd \P 
=
\lim_{n\to \infty}  \biggr(  \biggr(\int_\Om X'_n\intd \P \biggr) \otimes y \biggl) 
&=
\biggr(\lim_{n\to \infty}   \int_\Om X'_n\intd \P \biggr) \otimes y \\
&=
\biggr(\int_\Om X'\intd \P \biggl) \otimes y \, ,
\end{align*}
where the first and last step   used the construction of Bochner integrals \eqref{def:bochner-integral},
the second step used the already verified assertion in the case of step functions, and the third step used 
Lemma \ref{lem:elem-tens-conv}.
\end{proof}


Let us now turn to the announced integral representation of cross covariance operators. 
To this end, we fix two Banach spaces $E$ and $F$, as well as $x\in E$ and $y\in F$. For $x'' := \bidualmap_E x\in E''$, where 
$\bidualmap_E:E\to E''$ denotes the canonical, isometric embedding of $E$ into its bidual $E''$, 
we then write 
\begin{align*}
x\otimes y := x''\otimes y = (\bidualmap_E x)\otimes y\, .
\end{align*}
Obviously,
$x\otimes y\in E''\otimes F$ is the bounded linear map $E'\to F$ given by
\begin{align*}
(x\otimes y)(x') = \dualpair {x'}xE \cdot y\, , \myqquad x'\in E'.
\end{align*}
By construction the bilinearity is preserved, and since $\snorm x_E = \snorm{\bidualmap_E x}_{E''} = \snorm{x''}_{E''}$ holds true,
we quickly verify that \eqref{eq:elem-tens-nuc-norm} and \eqref{eq:elem-tens-dom-conv} read as 
\begin{align}\label{eq:elem-tens-nuc-norm-EF}
\snorm{x\otimes y} & = \nucnorm{x\otimes y} = \snorm  {x}_{E} \cdot \snorm y_F \, , \\ \label{eq:elem-tens-dom-conv-EF}
\nucnorm{x_1\otimes y_1}  & \leq \nucnorm{x_2\otimes y_2} 
\end{align}
for all $x,x_1,x_2\in E$ and $y,y_1,y_2\in F$ with $\snorm{x_1}_{E}\leq \snorm{x_2}_{E}$
and $\snorm{y_1}_F\leq \snorm{y_2}_F$. Moreover, Lemma \ref{lem:elem-tens-conv} also holds true, since $x_n\to x$ in $E$ implies $\bidualmap_E x_n\to \bidualmap_E x$ in $E''$.

Finally, Lemma \ref{lem:bochner-int-for-tensors} remains true: 
Indeed, if we have for example an $X\in   \sLx 1 {\P,E}$, then $X'':= \bidualmap_E \circ X\in \sLx 1 {\P,E''}$ as discussed around \eqref{eq:transform-bochner},
and therefore Lemma 
 \ref{lem:bochner-int-for-tensors} gives $X\otimes y = X''\otimes y \in \sLx 1 {\P,\blonspace {E'}F}$ with 
 \begin{align} \nonumber
 \int_\Om X\otimes y \intd \P 
 =
 \int_\Om X''\otimes y \intd \P 
 =
 \biggr(\int_\Om X''\intd \P \biggl) \otimes y 
  &=
 \biggr(\int_\Om \bidualmap_E \circ X \intd \P \biggl) \otimes y \\ \nonumber
  & =
 \biggl(\bidualmap_E\biggr(\int_\Om X \intd \P \biggl)\biggr) \otimes y \\ \label{lem:bochner-int-for-tensors-EF}
 & =\biggl( \int_\Om X \intd \P  \biggr) \otimes y
 \end{align}
by an  application of \eqref{eq:transform-bochner} and our nomenclature for tensors of the form $x\otimes y$. In particular, the Bochner integral of $X$ 
in \eqref{lem:bochner-int-for-tensors-EF}
is an $E$-valued 
Bochner integral and not just an $E''$-valued integral as a na\"ive application of  Lemma \ref{lem:bochner-int-for-tensors}  would suggest.

With these preparation we can now present an alternative way to compute the cross covariance operator.

\begin{lemma}\label{lem:cross-cov-as-integral}
Let $(\Om,\sA,\P)$ be a probability space, $E$ and $F$ be   Banach spaces, and $X\in \sLx 2 {\P,E}$ and $Y\in \sLx 2 {\P,F}$.
Then the map $X\otimes Y:   \Om\to \blonspace {E'}F$ 
is strongly measurable with $X\otimes Y \in \sLx 1 {\P, \blonspace {E'}F}$. Moreover, $\cov (Y,X)$ can be computed by the
$\blonspace {E'}F$-valued Bochner integral
\begin{align}\label{lem:cross-cov-as-integral-form}
\cov (Y,X) = \int_\Om X\otimes Y \intd \P  - (\E X)\otimes (\E Y) \, .
\end{align}
\end{lemma}

\begin{proof}[Proof of Lemma \ref{lem:cross-cov-as-integral}]
Our first goal is to show that it suffices to consider centered $X$ and $Y$. 
To this end, we pick some $X\in \sLx 2 {\P,E}$ and $Y\in \sLx 2 {\P,F}$ and define 
$X_0:= X-\E X$ and $Y_0 := Y-\E Y$. Furthermore, we assume that the assertion is true for $X_0$ and $Y_0$.
Our goal is therefore to show that the assertion is also true for $X$ and $Y$.
To this end, we first observe that the 
bilinearity of tensors   gives
\begin{align*}
X_0 \otimes Y_0 = (X-\E X) \otimes (Y-\E Y) = X\otimes Y - X\otimes (\E Y)  - (\E X) \otimes Y + (\E X)\otimes (\E Y) \, .
\end{align*}
Now, our assumption guarantees $X_0\otimes Y_0 \in \sLx 1 {\P, \blonspace {E'}F}$, and Lemma \ref{lem:bochner-int-for-tensors} 
ensures  both 
$X\otimes (\E Y) \in \sLx 1 {\P, \blonspace {E'}F}$ and $(\E X) \otimes Y \in \sLx 1 {\P, \blonspace {E'}F}$.
Finally, $ (\E X)\otimes (\E Y)\in \blonspace {E'}F$ does not depend on $\om$. In summary, we thus have 
$X\otimes Y \in \sLx 1 {\P, \blonspace {E'}F}$.
Moreover, our calculation above yields 
\begin{align*} 
\cov (Y,X)
&=
\int_\Om X_0\otimes Y_0 \intd \P   \\
&= 
\int_\Om  X\otimes Y - X\otimes (\E Y)  - (\E X) \otimes Y + (\E X)\otimes (\E Y) \intd \P   \\
&=
\int_\Om  X\otimes Y  \intd \P - \int_\Om   X\otimes (\E Y)  \intd \P - \int_\Om (\E X) \otimes Y \intd \P + (\E X)\otimes (\E Y) \\
&= 
\int_\Om  X\otimes Y  \intd \P - (\E X)\otimes (\E Y)  - (\E X)\otimes (\E Y) + (\E X)\otimes (\E Y) \\
&= 
\int_\Om  X\otimes Y  \intd \P - (\E X)\otimes (\E Y) \, ,
\end{align*}
where in the second to last step we used Lemma \ref{lem:bochner-int-for-tensors} in the form of \eqref{lem:bochner-int-for-tensors-EF} twice.

Let us now turn to the actual proof. 
Here, we begin by considering the 
 case, in which $X$ and $Y$ are measurable step functions. As discussed above,   we may assume without loss of generality that 
 $\E X = 0$ and $\E Y = 0$.
 We now fix some    representations
$X=\sum_{i=1}^m \eins_{A_i} x_i$ and $Y=\sum_{j=1}^n \eins_{B_j} y_j$, where $A_i,B_j\in \sA$, $x_i\in E$, and $y_j\in F$.
For $\om \in \Om$, 
the bilinearity of elementary tensors then gives
\begin{align*}
(X\otimes Y)(\om)
= \biggl(  \sum_{i=1}^m \eins_{A_i}(\om) x_i \biggr)  \otimes  \biggl(  \sum_{j=1}^n \eins_{B_j}(\om) y_j \biggr)
=  \sum_{i=1}^m \sum_{j=1}^n \eins_{A_i\cap B_j}(\om) (x_i\otimes y_j) \, .
\end{align*}
Consequently, $X\otimes Y:\Om \to \blonspace {E'}F$
is a measurable step function, and therefore we  have  $X\otimes Y \in \sLx 1 {\P, \blonspace {E'}F}$, see also the discussion in front of \eqref{eq:bochner-1-norm}.
Finally, for $x'\in E'$ the construction of the Bochner integral for step functions \eqref{def:step-fct-bochner-int} shows
\begin{align*}
\biggl( \int_\Om X\otimes Y \intd \P \biggr) (x')
&= \biggl( \sum_{i=1}^m \sum_{j=1}^n \P({A_i\cap B_j}) \cdot   (x_i\otimes y_j)   \biggr) (x') \\
&=  \sum_{i=1}^m \sum_{j=1}^n \P({A_i\cap B_j}) \cdot   (x_i\otimes y_j)  (x') \\
&=  \sum_{i=1}^m \sum_{j=1}^n \P({A_i\cap B_j}) \cdot   \dualpair {x'}{x_i}E \cdot y_j\, ,
\end{align*}
as well as
\begin{align*}
\cov(Y,X)(x')
= \int_\Om \dualpair {x'} {X } E \, Y \intd \P 
&= \int_\Om \biggl(\sum_{i=1}^m \eins_{A_i}\dualpair {x'} {x_i } E\biggr) \, \biggr(\sum_{j=1}^n \eins_{B_j} y_j\biggl) \intd \P \\
&= \sum_{i=1}^m \sum_{j=1}^n \int_\Om \eins_{A_i\cap B_j} \dualpair {x'} {x_i } E \cdot y_j \intd \P \\
&=  \sum_{i=1}^m \sum_{j=1}^n \P({A_i\cap B_j}) \cdot   \dualpair {x'}{x_i}E \cdot y_j \, .
\end{align*}
Combining both calculations yields \eqref{lem:cross-cov-as-integral-form}.

Let us finally consider the   case of centered $X\in \sLx 2 {\P,E}$ and $Y\in \sLx 2 {\P,F}$.
Since $X$ and $Y$ are strongly measurable,  there then  exist sequences $(X_n)$ and $(Y_n)$ of measurable step functions approximating
$X$, respectively $Y$ in the sense of \eqref{eq:step-fct-approx-1} and \eqref{eq:step-fct-approx-2}. 
 For each fixed $\om\in \Om$, Lemma \ref{lem:elem-tens-conv} in combination with \eqref{eq:step-fct-approx-1} for $X_n$ and $Y_n$ then
 shows
\begin{align*}
\nucnorm{(X_n\otimes Y_n)(\om)   -(X\otimes Y)(\om)} \to 0 \, .
\end{align*}
Furthermore, \eqref{eq:step-fct-approx-2} and \eqref{eq:elem-tens-dom-conv-EF} gives $\nucnorm{(X_n\otimes Y_n)(\om)}\leq   \nucnorm{(X\otimes Y)(\om)}$ 
and we already know that each $X_n\otimes Y_n:\Om\to \blonspace {E'}F$ is a measurable step function.
In summary, $(X_n'\otimes Y_n)$ is therefore a sequence approximating $X'\otimes Y$ in $\blonspace EF$
 in the sense of \eqref{eq:step-fct-approx-1} and \eqref{eq:step-fct-approx-2}.
Consequently, $X\otimes Y$ is strongly measurable as discussed around \eqref{eq:step-fct-approx-1} and \eqref{eq:step-fct-approx-2}.
Moreover, with the help of \eqref{eq:elem-tens-nuc-norm-EF}   we obtain
\begin{align*}
\int_\Om \nucnorm{(X\otimes Y)(\om)} \intd \P(\om)
=  \int_\Om \snorm{X(\om)}_{E} \cdot \snorm {Y(\om)}_F \intd \P(\om) < \infty\, ,
\end{align*}
where in the last step we used the Cauchy-Schwarz inequality. In other words we have found $X\otimes Y \in \sLx 1 {\P, \blonspace {E'}F}$.

To establish \eqref{lem:cross-cov-as-integral-form}, we first note that  \eqref{eq:step-fct-approx-1} and \eqref{eq:step-fct-approx-2}
for $X$ and $(X_n)$ give $\E X_n \to \E X = 0$ by the definition of Bochner integrals and our assumption that $X$ is centered.
Analogously, we find $\E Y_n \to 0$. Combining both gives $\nucnorm{\E X_n \otimes \E Y_n} \to 0$ by Lemma \ref{lem:elem-tens-conv}.
For $x'\in E'$ we thus find
\begin{align}\nonumber 
\biggl( \int_\Om X\otimes Y \intd \P \biggr) (x')
&=
\biggl(\lim_{n\to \infty}  \int_\Om X_n\otimes Y_n \intd \P \biggr) (x') \\ \nonumber
&=
\lim_{n\to \infty} \Biggl(  \biggl(  \int_\Om X_n\otimes Y_n \intd \P \biggr) (x') \Biggr) \\ \nonumber
&=
\lim_{n\to \infty} \Biggl(  \biggl(  \int_\Om X_n\otimes Y_n \intd \P \biggr) (x')-  (\E X_n \otimes \E Y_n)(x')  \Biggr) \\ \label{lem:cross-cov-as-integral-h1}
&=
\lim_{n\to \infty} \Bigl( \cov (Y_n,X_n) (x')  \Bigr) \, ,
\end{align}
where in the first step we used the definition of Bochner integrals, the second step used that the limit in the second expression
is with respect to the nuclear norm and thus also with respect to the operator norm, and in the last step
we used the already established identity \eqref{lem:cross-cov-as-integral-form} for step functions.
Moreover, we also have
\begin{align}\nonumber
\lim_{n\to \infty} \Bigl( \cov (Y_n,X_n) (x')  \Bigr)
=
\lim_{n\to \infty} \int_\Om \dualpair {x'} {X_n } E \cdot Y_n \intd \P
&= 
\int_\Om \dualpair {x'} {X } E \, Y \intd \P  \\ \label{lem:cross-cov-as-integral-h2}
&=
\cov(Y,X)(x')
\, ,
\end{align}
where the second step is based upon the dominated convergence property
discussed around \eqref{eq:meas-fct-approx-1} to \eqref{eq:dominated-conv-bnochner} with  the $\P$-integrable 
$g:\Om\to [0,\infty]$ defined by
$g(\om) := \snorm {x'}_{E'} \cdot \snorm{X(\om)}_E\cdot  \snorm{Y(\om)}_F$ and the convergence
\begin{align*}
\dualpair {x'} {X_n(\om) } E \cdot Y_n(\om) \to \dualpair {x'} {X(\om) } E \cdot Y(\om) 
\end{align*}
for all $\om \in \Om$.
Combining \eqref{lem:cross-cov-as-integral-h1} and \eqref{lem:cross-cov-as-integral-h2} yields \eqref{lem:cross-cov-as-integral-form}.
%
\end{proof}

%% file: appendix-cf+ft.tex
\section{Weak Convergence, Characteristic Functions, and Gaussian Random Variables}\label{app:cf}

In this supplement we mostly collect and unify known results on the objects listed in the title. 
In addition, a few results we could not find in the literature are presented and proved.

Let us begin by recalling 
characteristic functions in the Banach space case. To this end, 
 let $E$ be a separable Banach space and $\P$ be a 
probability measure on $\sborelnormx E = \sborelEx E$, see Theorem \ref{thm:pettis-var}. Then the characteristic function $\p_\P:E'\to \Cfield$ of $\P$ is defined by 
\begin{align*}
\p_\P(x') := \int_E \eul^{\imi \dualpair {x'}x E} \intd\P(x)\, , \myqquad x'\in E',  
\end{align*}
see e.g.~\cite[Chapter IV.2.1]{VaTaCh87}. Like in the finite dimensional case, the characteristic function can be used to identify 
probability measures. Namely, if $\P$ and $\Qm$ 
are probability measures with $\p_\P = \p_\Qm$, then we have $\P=\Qm$, see e.g.~\cite[Theorem 2.2 in Chapter IV.2.1]{VaTaCh87}.
Finally, if $X$ is an $E$-valued random variable  with distribution $\P_X$, we write $\p_X := \p_{\P_X}$ and call $\p_X$ the characteristic function of $X$.

 
%
%
%
%

Recall that characteristic functions of affine linear transformations of random variables can
be easily computed. Here we only mention the finite dimensional case.

\begin{lemma}\label{lem:char-func-d-prop}
Let $(\Om,\sA, \P)$ be a probability space, $X:\Om\to \Rd$ be measurable,
$A$ be an $m\times d$-matrix,  $a\in \Rd$, and $b\in \R^m$.
Then, for all $s\in \R$ and $t\in \R^m$, the following two identities hold true
\begin{align*}
\p_{\skprod aX}(s) &= \p_X(sa)\,,  \\
\p_{AX+b}(t) &= \exp( \imi \skprod tb ) \cdot \p_X(A\transpose t)\, .
\end{align*}
\end{lemma}

\begin{proof}[Proof of Lemma \ref{lem:char-func-d-prop}]
The second identity can  be found in e.g.~\cite[Theorem 13.3]{JaPr04}, 
and the first identity follows from the second one since $\skprod aX = a\transpose X$ and
$as = sa$, where $as$ denotes the matrix product and $a$ are $s$ are viewed as  $d\times 1$ and $1\times 1$-matrices, respectively.
\end{proof}

Let us now  recall the notion of weak convergence of probability measures. To this end, let 
 $E$ be a Banach space, $(\P_n)$ be a sequence of probability measures on $\sborelnorm E$, and $\P$
be another probability measure on $\sborelnorm E$. Then   $(\P_n)$ converges weakly to $\P$, if for all 
bounded and continuous $f:E\to \R$ we have 
\begin{align}\label{def:weak-converg-of-meas}
\int_E f \intd \P_n \to \int_E f \intd \P  \, .
\end{align}
Like in the finite dimensional case it turns out that the limit $\P$ is unique, see e.g.~\cite[Chapter I.3.7]{VaTaCh87}.
Moreover,  we say that 
a sequence $(X_n)$ of $E$-valued random variables converges to an $E$-valued random
variable $X$ in distribution, if the distributions of $(X_n)$  converge weakly to the distribution of $X$.
 
 The following result, whose implication \emph{iii)} $\Rightarrow$ \emph{i)} is 
known as L\'evy's continuity theorem relates convergence in distribution to 
pointwise convergence of the characteristic functions.

\begin{theorem}\label{thm:levy-continuity} 
Let $(X_n)$ be a sequence of $\Rd$-valued random variables  and $X$ be another $\Rd$-valued random variable. Then the following statements are equivalent:
\begin{enumerate}
\item We have $X_n \to X$ in distribution.
\item For all $a\in \Rd$  with $\snorm a_2 = 1$ we have $\skprod a{X_n} \to \skprod aX$ in distribution.
\item We have $\p_{X_n}(t) \to \p_X(t)$ for all $t\in \Rd$.
\end{enumerate}
\end{theorem}

\begin{proof}[Proof of Theorem \ref{thm:levy-continuity}]
\atob i {iii} Follows immediately from the definition of convergence in distribution and from the form of  characteristic functions.

\atob {iii} i See, for example, \cite[Theorem 9.8.2]{Dudley02} or \cite[Theorem 15.23]{Klenke14}.

\aeqb {ii} {iii} We first note that $\skprod a{X_n} \to \skprod aX$ in distribution is equivalent to $\p_{\skprod a{X_n}}(s) \to \p_{\skprod a{X}}(s)$ 
for all $s\in \R$ by the already established equivalence of \emph{i)} and \emph{iii)}. Now the assertion
quickly follows by applying the first formula  of  Lemma \ref{lem:char-func-d-prop}.
%
%
%
%
\end{proof}

The next result applies Theorem \ref{thm:levy-continuity} to $\R$-valued Gaussian random variables.

\begin{theorem}\label{thm:levy-grv}
Let $(X_n)$ be a sequence of $\R$-valued Gaussian random variables and $X$ be an $\R$-valued random variable with 
 $X_n \to X$ in distribution. Then $X$ is a Gaussian random variable and we have:
 \begin{align*}
 \E X_n &\to \E X \, , \\
 \var X_n &\to \var X \, .
 \end{align*}
\end{theorem}

\begin{proof}[Proof of Theorem \ref{thm:levy-grv}]
Let us write $\mu_n := \E X_n$ and $\s_n := \sqrt{\var X_n}$. By Theorem \ref{thm:levy-continuity}
we then know that 
\begin{align}\label{cor:levy-grv-h0}
 \eul^{\imi t \mu_n} \cdot \exp\Bigl( - \frac{\s_n^2 t^2}{2} \Bigr) = \p_{X_n}(t) \to \p_X(t) \, , \qquad \qquad t\in \R.
\end{align}
Taking the absolute value in $\Cfield$ yields
\begin{align}\label{cor:levy-grv-h1}
 \exp\Bigl( - \frac{\s_n^2 t^2}{2} \Bigr) = |\p_{X_n}(t)| \to |\p_X(t)| =: a_t \, , \qquad \qquad t\in \R.
\end{align}
Obviously, we have $a_0 = |\p_X(t)| = 1$. 
Since $\p_X$ is uniformly continuous, see e.g.~\cite[Theorem 15.21]{Klenke14}, there thus exists a $t_0>0$ such that $a_t > 0$ for all $t\in [-t_0,t_0]$.
Consequently, we have 
\begin{align}\label{cor:levy-grv-h3}
\s^2_n \to \s^2 := - \frac 2 {t_0^2} \, \ln(a_{t_0})\, ,
\end{align}
and in particular, we find  $b:= 1+\sup_{n}\s_n< \infty$.

Our next step is to show that $(\mu_n)$ is a bounded sequence. To this end,
let us first assume that we had a subsequence $\mu_{n_k}\to \infty$. 
Since $\PX$ is a probability
measure there exists  an $a>0$ with $\PX([-a,a]) \geq 1/2$. 
Clearly, we may assume without loss of generality that $\mu_{n_k}> a+1$ for all $k\geq 1$.
Let us fix a continuous function $f:\R\to [0,1]$ with $f(t) = 1$ for all $t\in [-a,a]$ and $f(t) = 0$ for all 
$t\not \in [-a-1, a+1]$. By the assumed convergence $X_n \to X$ in distribution, we then obtain 
\begin{align}\label{cor:levy-grv-h2}
\int_\R f\intd \P_{X_{n_k}} \to \int_\R f\intd \P_{X} \geq \int_{[-a,a]} f\intd \P_{X} \geq 1/2\, .
\end{align}
To find a contradiction, let us first assume that there is an $n_k$ with
 $\s_{n_k} = 0$. This gives $\P_{X_{n_k}} = \dirac{\mu_{n_k}}$, and 
 since $\mu_{n_k}> a+1$ 
 our construction of $f$ would then yield
\begin{align*}
\int_\R f\intd \P_{X_{n_k}} = f(\mu_{n_k}) = 0\, .
\end{align*}
In view of \eqref{cor:levy-grv-h2} we thus conclude that there exists a $k_0\geq 1$ with 
$\s_{n_k} > 0$ for all $k\geq k_0$. 
For such $k$, however, we find
\begin{align*}
\int_\R f\intd \P_{X_{n_k}}
& = 
\frac {1}{\sqrt{2\pi \s_{n_k}^2}}\int_\R f(t) \exp\Bigl( -\frac{(t-\mu_{n_k})^2}{2\s_{n_k}^2}   \Bigr) \intd t \\
& = 
\frac {1}{\sqrt{2\pi \s_{n_k}^2}}\int_\R f(t + \mu_{n_k}) \exp\Bigl( -\frac{t^2}{2\s_{n_k}^2}   \Bigr) \intd t \\
& = 
\frac {1}{\sqrt{2\pi }}\int_\R f(\s_{n_k}t + \mu_{n_k}) \exp\Bigl( -\frac{t^2}{2}   \Bigr) \intd t \\
& = 
\frac {1}{\sqrt{2\pi }}\int_{-\infty}^{\s_{n_k}^{-1}(a+1-\mu_{n_k})} f(\s_{n_k}t + \mu_{n_k}) \exp\Bigl( -\frac{t^2}{2}   \Bigr) \intd t \\
& \leq 
\gauss 01 \bigl((-\infty, \s_{n_k}^{-1}(a+1-\mu_{n_k})] \bigr) \, ,
\end{align*}
where in the second to last step we used that $\s_{n_k} t+\mu_{n_k}\leq a+1$ is equivalent to $t\leq \s^{-1}_{n_k}(a+1-\mu_{n_k})$.
Now observe that the definition of $b$ implies $\s_{n_k} \leq b$ and thus 
$\s_{n_k}^{-1} \geq b^{-1}>0$ for all $k\geq k_0$. Since we also have 
$a+1-\mu_{n_k}<0$ for all $k\geq 1$, we thus conclude
\begin{align*}
\s_{n_k}^{-1}(a+1-\mu_{n_k}) \leq b^{-1}(a+1-\mu_{n_k})
\end{align*}
for all $k\geq k_0$. Consequently, the estimation above can be continued to 
\begin{align*}
\int_\R f\intd \P_{X_{n_k}} 
\leq \gauss 01 \bigl((-\infty, \s_{n_k}^{-1}(a+1-\mu_{n_k})] \bigr) 
&\leq  \gauss 01 \bigl((-\infty, b^{-1}(a+1-\mu_{n_k})] \bigr) \\
& \to 0
\end{align*}
for $k\to \infty$, where in the last step we used $\mu_{n_k}\to \infty$.
Clearly, this contradicts \eqref{cor:levy-grv-h2} and hence there is no
subsequence $\mu_{n_k}\to \infty$. Analogously, we find that there is no subsequence
$\mu_{n_k}\to -\infty$. In other words, $(\mu_n)$ is a bounded sequence.

Now, since $(\mu_n)$ is a bounded sequence, there exist a $\mu\in \R$ and a subsequence $(\mu_{n_k})$
with $\mu_{n_k}\to \mu$, and this in turn yields 
\begin{align}\label{cor:levy-grv-h4}
\p_{X_{n_k}}(t) 
= \eul^{\imi t \mu_{n_k}} \cdot \exp\Bigl( - \frac{\s_{n_k}^2 t^2}{2} \Bigr)  
\to
\eul^{\imi t \mu} \cdot \exp\Bigl( - \frac{\s^2 t^2}{2} \Bigr) 
\end{align}
for all $t\in \R$.
Since the limit function is the characteristic function of $\gauss \mu {\s^2}$ 
we see by \eqref{cor:levy-grv-h0} and the uniqueness of the characteristic function
 that $\PX = \gauss \mu {\s^2}$, that is, $X$ is indeed a Gaussian random variable.

In addition, $\var X_n \to \var X$
immediately follows from \eqref{cor:levy-grv-h3}.
Finally, we already know that $\E X_{n_k} =\mu_{n_k} \to \mu= \E X$.
If, however, the full sequence $(\E X_{n})$ did not converge to  $\E X$, then there would be another subsequence $(\mu_{m_k})$  
and an $\e>0$ with $|\mu_{m_k} - \mu|\geq \e$ for all $k\geq 1$. Since $(\mu_n)$ is bounded, so is
$(\mu_{m_k})$, and hence there would exist a subsequence of $(\mu_{m_k})$ converging to some $\tilde \mu\in \R$. Our construction
ensures $\mu\neq \tilde \mu$ and repeating the arguments around \eqref{cor:levy-grv-h4} would give us 
$\PX = \gauss{\tilde \mu}{\s^2}$. This, however, is 
%
impossible by the uniqueness of the weak limit of probability measures, see discussion below \eqref{def:weak-converg-of-meas}.
\end{proof}

Let us now turn to general $E$-valued Gaussian random variables.
The first result in this direction  shows that the Gaussianity of an $E$-valued random variable can be checked with the help 
of $\tauweaks$-dense   subsets $\denseblo$ of $B_E'$, where we recall from Theorem \ref{thm:alaoglu-and-more} that there exist such sets that are even countable.

\begin{theorem}\label{thm:test-for-gms}
Let $E$ be a separable Banach space 
and $\denseblo \subset B_{E'}$ be  $\tauweaks$-dense in $B_{E'}$.
Then a probability measure $\P$ on $\sborelnorm E = \sborelEx E$
is a Gaussian measure, if and only if the image measure $\P_{x'}$ 
is a Gaussian measure on $\sborel$  for all $x'\in \denseblo$.

Moreover, an $E$-valued  random variable $X$ is a Gaussian random variable, if and only if 
$\dualpair {x'}XE$ is an $\R$-valued  Gaussian random variable for all $x'\in \denseblo$.
\end{theorem}

\begin{proof}[Proof of Theorem \ref{thm:test-for-gms}]
Let us first consider the case of an $E$-valued  random variable $X:\Om\to E$. 
We begin by assuming that $\dualpair {x'}XE$ is a Gaussian random variable  for all $x'\in \denseblo$.
Let us now fix an  $x'\in E'$. Since the relative $\tauweaks$-topology of $B_{E'}$ is metrizable, see Theorem \ref{thm:alaoglu-and-more},
 there then exists a sequence $(x'_n)\subset \denseblo$ with $\langle x_n', x\rangle \to \langle x',x\rangle$ for all $x\in E$.
Consequently, we have 
\begin{align*}
\dualpair {x_n'}{X(\om)}E \to \dualpair {x'}{X(\om)}E\, , \myqquad \om\in \Om ,
\end{align*}
and this in turn shows that $\dualpair {x_n'}{X}E \to \dualpair {x'}{X}E$ in distribution. 
Now, all $\langle x'_n, X\rangle$ are $\R$-valued Gaussian random variables by assumption, and consequently Theorem \ref{thm:levy-grv}
ensures that $\langle x', X\rangle$ is also a Gaussian random variable. Consequently, $X$ is a Gaussian random variable. The converse implication is trivial.

Finally, the assertion for Gaussian measures follows from considering the random variable $X:= \id_E:E\to E$ on 
the probability space $(E,\sborelEx E,\P)$, since in this case we have $\P_X = \P$ and $\P_{\dualpair {x'}XE} = \P_{x'}$.
\end{proof}

By definition, every one-dimensional projection $\dualpair {x'}{X}E$ of a centered
Gaussian random variable is a  centered Gaussian random variable, and it is obvious that 
the set of these collections form a linear space. This motivates the following definition taken from \cite{Janson97}.

\begin{definition}
Let $(\Om,\sA,\P)$ be a probability space and $H\subset \Lx 2 \P$. Then $H$ is called a Gaussian Hilbert space, if 
$H$ is a closed subspace of  $\Lx 2 \P$ and every $h\in H$ is a centered Gaussian random variable.
\end{definition}

The next result connects Gaussian Hilbert spaces to Gaussian random variables.

\begin{lemma}\label{lem:gaussspace-of-grv}
Let $(\Om,\sA,\P)$ be a probability space, $E$ be a separable Banach space, and $X:\Om\to E$ be a  Gaussian random variable.
Then
\begin{align*}
\gaussspace X := \overline{\bigl\{  \dualpair{x'}{X-\E X}E: x'\in E' \bigr\}}^{\Lx 2 \P}
\end{align*}
is a Gaussian Hilbert space and each $\P$-equivalence class in $\gaussspace X$ has a $\s(X)$-measurable representative.
In addition, if $F$ is another separable Banach space, then the following statements hold true:
\begin{enumerate}
\item If $A:E\to F$ is a bounded and linear operator, then $\gaussspace {A\circ X} \subset \gaussspace X$.
\item If $Y:\Om\to F$ is a random variable such that $X$ and $Y$ are jointly Gaussian random variables, then 
\begin{align*}
\gaussspace{X,Y}  := \gaussspace{(X,Y)} =  \overline{\gaussspace X + \gaussspace Y}^{\Lx 2 \P} \, .
\end{align*}
\end{enumerate}
\end{lemma}


\begin{proof}[Proof of Lemma \ref{lem:gaussspace-of-grv}]
Without loss of generality we assume that $X$ and $Y$ are centered. 

Now, \cite[Example 1.8]{Janson97} shows that $\gaussspace X$ is a Gaussian 
Hilbert space. However, this also directly follows from the observation that $\gaussspace X$ is a closed subspace of $\Lx 2 \P$ by construction 
and from the fact that $\Lx 2 \P$-limits of centered Gaussian random variables are centered Gaussian random variables, see also Theorem \ref{thm:levy-grv}.

Let us now find a  $\s(X)$-measurable representative. Clearly, 
 each $\dualpair{x'}XE \in \gaussspace X$ is $\s(X)$-measurable, and therefore there is nothing to prove for the $\P$-equivalence classes
of these functions. Let us now fix an arbitrary $f\in \gaussspace X$, so that our goal is to find a $\s(X)$-measurable $g:\Om\to \R$ with $\P(\{f\neq g\}) = 0$. To this end, we first note that by the definition of $\gaussspace X$, there exists a sequence $(x'_n)\subset E'$ with
$f_n := \dualpair{x'_n}{X}E \to f$ in $\Lx 2 \P$. Since there exists a $\P$-almost surely converging subsequence,  we may assume without loss
of generality that we actually have the $\P$-almost sure convergence $f_n\to f$.
Moreover, each $f_n$ is $\s(X)$-measurable, and therefore we have
\begin{align*}
A:= \bigl\{ \om\in \Om : \exists \lim_{n\to \infty} f_n(\om)  \bigr\} \in \s(X)\, ,
\end{align*}
see e.g.~\cite[Lemma 2.1.7]{Bogachev07_I}. In addition, the $\P$-almost sure convergence $f_n\to f$ shows $\P(A) = 1$.
We define $g_n := \eins_A f_n$ and $g:\Om\to \R$ by $g(\om) := \lim_{n\to \infty} f_n(\om)$ for $\om \in A$ and $g(\om ) := 0$ for $\om \in \Om\setminus A$.
Since each $g_n$ is $\s(X)$-measurable and $g_n(\om) \to g(\om)$ for all $\om\in \Om$ we then see that
 $g$ is $\s(X)$-measurable. Moreover, we have
 \begin{align*}
\P(\{f\neq g\})
= \P (A \cap \{f\neq g\} )  + \P\bigl( (\Om\setminus A ) \cap \{f\neq g\} ) \bigr)
= \P (A \cap \{f\neq \lim_{n\to \infty} f_n\} )
= 0\, ,
 \end{align*}
and hence we have found the desired $\s(X)$-measurable representative $g$ of the $\P$-equivalence class of $f$.

\ada i For $y'\in F'$ we have $x':= y'\circ A\in E'$, and this yields
\begin{align*}
\dualpair {y'}{A\circ X}F = y'\circ A\circ X = \dualpair{x'}{  X}E \in \gaussspace X\, .
\end{align*}
In other words we have $\{  \dualpair{y'}{A\circ X}F: y'\in F'\} \subset \gaussspace X$ and
taking the $\Lx 2 \P$-closure then yields the assertion.

\ada {ii} For the inclusion ``$\supset$'' we consider the coordinate 
projection $\coordproj E:E\times F\to E$. Then we have $\coordproj E(X,Y) = X$ and therefore 
 \emph{i)}  gives $\gaussspace X = \gaussspace {\coordproj E(X,Y)} \subset \gaussspace{X,Y}$. Since we analogously find 
 $\gaussspace Y \subset \gaussspace{X,Y}$, we conclude that $\gaussspace X + \gaussspace Y \subset \gaussspace{X,Y}$. Taking the $\Lx2 \P$-closure 
 then yields the desired inclusion since $\gaussspace{X,Y}$ is a closed linear subspace of $\Lx 2 \P$.
%

To prove the converse inclusion, 
we first recall that  $E'\times F' \to   (E\times F)'$ defined by 
$(x',y') (x,y) := x'(x) + y'(y)$ for all $x'\in E'$, $y'\in F'$, $x\in E$, and $y\in F$
is an isomorphism.
If we now 
 fix an $z'\in (E\times F)'$,
there then exist $x'\in E'$ and $y'\in F'$ such that
$z'(x,y) = x'(x) + y'(y)$ 
for all  $x\in E$  and $y\in F$, and this in turn implies 
\begin{align*}
\dualpairxy {z'}{(X,Y)}{(E\times F)'}{E\times F} = \dualpair{x'}XE + \dualpair{y'}YF  \in \gaussspace X + \gaussspace Y\, .
\end{align*}
Taking the $\Lx 2 \P$-closure thus gives the inclusion ``$\subset$''.
\end{proof}


Similar to the finite dimensional case, a Gaussian measure $\P$ on a separable Banach space $E$ has a
characteristic function that is solely described by its mean and covariance.
The following result, which can be found in e.g.~\cite[Proposition 2.8 in Chapter IV.2.4]{VaTaCh87}, provides the details, where 
we refer to Definition \ref{def:abstr-cov} for the notion of 
abstract covariance operators.

\begin{proposition}\label{prop:cf-of-general-gm}
Let $E$ be a separable Banach space and $\P$ be a probability measure on $\sborelnormx E$. Then the following statements are equivalent:
\begin{enumerate}
\item   $\P$ is a Gaussian measure.
\item There exist a $\mu\in E$ and an abstract covariance operator $K:E'\to E$ such that 
\begin{align}\label{eq:cf-of-gm-general}
\p_\P(x') =  \eul^{\imi \dualpair{x'}\mu E} \cdot \exp\biggl( -\frac{\dualpair {x'}{Kx'}E}{2} \biggr) \, , \myqquad x'\in E'\, .
\end{align}
%
\end{enumerate}
In this case, $\mu$ is the mean of $\P$ and $K$ is the covariance operator of $\P$.
\end{proposition}

%

Our final goal is to generalize  Theorem \ref{thm:levy-continuity} to the infinite dimensional case. To this end,  
let $E$ be a separable Banach space. Then  a 
  set $\ca M$ of probability measures on $(E,\sborelnormx E)$ is called uniformly tight, if for all $\e>0$ there exists 
a compact $K\subset E$ with $\mu(E\setminus K) \leq \e$ for all $\mu \in \ca M$.
To quickly discuss the definition recall
 that 
given a probability measure $\mu$ on $(E,\sborelnormx E)$, Ulam's theorem, see e.g.~\cite[Lemma 26.2]{Bauer01} or
\cite[Theorem 3.1 in Chapter I.3.2]{VaTaCh87}, 
shows that $\mu$ is a Radon measure in the sense of inner regularity, that is
\begin{align}\label{def:radon-by-inner-reg}
\mu(B) = \sup\{\mu(K) : K\subset B \mbox{ and $K$ is compact} \}
\end{align}
for all $B\in \sborelnormx E$. In particular, $\{\mu\}$ is (uniformly) tight.
%
%
Here we emphasize that we use the term ``Radon measure'' in the sense of \cite[p.~28]{VaTaCh87}, \cite[p.~97]{Bogachev98} and \cite[Definition 7.1.1]{Bogachev07_II},
as we will need this definition also in the proof of Theorem \ref{thm:general-levy} below when applying  results from \cite{VaTaCh87} and \cite{Bogachev98}.
By \eqref{def:radon-by-inner-reg} we conclude that  every finite set $\ca M$ of probability measures on $(E,\sborelnormx E)$ is  uniformly tight. 
Similarly, the finite union of uniformly tight sets is also uniformly tight.

Before we generalize  Theorem \ref{thm:levy-continuity}, we also like to recall a  characterization of uniform tightness
 from \cite[Example 8.6.5]{Bogachev07_II}. 
In view of our formulation below we note that \cite[Example 8.6.5]{Bogachev07_II}   additionally lists the
condition
$\mu(\{V=\infty\}) = 0$ for all $\mu \in \ca M$. However, this condition is   implied by the finiteness 
of the integrals in \emph{ii)}, and hence we omitted it.

\begin{lemma}\label{lem:uni-tight}
Let $E$ be a separable Banach space and $\ca M$ be a set of probability measures on $(E,\sborelnormx E)$. 
Then the following statements are equivalent:
\begin{enumerate}
\item $\ca M$ is uniformly tight.
\item There exists a $\sborelnormx E$-measurable $V:E\to [0,\infty]$ such that $\{V\leq c\}$ is compact for all $c\in [0,\infty)$ and 
\begin{align*}
\sup_{\mu \in \ca M} \int_E V \intd \mu < \infty \, .
\end{align*}
\end{enumerate}
\end{lemma}

Finally, we present the announced generalization of Theorem \ref{thm:levy-continuity}  to the infinite dimensional case.

\begin{theorem}\label{thm:general-levy}
Let $E$ be a separable Banach space and $(\P_n)$ be a sequence of probability measures on $\sborelnormx E$ such that the following assumptions are satisfied:
\begin{enumerate}
\item The set $\{\P_n: n\geq 1\}$ is uniformly tight.
\item There exists a function $\p:E'\to \Cfield$ such that $\p_{\P_n}(x') \to \p(x')$ for all $x'\in E'$.
\end{enumerate}
Then there exists a probability measure $\P$ on $E$ such $\P_n\to \P$ weakly and $\p_\P = \p$. 
\end{theorem}

\begin{proof}[Proof of Theorem \ref{thm:general-levy}]
We will apply \cite[Theorem 3.1 in Chapter IV.3.1]{VaTaCh87} for $\Gamma := E'$. Here we first note that Hahn-Banach's theorem ensures that $E'$ is
a separating set consisting of continuous functions $E\to \R$. Moreover, every probability measure on $\sborelnormx E$ is a Radon probability measure, see 
the references around \eqref{def:radon-by-inner-reg}.
%
Finally, \cite[Theorem 3.8.4]{Bogachev98} in combination with  \emph{i)} ensures that  
$\{\P_n: n\geq 1\}$ is relatively weakly compact since the underlying space $E$ is a metric space, and thus  a completely regular topological space,
where the latter notion can be found in e.g.~\cite[Appendix A.1 on page 363]{Bogachev98}.
Now the assertion follows from 
\cite[Theorem 3.1 in Chapter IV.3.1]{VaTaCh87}. 
\end{proof}

%% file: appendix-rcp.tex
\section{Regular Conditional Probabilities}\label{app:rcp+cond-ex}

In this supplement we collect some useful properties of regular conditional probabilities and their 
relationship to Banach space valued conditional expectations.

We begin with a general statement that ensures the existence and uniqueness of 
regular conditional probabilities. For its formulation recall 
that a topological space $(T,\topol)$ is called Polish, if there exists 
a metric $\metric:T\times T\to [0,\infty)$ generating the topology $\topol$, such that $(T,\metric)$ is 
complete and separable. Obviously, separable Banach spaces are Polish spaces. 
Now, the following result, which can be found in e.g.~\cite[Theorem 44.3]{Bauer96},
\cite[Theorem 10.2.2]{Dudley02}, and
\cite[Theorem 8.37]{Klenke14}, shows that the
existence and uniqueness  of regular conditional probabilities can be guaranteed as long as the space 
$T$ in Definition \ref{def:reg-cond-prob}
is a Polish space and $\sB$ is the corresponding Borel $\s$-algebra.

\begin{theorem}\label{thm:ex-ei-reg-cond-prob}
Let $(T,\topol)$ be a Polish space with Borel $\s$-algebra $\sborel$, $(U,\sA)$ be a measurable space, and
$\P$ be a probability measure on $(T\times U, \sborel\otimes \sA)$. Then the following statements hold true:
\begin{enumerate}
\item There exists a
regular conditional probability $\P(\mycdot|\mycdot):\sborel \times U\to [0,1]$ of $\P$ given $U$.
\item If $\P_1(\mycdot|\mycdot)$ and $\P_2(\mycdot|\mycdot)$ are regular conditional probabilities of $\P$ given $U$,
then there exists an $N\in \sA$ with $\PU(N) = 0$ such that for all $B\in \sborel$ and $u\in U\setminus N$ we have 
\begin{align*}
\P_1(B|u) = \P_2(B|u)\, .
\end{align*}
\end{enumerate}
\end{theorem}


Our next result describes how regular conditional probabilities of image measures can be computed
with the help of regular conditional probabilities of the initial measure.

\begin{theorem}\label{thm:rcp-under-trafo}
For $i=0,1$, let $(T_i,\topol_i)$ be Polish spaces with Borel $\s$-algebras $\sborel_i$,
 $(U_i,\sA_i)$ be measurable spaces, 
 and $\P\hochkl i$ be  probability measures on $(T_i\times U_i, \sborel_i\otimes \sA_i)$. 
Moreover, let $\Phi:T_0\to T_1$ and $\Psi:U_0\to U_1$ be measurable maps, such that we have 
\begin{align*}
\s(\Psi) = \sA_0 \myqquad \mbox{ and } \myqquad \P\hochkl 1 = \P\hochkl 0_{(\Phi,\Psi)}\, .
\end{align*}
Then the following statements hold true:
\begin{enumerate}
\item We have $\P\hochkl 1_{U_1} = (\P\hochkl 0_{U_0})_\Psi$.
\item For all  regular conditional probabilities $\P\hochkl 0(\mycdot |\mycdot)$ of $\P\hochkl 0$ given $U_0$ and $\P\hochkl 1(\mycdot |\mycdot)$ of $\P\hochkl 1$ given $U_1$ 
there exists an $N\in \sA_0$ with $\P\hochkl 0_{U_0}(N) =0$ such that for all $B_1\in \sborel_1$ we have 
\begin{align}\label{eq:rcp-under-trafo}
\P\hochkl 0\bigl((\Phi^{-1}(B_1)| u_0\bigr) = \P\hochkl 1 \bigl(B_1|\Psi(u_0)\bigr) \, , \myqquad u_0\in U_0\setminus N.
\end{align}
\end{enumerate}
\end{theorem}


\begin{proof}[Proof of Theorem \ref{thm:rcp-under-trafo}]
We begin by making some initial observations. To this end, 
we write $\Upsilon :=(\Phi,\Psi): T_0\times U_0\to T_1\times U_1$.  
A simple calculation then shows  
\begin{align*}
\coordproj{T_1}\circ \Upsilon = \Phi \circ \coordproj{T_0} \myqquad \mbox{ and } \myqquad \coordproj {U_1}\circ \Upsilon = \Psi \circ \coordproj {U_0}\, ,
\end{align*}
where $\coordproj{T_i}$ and $\coordproj{U_i}$ are the coordinate projections for $i=0,1$.
Since $ \Phi \circ \coordproj{T_0}$ and $\Psi \circ \coordproj {U_0}$ are measurable, we then see that 
$\Upsilon$ is measurable with respect to the product $\s$-algebras, see e.g.~\cite[Corollary 1.82]{Klenke14}, and in particular our assumption
$\P\hochkl 1 = \P\hochkl 0_{(\Phi,\Psi)} = \P\hochkl 0_{\Upsilon}$ makes sense.
%
For $B_1\in \sborel_1$ and $A_1\in\sA_1$ this in turn implies 
\begin{align}\label{thm:rcp-under-trafo-hi0-new}
 \P\hochkl 1(B_1\times A_1) 
 = 
 \P\hochkl 0 (\Upsilon^{-1}(B_1\times A_1))  
 &=
 \P\hochkl 0 \bigl(\Phi^{-1}(B_1)\times \Psi^{-1}(A_1)\bigr) \, .
\end{align}

\ada i
Our initial observations together with the standard 
formula 
for  image measures of composed transformations   imply
\begin{align*}
\P\hochkl 1_{U_1} 
= \P\hochkl 1_{\coordproj {U_1}} 
= (\P\hochkl 0_\Upsilon)_{\coordproj {U_1}}
= \P\hochkl 0_{\coordproj {U_1}\circ \Upsilon}
= \P\hochkl 0_{\Psi \circ \coordproj {U_0}}
= (\P\hochkl 0 _{\coordproj {U_0}})_\Psi
= (\P\hochkl 0_{U_0})_\Psi \, .
\end{align*}

\ada {ii}
For $A_0\in \sA_0$ and $A_1\in \sA_1$ with $\Psi^{-1}(A_1) = A_0$ and 
measurable $g:U_1\to [0,\infty)$ an  application of 
the transformation formula
for integrals of image measures, see e.g.~\cite[Proposition 2.6.8]{Cohn13}
 yields
\begin{align}\nonumber 
 \int_{U_0} \!\eins_{A_0} \cdot  (g\circ \Psi)  \intd  \P\hochkl 0_{U_0} 
 = 
 \int_{U_0} \! (\eins_{A_1}\circ \Psi) \cdot  (g\circ \Psi)  \intd (\P\hochkl 0_{U_0})  
&= 
\int_{U_1}\! \eins_{A_1}  \cdot g  \intd (\P\hochkl 0_{U_0})_\Psi  \\ \label{thm:rcp-under-trafo-hi1}
 &= 
 \int_{U_1} \!\eins_{A_1} \cdot  g  \intd \P\hochkl 1_{U_1} \, ,
\end{align}
where in the last step we used the already established \emph{i)}.
Applying \eqref{thm:rcp-under-trafo-hi1} to $g:= \P\hochkl 1(B_1|\mycdot)$ shows
\begin{align}\nonumber
 \int_{U_0} \eins_{A_0}(u_0) \P\hochkl 1(B_1|\Psi(u_0)) \intd  \P\hochkl 0_{U_0}(u_0) 
 &= 
 \int_{U_1} \eins_{A_1}(u_1) \P\hochkl 1(B_1|u_1) \intd \P\hochkl 1_{U_1}(u_1) \\  \label{thm:rcp-under-trafo-hi11}
 &= 
 \P\hochkl 1(B_1\times A_1) \, .
\end{align}
Moreover, for a regular conditional probability
$\P\hochkl 0(\mycdot |\mycdot)$   of $\P\hochkl 0$ given $U_0$ our observation \eqref{thm:rcp-under-trafo-hi0-new} gives
\begin{align}\nonumber
 \P\hochkl 1(B_1\times A_1) 
 &=
 \P\hochkl 0 \bigl(\Phi^{-1}(B_1)\times A_0\bigr)  \\ \label{thm:rcp-under-trafo-hi00}
 & = 
  \int_{U_0} \eins_{A_0}(u_0) \P\hochkl 0(\Phi^{-1}(B_1)|u_0) \intd \P\hochkl 0_{U_0}(u_0) \, .
\end{align}
To motivate the next steps we note that combining \eqref{thm:rcp-under-trafo-hi11} and \eqref{thm:rcp-under-trafo-hi00} gives 
the $\P\hochkl 0_{U_0}$-almost sure identity
\begin{align*}
\P\hochkl 0(\Phi^{-1}(B_1)|\mycdot ) = \P\hochkl 1(B_1|\Psi(\mycdot))\, .
\end{align*}
 However, the occurring $\P\hochkl 0_{U_0}$-zero set depends on $B_1$, and hence
we need to work harder to establish the assertion.

Let us therefore
define the probability measure $\Qm := \P\hochkl 0_{(\Phi,\id_{U_0})}$ on $\sborel_1\otimes \sA_0$.
For $B_1\in \sborel_1$, $A_0\in \sA_0$, and $A_1\in \sA_1$ with $\Psi^{-1}(A_1) = A_0$ the calculation \eqref{thm:rcp-under-trafo-hi0-new} then shows 
\begin{align}\label{thm:rcp-under-trafo-qm}
\Qm (B_1\times A_0) = \P\hochkl 0 \bigl(\Phi^{-1}(B_1)\times A_0\bigr)  =  \P\hochkl 1(B_1\times A_1) \, .
\end{align}
Moreover, we have $\Qm_{U_0}(A_0) = \Qm (T_1\times A_0) = \P\hochkl 0(T_0\times A_0) = \P\hochkl0_{U_0}(A_0)$, that is 
$\Qm_{U_0} = \P\hochkl 0_{U_0}$.

With these preparations let us first
consider the map
$\Qm_1 :\sborel_1\times U_0\to [0,1]$ given by  
\begin{align*}
\Qm_1(B_1,u_0) := \P\hochkl 1(B_1|\Psi(u_0)) \, .
\end{align*}
Obviously, $\Qm_1(B_1,\mycdot):U_0\to [0,1]$ is $\sA_0$-measurable for all $B_1\in \sborel_1$ 
and $\Qm_1(\mycdot ,u_0):\sborel_1\to [0,1]$ is a probability measure on $\sborel_1$ for all $u_0\in U_0$.
Let us fix some $B_1\in \sborel_1$ and $A_0\in \sA_0$.
Since we assumed
\begin{align*}
\sA_0 = \s(\Psi) = \{ \Psi^{-1}(A_1): A_1\in \sA_1  \}\, ,
\end{align*}
where the latter identity is the very definition of $\s(\Psi)$, there then exists an $A_1\in \sA_1$
with $\Psi^{-1}(A_1) = A_0$. Combining the observed $\Qm_{U_0} = \P\hochkl 0_{U_0}$ with 
\eqref{thm:rcp-under-trafo-hi11} and \eqref{thm:rcp-under-trafo-qm}
then yields
\begin{align*}
 \int_{U_0} \eins_{A_0}(u_0) \Qm_1(B_1,u_0) \intd  \Qm_{U_0}(u_0) 
& =
 \int_{U_0} \eins_{A_0}(u_0) \P\hochkl 1(B_1|\Psi(u_0)) \intd  \P\hochkl 0_{U_0}(u_0) \\
 &= 
 \P\hochkl 1(B_1\times A_1) \\
&= 
\Qm (B_1\times A_0) \, .
\end{align*}
In other words, $\Qm_1$ is a regular conditional probability of $\Qm$ given $U_0$.

Besides $\Qm_1$, we also
 consider the map $\Qm_0 :\sborel_1\times U_0\to [0,1]$ given by
\begin{align*}
\Qm_0(B_1,u_0) := \P\hochkl 0(\Phi^{-1}(B_1)|u_0)\, .
\end{align*}
Again, $\Qm_0(B_1,\mycdot):U_0\to [0,1]$ is $\sA_0$-measurable for all $B_1\in \sborel_1$ 
and $\Qm_0(\mycdot ,u_0):\sborel_1\to [0,1]$ is a probability measure on $\sborel_1$ for all $u_0\in U_0$, as it is
the image measure of $\P\hochkl 0(\mycdot|u_0)$ under $\Phi$.
Let us fix some $B_1\in \sborel_1$ and $A_0\in \sA_0$. Again there then exists 
an $A_1\in \sA_1$
with $\Psi^{-1}(A_1) = A_0$ and by combining $\Qm_{U_0} = \P\hochkl 0_{U_0}$ 
with 
\eqref{thm:rcp-under-trafo-hi00} and \eqref{thm:rcp-under-trafo-qm}
we find 
\begin{align*}
 \int_{U_0} \eins_{A_0}(u_0) \Qm_0(B_1,u_0) \intd  \Qm_{U_0}(u_0) 
& = 
  \int_{U_0} \eins_{A_0}(u_0) \P\hochkl 0(\Phi^{-1}(B_1)|u_0) \intd  \P\hochkl 0_{U_0}(u_0) \\
&= \P\hochkl 1(B_1\times A_1) \\
&= 
\Qm (B_1\times A_0) \, .
\end{align*}
In other words, $\Qm_0$ is also a regular conditional probability of $\Qm$ given $U_0$.
Now the uniqueness of regular conditional probabilities, see Theorem \ref{thm:ex-ei-reg-cond-prob}, yields the assertion, since we have 
already noted $\Qm_{U_0} = \P\hochkl 0_{U_0}$.
\end{proof}


Our next goal is to investigate Bochner integrals with respect to regular conditional
probabilities and how they relate to Banach space valued conditional expectations. 
Let us begin by introducing the latter. 
To this end, let $E$ be a Banach space, $(\Om,\sA,\P)$ be a probability space, and 
$\siC\subset \sA$ be a sub-$\s$-algebra. 
We denote the restriction of $\P$ to $\siC$ by $\P_{|\siC}$ and write accordingly
\begin{align*}
\sLx 1 {\P_{|\siC},E} := \bigl\{Z\in \sLx 1 {\P,E}: Z:\Om\to E \mbox{ is $(\siC,  \sborel (E))$-measurable } \bigr\}\, ,
\end{align*}
%
With this notation, we can now define conditional expectations.

\begin{definition}\label{def:con-exp}
Let $E$ be a  Banach space,  $(\Om,\sA,\P)$ be a probability space,  $\siC\subset \sA$ be a
sub-$\s$-algebra, and  $X\in \sLx 1 {\P,E}$. Then  
$Z\in \sLx 1 {\P_{|\siC},E}$ is called  a conditional expectation of $X$ given $\siC$, if 
\begin{align*}
\int_\Om \eins_C Z\intd \P_{|\siC} = \int_\Om \eins_C X \intd \P \, , \myqquad C\in \siC.
\end{align*}
\end{definition}

One can show that for each $X\in \sLx 1 {\P,E}$ there exists a 
conditional expectation of $X$ given $\siC$, 
see e.g.~\cite[Proposition 4.1 in Chapter II.4.1]{VaTaCh87} or \cite[Theorem 2.6.20]{HyvNVeWe16}.
Moreover, if we have two conditional expectations $Z_1$ and $Z_2$ of $X$ given $\siC$, then
$\P_{|\siC}(\{Z_1\neq Z_2\}) = 0$, 
see e.g.~\cite[Corollary 5 in Chapter II.2]{DiUh77} or \cite[Theorem 2.6.20]{HyvNVeWe16}. As in the case $E=\R$, we thus 
write $\E(X|\siC)$ for any version of the conditional expectation of $X$ given $\siC$.
Finally,   in the case $\siC = \s(Y)$, we write, as usual, $\E(X|Y) := \E(X|\siC)$.

Moreover,
 note that
for $X_1,X_2\in \sLx 1 {\P,E}$ with $\P(\{X_1\neq X_2\})=0$ we obviously have 
$\E(X_1|\siC)(\om) = \E(X_2|\siC)(\om)$ for $\P_{|\siC}$-almost all $\om \in \Om$.
Consequently we can view the conditional expectation as a map $\Lx 1 {\P,E} \to \Lx 1 {\P_{|\siC},E}$.

The next results collects some useful properties of conditional expectations, where we note that most of them 
are literal generalizations of   well-known properties in the case $E=\R$. For a proof we refer to 
\cite[Chapter II.4.1]{VaTaCh87} or \cite[Chapter 2.6.d]{HyvNVeWe16}.

\begin{theorem}\label{thm:prop-cond-ex}
Let $E$ be a Banach space,  $(\Om,\sA,\P)$ be a probability space, and $\siC\subset \sA$ be a 
sub-$\s$-algebra. Then the following statements hold true:
\begin{enumerate}
\item The map $\E (\mycdot |\siC): \Lx 1 {\P,E} \to \Lx 1 {\P_{|\siC},E}$ is linear and bounded 
with $\snorm{\E (\mycdot |\siC)} = 1$ and $\E (X |\siC) = X$ for all $X\in \Lx 1 {\P_{|\siC},E}$.
\item For all $1\leq p\leq \infty$ the map $\E (\mycdot |\siC): \Lx p {\P,E} \to \Lx p {\P_{|\siC},E}$ is well-defined, linear, and
continuous with $\snorm{\E (\mycdot |\siC)} = 1$.
\item For all sub-$\s$-algebras $\siC_1\subset \siC_2\subset \sA$ and all $X\in \sLx 1 {\P,E}$
we have $\E ( \E (X |\siC_2) |\siC_1) = \E (X |\siC_1)$.
\item If $T:E\to F$ is bounded and linear, then $\E (T\circ X |\siC) = T\circ \E (X |\siC)$ for all $X\in \sLx 1 {\P,E}$.
\item We have $\E(\E (X |\siC) ) = \E X$ for all $X\in \sLx 1 {\P,E}$.
\end{enumerate}
\end{theorem}

Our next goal is to collect information about $E$-valued martingales. To this end, let 
$(\Om,\sA,\P)$ be a probability space and $E$ be a Banach space.
Then  a sequence $(\siC_n)$ of 
sub-$\s$-algebras of $\sA$ is said to be a filtration, if $\siC_n \subset \siC_{n+1}$ for all $n\geq 1$.
In this case a sequence $(X_n)$ of maps $X_n:\Om\to E$ is called a $(\siC_n)$-martingale, if for all $n\geq 1$ we have 
$X_n\in \sLx 1 {\P_{|\siC_n},E}$ and 
\begin{align*}
\E (X_{n+1}|\siC_n) = X_n\, .
\end{align*}
Recall that $\R$-valued martingales enjoy well-known convergence results. The following theorem, whose proof can be 
found in e.g.~\cite[Theorems 4.1 and 4.2 in Chapter II.4.2]{VaTaCh87}
or \cite[Theorem 3.3.2]{HyvNVeWe16}, collects analogous results for $E$-valued martingales.

\begin{theorem}\label{thm:martingale-convergence}
Let $E$ be a  Banach space,  $(\Om,\sA,\P)$ be a probability space,  $(\siC_n)$ be a filtration,
and  $\siC_\infty :=  \s(\siC_n:n\geq 1)$. Moreover, let $1\leq p<\infty$ and $X\in \sLx p {\P,E}$, and for $n\geq 1$ we
define 
\begin{align*}
X_n := \E (X|\siC_n) \myqquad \mbox{ and } \myqquad X_\infty := \E (X|\siC_\infty)\, .
\end{align*}
Then $(X_n)$ is a $(\siC_n)$-martingale and we have both
\begin{align*}
\snorm{X_n-X_\infty}_{\sLx p {\P,E}} \to 0 \myqquad \mbox{ and } \myqquad \P(\lim_{n\to \infty} X_n \neq X_\infty\}) = 0.
\end{align*}
\end{theorem}

Our next result establishes a disintegration formula for Bochner integrals and regular conditional
probabilities. In addition, it shows that regular conditional probabilities can be used to construct
conditional expectations. In the case $E=\R$, these results are well-known, see
e.g.~\cite[Theorem 10.2.5]{Dudley02}  and to some extent also
 \cite[Theorem 8.38]{Klenke14}, where in all cases actually a more general setting
 that goes beyond the ``product case'' is treated.

\begin{theorem}\label{thm:disintegration-bs}
Let $(T,\sB)$ and $(U,\sA)$ be measurable spaces, $\P$ be a probability measure on $(T\times U, \sB\otimes \sA)$,
and
$\P(\mycdot|\mycdot):\sB\times U\to [0,1]$ be a  regular conditional probability
of $\P$ given $U$. Moreover, let $E$ be a   Banach space 
and $X\in \sLx 1 {\P, E}$.
Then there exists an $N\in \sA$ with $\PU(N) = 0$ such that for all $u\in U\setminus N$ we have 
\begin{align}\label{eq:disintegration-bochner-int}
X(\mycdot, u) \in \sLx 1 {\P(\mycdot|u), E}\, .
\end{align}
Moreover, for all $N\in \sA$ with $\PU(N) = 0$ and \eqref{eq:disintegration-bochner-int},
 the map  $Z:U\to E$ defined by 
\begin{align}\label{eq:disintegration:inner-integral-bs}
Z(u) := \int_T \eins_{U\setminus N}(u) X(t,u) \, \P(\ind t|u)\, , \myqquad u\in U.
\end{align}
 is strongly measurable
 and we have both $Z\in \sLx 1 {\PU,E}$ and
\begin{align}\label{eq:disintegration}
\int_{T\times U} X \intd \P
=
\int_U   Z \intd \PU
=
\int_U \int_T \eins_{U\setminus N}(u) X(t,u) \, \P(\ind t|u) \PU(u)\, .
\end{align}
Finally, the function $Z\circ \coordproj U:T\times U\to E$ is a version of $\E(X|\siC)$, where $\siC:= \coordproj U^{-1} (\sA)$.
\end{theorem}

\begin{proof}[Proof of Theorem \ref{thm:disintegration-bs}]
We first note that for $X:T\times U\to [0,\infty)$  the assertions \eqref{eq:disintegration-bochner-int} and \eqref{eq:disintegration} are  well-known: Indeed,
in this case, the measurability of $Z:U\to [0,\infty]$ defined by \eqref{eq:disintegration:inner-integral-bs}
with $N:= \emptyset$
can be found in e.g.~\cite[Lemma 14.20]{Klenke14}. Moreover, \eqref{eq:disintegration}
is shown in e.g.~\cite[Theorem 14.29]{Klenke14}. 
 For $\P$-integrable $X:T\times U\to [0,\infty)$, the just verified
 \eqref{eq:disintegration} then shows that there exists an $N\in \sA$ with $\PU(N)=0$ and 
$Z(u) \in [0,\infty)$ for all $u\in U\setminus N$, that is, we have \eqref{eq:disintegration-bochner-int}. 
Finally, if we fix such an $N$ and
apply the previous results to $X_0 := \eins_{U\setminus N} X$,
we again have the measurability of $Z$ and \eqref{eq:disintegration}. 

Let us now assume that $E$ is a general Banach space. We first consider the case in which
$X\in \sLx 1 {\P,E}$ is a step function. To this end, we pick a representation
\begin{align*}
X = \sum_{i=1}^{m}\eins_{D_{i}} x_{i}
\end{align*}
for some $x_{i}\in E$ and $D_{i}\in \sB\otimes \sA$. 
For $u\in U$ we have  $D_{i,u}:= \{t\in T: (t,u)\in D_i\}\in \sB$, see e.g.~\cite[Lemma 14.13]{Klenke14},
and hence we see that 
\begin{align}\label{thm:disintegration-bs-h666}
X(\mycdot, u) = \sum_{i=1}^{m}\eins_{D_{i}}(\mycdot, u) x_{i} = \sum_{i=1}^{m}\eins_{D_{i,u}} x_{i}
\end{align}
is a measurable step functions for all fixed $u\in U$. Consequently, \eqref{eq:disintegration-bochner-int} holds for $N:= \emptyset$. 
For this $N$ we 
define $Z$ by \eqref{eq:disintegration:inner-integral-bs}.
For $u\in U$  
the definition of Bochner integrals for step functions, see \eqref{def:step-fct-bochner-int}, yields
\begin{align*}
Z(u)
= \int_T  \sum_{i=1}^{m}  \eins_{D_{i}}(t,u) x_{i} \, \P(\ind t|u)
&= \int_T  \sum_{i=1}^{m}  \eins_{D_{i,u}}(t) x_{i} \, \P(\ind t|u) \\
&=   \sum_{i=1}^{m}  \biggl(\int_T\eins_{D_{i,u}}(t)  \, \P(\ind t|u) \biggr)\cdot  x_{i} \\
&=   \sum_{i=1}^{m}  \biggl(\int_T\eins_{D_{i}}(t,u)  \, \P(\ind t|u) \biggr)\cdot  x_{i} \, .
\end{align*}
Let us   define $f_i:U \to [0,\infty)$ by
\begin{align*}
f_i(u) :=  \int_T\eins_{D_{i}}(t,u)  \, \P(\ind t|u)  \, .
\end{align*}
From our initial remarks we know that each $f_i$ is measurable, and hence $u\mapsto   f_i(u)x_i$ is a strongly measurable, $E$-valued function.
Since $Z$ is the sum of these $m$ functions, it is also strongly measurable.
Moreover, the already established \eqref{eq:disintegration} for non-negative, measurable functions gives
\begin{align*}
\int_U   Z \intd \PU
&=
\int_U \sum_{i=1}^{m} \biggl(\int_T\eins_{D_{i}}(t,u)  \, \P(\ind t|u)  \biggr)\cdot  x_{i}  \intd \PU(u) \\
&=
 \sum_{i=1}^{m} \biggl( \int_U  \int_T\eins_{D_{i}}(t,u)   \, \P(\ind t|u) \intd \PU(u) \biggr) \cdot x_{i} \\
 &=
 \sum_{i=1}^{m}  \biggl( \int_{T\times U} \eins_{D_{i}}    \intd \P \biggr)  \cdot x_{i} \\
 &=
 \int_{T\times U} X  \intd \P\, .
\end{align*}
In other words, \eqref{eq:disintegration} holds for measurable $E$-valued step functions and $N= \emptyset$.

Let us now consider the case of general Banach spaces $E$ and $X\in \sLx 1 {\P,E}$.
To prove \eqref{eq:disintegration-bochner-int}, 
we write $Y:= \snorm{X(\mycdot, \mycdot)}_E:T\times U\to [0,\infty)$. 
Since $X$ is strongly measurable,
 $Y$ is measurable. Applying  Equation \eqref{eq:disintegration} in the already established case of $E=\R$ with
 non-negative functions to $Y$ we then obtain
 \begin{align}\label{thm:disintegration-bs-h1}
 \int_U \int_T \snorm{X(t,u)}_E \, \P(\ind t|u) \PU(u)
 = 
 \int_{T\times U} \snorm{X(\mycdot, \mycdot)}_E \intd \P
 < \infty\, .
 \end{align}
Consequently, there exists an $N\in \sA$ with $\PU(N) = 0$ such that for all $u\in U\setminus N$ we have 
\begin{align*}
\int_T \snorm{X(t,u)}_E \, \P(\ind t|u) < \infty. 
\end{align*}
This shows \eqref{eq:disintegration-bochner-int}.

For  the following steps we fix an $N\in \sA$ with $\PU(N) = 0$ such that \eqref{eq:disintegration-bochner-int} holds.
We write $X_0 := \eins_{U\setminus N} X$ and define $g:T\times U\to [0,\infty)$ by
$g(t,u) := \snorm{X_0(t,u)}_E$ for all
$t\in T$ and $u\in U$. By our assumption,  $g:T\times U\to [0,\infty)$ is $\P$-integrable.

Let us first verify that $Z$ is strongly measurable. Since $X$ is strongly measurable, so is $X_0$, and therefore
there exists a sequence $(X_n)$ of
measurable step functions $X_n:T\times U\to E$ approximating $X_0$ in the sense of
 \eqref{eq:step-fct-approx-1} and \eqref{eq:step-fct-approx-2}.
We define $Z_n:U\to E$ in the sense of \eqref{eq:disintegration:inner-integral-bs}, that is
\begin{align}\label{eq:Zn_def}
Z_n(u) := \int_T   X_n(t,u) \, \P(\ind t|u)\, , \myqquad u\in U.
\end{align}
Since $X_n$ is a measurable step function
we already know that this definition is possible and 
that $Z_n$ is strongly measurable. Moreover, for each fixed $u$, the sequence $(X_n(\mycdot,u))$ consists of
measurable step functions, see \eqref{thm:disintegration-bs-h666}, which approximates $X_0(\mycdot,u)$ in the sense of \eqref{eq:step-fct-approx-1} and \eqref{eq:step-fct-approx-2}.
By the definition of the Bochner integral \eqref{def:bochner-integral} we thus find
\begin{align}\label{eq:Zn_conv}
Z_n(u) = \int_T  X_n(t,u) \, \P(\ind t|u) \to \int_T  X_0(t,u) \, \P(\ind t|u) = Z(u) \, , \myqquad u\in U.
\end{align}
In other words, $Z$ is the pointwise limit of a sequence of strongly measurable functions, and thus also strongly measurable.

To verify that  $Z\in \sLx 1 {\PU,E}$, we simply note that the norm inequality for Bochner integrals
 \eqref{eq:bochner-int-cont} together with \eqref{thm:disintegration-bs-h1} yields
\begin{align*}
\int_U \snorm{Z}_E \intd \PU
&= 
\int_U \bnorm{ \int_T \eins_{U\setminus N}(u) X(t,u) \, \P(\ind t|u)}_E \intd \PU(u) \\
&\leq 
\int_U  \int_T  \snorm{X(t,u)}_E \, \P(\ind t|u)   \intd \PU(u) \\
&< \infty\, .
\end{align*}

To
 establish \eqref{eq:disintegration}
for general $X\in \sLx 1 {\P,E}$ we again fix  a sequence $(X_n)$ of
measurable step functions $X_n:T\times U\to E$ approximating $X_0$ in the sense of
 \eqref{eq:step-fct-approx-1} and \eqref{eq:step-fct-approx-2}.
For the strongly measurable functions $Z_n:U\to E$ defined by \eqref{eq:Zn_def} we have already established
\eqref{eq:Zn_conv}.
By \eqref{eq:step-fct-approx-2} we further find
\begin{align*}
\snorm{Z_n(u)}_E
=
\bnorm  {\int_T  X_n(t,u) \, \P(\ind t|u)   }_E
\leq
\int_T  \snorm{ X_n(t,u)}_E \, \P(\ind t|u)
&\leq
\int_T  g(t,u) \, \P(\ind t|u) \\
&=: h(u)
\end{align*}
for all $u\in U$. Since $g$ is $\P$-integrable, the initially considered case of non-negative, measurable functions
shows that $h:U\to [0,\infty]$ is $\PU$-integrable. 
In summary, the sequence $(Z_n)$ together with $Z$ and $h$ satisfy \eqref{eq:meas-fct-approx-1}
and \eqref{eq:meas-fct-approx-2}, and hence an application of the dominated convergence theorem for Bochner integrals 
\eqref{eq:dominated-conv-bnochner} shows
\begin{align}\label{eq:disintegration-limit-11}
\int_U   Z_n \intd \PU \to  \int_U   Z \intd \PU \, .
\end{align}
Since the sequence $(X_n)$ of
measurable step functions $X_n$ approximates $X_0$ in the sense of
 \eqref{eq:step-fct-approx-1} and \eqref{eq:step-fct-approx-2}, the definition of the Bochner integral further shows
\begin{align}\label{eq:disintegration-limit-1}
\int_{T\times U} X_n \intd \P \to  \int_{T\times U} X_0 \intd \P\, .
\end{align}
Moreover,  we already know that  \eqref{eq:disintegration} holds for all pairs of $X_n$ and $Z_n$,
that is 
\begin{align*}
\int_U   Z_n \intd \PU  = \int_{T\times U} X_n \intd \P \, ,
\end{align*}
and therefore
the just established
\eqref{eq:disintegration-limit-11} and \eqref{eq:disintegration-limit-1} then yield
  \eqref{eq:disintegration}  in the general case.

Finally, we need to verify that $Z\circ \coordproj U$ is a version of $\E(X|\siC)$.
Here we first note that $Z\circ \coordproj U$ is $\siC$-measurable by construction.
Let us now pick a $C\in \siC$. Then there exists an $A\in \sA$ with $C=T\times A$.
Now, the function  $X_A:= \eins_A X_0$ satisfies $X_A\in \sLx 1 {\P,E}$
and by construction we have \eqref{eq:disintegration-bochner-int} for
$N_A := \emptyset$. In addition, we find
%
\begin{align*}
Z_A(u) := \int_T X_A(t,u) \, \P(\ind t|u)=  \eins_A(u) \int_T X_0(t,u) \, \P(\ind t|u) =  \eins_A(u) Z(u)
\end{align*}
for all $u\in U$. 
From this we can conclude $Z_A\circ \coordproj U = 1_C \cdot (Z\circ \coordproj U)$, and
applying \eqref{eq:disintegration} to $X_A$  thus yields
\begin{align*}
\int_{T\times U} \eins_{C} X\intd \P
=
\int_{T\times U} X_A \intd \P
=
\int_U   Z_A  \intd  \PU
=
\int_U   Z_A  \intd  \P_{\coordproj U}
&=
\int_{T\times U} Z_A \circ \coordproj U \intd \P \\
&=
\int_{T\times U} \eins_C \cdot (Z \circ \coordproj U) \intd  \P  \, .
\end{align*}
Consequently, $Z\circ \coordproj U$ is a version of $\E(X|\siC)$.
\end{proof}

%% file: appendix-gp-on-ct.tex
\section{Continuous Gaussian Processes}\label{app:gp-on-ct}

In this supplement we recall some basics about the relationship between $\sC T$-valued Gaussian random variables and Gaussian processes with continuous paths.
Since we could not find a source that really fits to our purposes, all material presented here is also proven.
 
We begin with a technical result in the spirit of  Krein-Milman's theorem.
 
\begin{lemma}\label{lem:krein-milman-for-mt}
Let $(T,\metric)$ be a compact metric space and $T_0 \subset T$ be dense. Then we have
\begin{align*}
\overline {\aco \{   \diracf t: t\in  T_0  \}}^{\tauweaksseq} = B_{\sC T'}\, ,
\end{align*}
where $\aco A$ denotes the absolute convex hull of a set $A$.
\end{lemma}

\begin{proof}[Proof of Lemma \ref{lem:krein-milman-for-mt}]
We first consider the case $T_0 = T$. Here, we   note that the set of extreme points $\expoints B_{\measpace T}$ of $B_{\measpace T}$ is
given by
\begin{align*}
\expoints B_{\measpace T}=   \bigl\{ \a \dirac t: t\in  T \mbox{ and } \a\in \{-1,1\}\bigr\} \, ,
\end{align*}
see e.g.~\cite[Theorem V.8.4]{Conway90}.
Using the identification \eqref{eq:riesz-represent} of $\sC T'$ with $\measpace T$ and the obvious $\Cdualmap \d_{\{t\}} = \diracf t$
we conclude that the set
$\expoints B_{\sC T'}$ of extreme points  of $B_{\sC T'}$ is
given by
\begin{align*}
\expoints B_{\sC T'}=  \bigl\{ \a \diracf t: t\in  T \mbox{ and } \a\in \{-1,1\}\bigr\} \, .
\end{align*}
Krein-Milman's theorem, see e.g.~\cite[Corollary 2.10.9]{Megginson98}, thus gives
\begin{align*}
\overline {\aco  \{   \diracf t: t\in  T  \}}^{\tauweaks}
=
\overline {\co \bigl\{ \a \diracf t: t\in  T \mbox{ and } \a\in \{-1,1\}\bigr\}}^{\tauweaks} = B_{\sC T'}\, ,
\end{align*}
where  $\co A$ denotes the convex hull of a set $A$ and the first identity is trivial.
Moreover, $\sC T$ is separable, and hence the relative  $\tauweaks$-topology on $B_{\sC T'}$ is metrizable. 
This gives
\begin{align*}
\overline {\aco  \{  \diracf t: t\in  T   \}}^{\tauweaks} = \overline {\aco  \{   \diracf t: t\in  T  \}}^{\tauweaksseq}\, ,
\end{align*}
where we note   this identity can also be deduced from \cite[Corollary 2.7.13]{Megginson98}.

Let us now prove the assertion in  the case of a general, dense $T_0\subset T$.
To this end, we fix a $\Cdualmap \mu\in B_{\sC T'}$. By the already established case we then find a
sequence
\begin{align*}
(\Cdualmap\mu_n)\subset \aco  \{   \diracf t: t\in  T  \}
\end{align*}
with $\dualpair {\Cdualmap\mu_n}f {\sC T} \to  \dualpair {\Cdualmap \mu}f {\sC T}$ for all $f\in \sC T$.
Clearly, each $\Cdualmap\mu_n$ is of the form
\begin{align}\label{cor:dirac-pointwise-dense-h1}
\Cdualmap\mu_n = \sum_{i=1}^{m_n} \a_{i,n}\d_{t_{i,n}}\, ,
\end{align}
where $t_{i,n}\in T$ and $\sum_{i=1}^{m_n} |\a_{i,n}|\leq 1$. Since $T_0\subset T$ is dense, we then find $s_{i,n}\in T_0$
with $\metric(t_{i,n},s_{i,n})\leq 1/n$ for all $i=1,\dots,m_n$ and $n\geq 1$. We define
\begin{align*}
\Cdualmap\nu_n := \sum_{i=1}^{m_n} \a_{i,n}\d_{s_{i,n}} \, .
\end{align*}
Let
us fix an $f\in \sC T$ and denote its uniform modulus of continuity by $\unimodcont f\mycdot$, see e.g.~\cite[page 232]{AdFo03}. Then we have
\begin{align*}
\bigl| \dualpair {\Cdualmap\mu_n-\Cdualmap\nu_n} f{\sC T}  \bigr|
 =
\biggl| \sum_{i=1}^{m_n} \a_{i,n}\dualpair{ \d_{t_{i,n}}\!-\!\d_{s_{i,n}}}f{\sC T}  \biggr|
&\leq
 \sum_{i=1}^{m_n} |\a_{i,n}|   |f(t_{i,n}) - f(s_{i,n})| \\
 &\leq
 \sum_{i=1}^{m_n} |\a_{i,n}| \cdot \unimodcont f{t_{i,n}-s_{i,n}} \\
 & \leq \unimodcont f {1/n}\, ,
\end{align*}
and $\lim_{n\to \infty}  \unimodcont f {1/n} = 0$ as $f$ is uniformly continuous by the compactness of $T$.
This gives
\begin{align*}
\bigl|\dualpair{\Cdualmap\mu- \Cdualmap\nu_n}f{\sC T} \bigr|
\leq  \bigl|\dualpair {\Cdualmap\mu- \Cdualmap\mu_n}f{\sC T}\bigr| + \bigl|\dualpair{ \Cdualmap\mu_n- \Cdualmap\nu_n}f{\sC T}\bigr| \to 0 \, .
\end{align*}
This gives $\Cdualmap\mu \in \overline {\aco  \{  \diracf t: t\in  T_0   \}}^{\tauweaksseq}$, and hence we have shown the inclusion
``$\supset$''. The  inclusion ``$\subset$'' follows by  monotonicity from the already established case $T=T_0$.
\end{proof}

To describe the relationship between  $\sC T$-valued Gaussian random variables and Gaussian processes with continuous paths, we now fix a 
map $X:\Om\to \sC T$. As in Section \ref{sec:examples} we then  obtain
a family $(X_t)_{t\in T}$ with $X_t:\Om\to \R$ by setting
\begin{align}\label{eq:rv-gives-sp-app}
X_t(\om) := \dualpair  {\diracf t} {X(\om)} {\sC T} = (X(\om))(t) \, , \myqquad \om\in \Om.
\end{align}
Obviously,  all maps $t\mapsto X_t(\om)$ are continuous.
Conversely, if we have a family $(X_t)_{t\in T}$ of functions $X_t:\Om\to \R$ such that for each $\om \in \Om$ the path $\sppath X(\om)$ is
continuous, then
\begin{align} 
\om & \mapsto \sppath X(\om)
\label{eq:sp-gives-rv-app}
\end{align}
defines a map $X:\Om \to \sC T $.
Clearly, the operations  \eqref{eq:rv-gives-sp-app} and \eqref{eq:sp-gives-rv-app} are inverse to each other.
The following folklore result shows that they map $\sC T$-valued (Gaussian) random variables to (Gaussian) processes with continuous paths and vice versa.

\begin{lemma}\label{lem:sp-gives-rv-app}
Let $(T,\metric)$ be a compact metric space, $(\Om, \sA, \P)$ be a probability space,
$X:\Om\to \sC T$ be a map, and
$(X_t)_{t\in T}$ be given by \eqref{eq:rv-gives-sp-app}. Then we have:
\begin{enumerate}
\item $X$ is weakly measurable if and only if $(X_t)_{t\in T}$ is a stochastic process with continuous paths.
\item $X$ is a Gaussian random variable if and only if $(X_t)_{t\in T}$ is a Gaussian process with continuous paths.
\end{enumerate}
\end{lemma}

\begin{proof}[Proof of Lemma \ref{lem:sp-gives-rv-app}]
\ada i If $X$ is weakly measurable, then \eqref{eq:rv-gives-sp-app} ensures that $(X_t)_{t\in T}$ is a stochastic process with continuous paths.
To show the converse, we
fix a $\mu\in \measpace T$
with Hahn-Jordan decomposition $\mu = \mu^+ - \mu^-$ and
we consider the map
\begin{align}\label{lem:sp-gives-rv-h1}
\Om\times T & \to \R \\ \nonumber
(\om,t) & \mapsto X_t(\om) \, .
\end{align}
Then, $\om\mapsto X_t(\om)$ is measurable for all $t\in T$ since we have a stochastic process. In addition, 
$t\mapsto X_t(\om)$ is continuous for all $\om \in \Om$ since the stochastic process has continuous paths.
Combining these observations with the separability of $(T,\metric)$, we conclude by e.g.~\cite[Lemma 4.51]{AlBo06} that 
  \eqref{lem:sp-gives-rv-h1} is $\sA\otimes \sborel(T)$-measurable.
Fubini-Tonelli's theorem, see e.g.~\cite[Theorem 14.16]{Klenke14}, then shows that 
 the two functions
\begin{align*}
\om \mapsto \dualpairmc {\mu^\pm} {X(\om)} T  := \int_\Om  X_t(\om) \intd \mu^\pm(t)
\end{align*}
are measurable, and hence $\om \mapsto   \dualpairmc {\mu} {X(\om)} T$ is also measurable. Using the identification of $\sC T'$ with $\measpace T$ given in
\eqref{eq:riesz-represent} we conclude that $X$ is weakly measurable.

\ada {ii} Let us first assume that $X$ is Gaussian. We fix some  $t_1,\dots,t_n\in T$ and write $Y:= (X_{t_1},\dots,X_{t_n})$.
Clearly, our goal is to show that the $\R^n$-valued random variable $Y$ is Gaussian. To this end, we fix some
$a\in \R^n$ and write $\Cdualmap \mu := a_1 \diracf{t_1}+\dots+a_n \diracf{t_n}$.
Then a simple calculation using \eqref{eq:rv-gives-sp-app}  shows
\begin{align}\label{thm:grs-vs-gps-h1}
\skprod  aY_{\R^n}
= \sum_{i=1}^n a_i X_{t_i}
=  \sum_{i=1}^n a_i \dualpair {\diracf{t_i}} X{\sC T}
= \dualpair {\Cdualmap\mu} X{\sC T} \, ,
\end{align}
and since $X$ is a Gaussian random variable, so is $\dualpair {\Cdualmap\mu} X{\sC T}$, and thus also $\skprod  aY_{\R^n}$. Since $a\in \R^n$ was chosen arbitrarily,
we conclude that $Y$ is Gaussian.

Let us now assume that $(X_t)_{t\in T}$ is a Gaussian process with continuous paths. It suffices to show that $X:\Om\to \sC T$ is Gaussian.
To this end, we consider the set
\begin{align*}
\denseblo := \aco  \{   \diracf t: t\in  T   \}\, .
\end{align*}
By Lemma \ref{lem:krein-milman-for-mt} we know that $\denseblo$ is $\tauweaks$-dense in $B_{\sC T'}$. By \eqref{eq:riesz-represent}
and Theorem \ref{thm:test-for-gms} it therefore
  suffices to show that $\dualpair {\Cdualmap\mu} X{\sC T}$ is Gaussian for all $\Cdualmap\mu\in \denseblo$. 
Let us therefore fix a $\Cdualmap\mu\in \denseblo$. Then there exist an $n\geq 1$ and
  $a_1,\dots,a_n\in \R$ with  $|a_1|+\dots+ |a_n|\leq 1$ and $\Cdualmap\mu = a_1 \diracf{t_1}+\dots+a_n \diracf{t_n}$.
%
Defining $Y:= (X_{t_1},\dots,X_{t_n})$, we then find
\begin{align*}
\dualpair {\Cdualmap\mu} X{\sC T} = \skprod  aY_{\R^n}
\end{align*}
as in \eqref{thm:grs-vs-gps-h1}. Since $Y$ is Gaussian, so is $\skprod  aY_{\R^n}$,
and thus $\dualpair {\Cdualmap\mu} X{\sC T}$ is   Gaussian.
\end{proof}

\begin{lemma}\label{lem:cont-mean+cov-4-cont-GP}
Let $(T,\metric)$ be a  metric space, $(\Om, \sA, \P)$ be a probability space,
and
$(X_t)_{t\in T}$ be a  Gaussian process with continuous paths. Then its mean and covariance functions are continuous.
\end{lemma}

\begin{proof}[Proof of Lemma \ref{lem:cont-mean+cov-4-cont-GP}]
To show the continuity of the mean function, we fix a $t_\infty\in T$ and a sequence $(t_n)\subset T$ with $t_n\to t_\infty$. Our assumption then guarantees
$X_{t_n}(\om) \to X_{t_\infty}(\om)$ for all $\om \in \Om$, and therefore also $X_{t_n} \to X_{t_\infty}$ is distribution. Now, Theorem \ref{thm:levy-grv} gives $\E X_{t_n}\to \E X_{t_\infty}$.

For the continuity of the covariance function $k$, we again fix a $t_\infty\in T$  and a sequence $(t_n)\subset T$ with  $t_n\to t_\infty$. As for the mean function, 
Theorem \ref{thm:levy-grv} then shows  
\begin{align}\label{lem:cont-mean+cov-4-cont-GP-h1}
k(t_n,t_n) = \var X_{t_n} \to \var X_t = k(t,t)\, .
\end{align}
Let us now fix an $s_\infty\in T$ and a sequence $(s_n)\subset T$ with  $s_n\to s_\infty$.
We write   $Y_n := X_{s_n}- \E X_{s_n}$ and  $Z_n := X_{t_n} - \E X_{t_n}$ for $n\in \N\cup\{\infty\}$. Then we have 
$Z_n(\om) -  Z_\infty(\om) \to 0$ for all $\om \in \Om$ by our assumption and the already established continuity of the mean function. Moreover, each
$Z_n - Z_\infty$ is a centered Gaussian random variable, and therefore another application of Theorem \ref{thm:levy-grv}  shows $\var (Z_n-Z_\infty) \to 0$. The latter implies 
\begin{align}\nonumber
| k(s_\infty,t_n) - k(s_\infty,t_\infty)  |
= 
| \E Y_\infty Z_n  - \E Y_\infty Z_\infty  |
&\leq 
( \E Y_\infty^2)^{1/2} ( \E (Z_n - Z_\infty)^2)^{1/2} \\ \label{lem:cont-mean+cov-4-cont-GP-h2}
&\to 0\, .
\end{align}
In combination with \eqref{lem:cont-mean+cov-4-cont-GP-h1} we analogously obtain
\begin{align}\label{lem:cont-mean+cov-4-cont-GP-h3}
| k(s_n,t_n) - k(s_\infty,t_n)  |
= 
| \E Y_n Z_n  - \E Y_\infty Z_n |
\leq 
 ( \E (Y_n - Y_\infty)^2)^{1/2} \cdot ( \E Z_n^2)^{1/2}\to 0
\end{align}
Combining \eqref{lem:cont-mean+cov-4-cont-GP-h2} and \eqref{lem:cont-mean+cov-4-cont-GP-h3} with 
\begin{align*}
| k(s_n,t_n) - k(s_\infty,t_\infty)  |
\leq 
| k(s_n,t_n) - k(s_\infty,t_n)  | + | k(s_\infty,t_n) - k(s_\infty,t_\infty)  |
\end{align*}
then yields the assertion.
\end{proof}

\begin{lemma}\label{lem:cov-gp-comp}
Let $(T,\metric)$ be a compact metric space, $(\Om, \sA, \P)$ be a probability space,
$X:\Om\to \sC T$ be a Gaussian random variable. Moreover, let 
$(X_t)_{t\in T}$ be the associated Gaussian process given by \eqref{eq:rv-gives-sp-app} and $k$ be its covariance function.
Then for all $\mu,\nu \in \measpace T$ we have 
\begin{align*}
\dualpair {\Cdualmap\mu}{\cov (X) \Cdualmap\nu}{\sC T} = \int_T \int_T k(s,t) \intd\mu(s) \intd\nu(t) \, .
\end{align*}
\end{lemma}

\begin{proof}[Proof of Lemma \ref{lem:cov-gp-comp}]
We first note that using Hahn-Jordan decompositions for $\mu,\nu \in \measpace T$ we may assume without loss of generality that $\mu$ and $\nu$ are finite measures.
Moreover, 
from Lemma \ref{lem:sp-gives-rv-app} and Lemma \ref{lem:cont-mean+cov-4-cont-GP} we know that $k$ is continuous, and consequently, we can consider the map $\Phi:T\to \sC T$ defined by 
$\Phi(t) := k(\mycdot,t)$. 

Our first goal is to show that $\Phi\in \sLx 1 {\nu, {\sC T}}$. To this end, we first note that for all $\mu\in \measpace T$
the measurability of 
\begin{align*}
t\mapsto \dualpair {\Cdualmap\mu}{\Phi(t)}{\sC T} = \int_T k(s,t) \intd \mu(s)  
\end{align*}
follows from measurability of $k$ and Fubini-Tonelli's theorem, see e.g.~\cite[Theorem 14.16]{Klenke14}. Consequently, $\Phi$ is weakly measurable, and since $\sC T$ is separable,
$\Phi$ is also strongly measurable. In addition we have 
\begin{align*}
\sup_{t\in T} \inorm{\Phi(t)} = \sup_{s,t\in T} |k(s,t)| < \infty
\end{align*}
since $k$ is continuous and $T\times T$ is compact. Combining these considerations yields $\Phi\in \sLx 1 {\nu, {\sC T}}$. 

Now, \eqref{eq:covfct-by-cov} and the symmetry of $k$ shows $k(s, \mycdot) =  \cov (X) \diracf {s}$ for all $s\in T$. Together with  \eqref{eq:weak-cov-sym} 
and the Pettis property \eqref{eq:bochner-is-pettis} of Bochner integrals
this in turn implies 
\begin{align*}
\dualpair {\diracf {s}}{\cov (X) \Cdualmap\nu}{\sC T}
=
\dualpair {\Cdualmap\nu }{\cov (X) \diracf {s}}{\sC T}
&= 
\dualpair {\Cdualmap\nu }{k(s, \mycdot) }{\sC T} \\
&= 
\int_T k(s,t) \intd \nu(t) \\
&=
\int_T \dualpair {\diracf {s}}{\Phi(t)}{\sC T} \intd \nu(t) \\
&= 
\dualpairB {\diracf {s}}{\int_T \Phi(t)\intd \nu(t) }{\sC T} 
\end{align*}
for all $s\in T$. In other words, the functions $\cov (X) \Cdualmap\nu$ and $\int_T \Phi(t)\intd \nu(t)$ have the same values for all $s\in T$, and hence we have 
\begin{align*}
\cov (X) \Cdualmap\nu = \int_T \Phi(t)\intd \nu(t)\, .
\end{align*}
This in turn gives
\begin{align*}
\dualpair {\Cdualmap\mu}{\cov (X) \Cdualmap\nu}{\sC T} 
= 
\dualpairB {\Cdualmap\mu}{\int_T \Phi(t)\intd \nu(t) }{\sC T} 
&= 
\int_T \dualpair {\Cdualmap\mu}{k(\mycdot, t)}{\sC T} \intd \nu(t) \\
&= 
\int_T \int_T k(s,t) \intd\mu(s) \intd\nu(t)\, ,
\end{align*}
where in the second step we again used \eqref{eq:bochner-is-pettis}.
\end{proof}